\documentclass{amsart}
\usepackage{amsfonts,amsmath,amssymb}
\usepackage{pgfplots}
\usepackage{subfig}
\usepackage{hyperref}
\usepackage{algorithm}
\usepackage{algpseudocode}
\numberwithin{equation}{section}
\DeclareMathOperator{\vdiv}{div}
\newcommand{\dualp}[1]{\left\langle #1 \right\rangle} 
\DeclareMathOperator*{\argmin}{argmin}
\DeclareMathOperator*{\argmax}{argmax}
\newcommand{\dist}{\mathop{\rm max\,dist}}
\newcommand{\truth}{\mathcal{N}}
\newcommand{\outn}{H_b}
\newcommand{\bHs}{\bar{Y}_\mu}
\newcommand{\opOut}{\bar{B}_\mu}
\newcommand{\bo}{B}
\newcommand{\bbo}{\boldsymbol{B}}
\newcommand{\sur}{R_n(\mu)}
\newcommand{\mnew}{\hat{X}_\mu}
\newcommand{\hf}{V}
\newcommand{\mf}{W}
\newcommand{\oldHs}{Y_\mu}
\newcommand{\Ps}{P_{\oldHs,\HH_n}}
\newcommand{\HH}{Y}
\newcommand{\MM}{X}
\newcommand{\MnOne}{\MM_n^1}
\newcommand{\e}{\epsilon}
\newcommand{\lsim}{\raisebox{-1ex}{$~\stackrel{\textstyle<}{\sim}~$}}
\renewcommand{\t}{\tilde}
\newcommand{\N}{\mathbb{N}}
\newcommand{\R}{\mathbb{R}}
\newtheorem{lemma}{Lemma}[section]
\newtheorem{prop}[lemma]{Proposition}
\newtheorem{cor}[lemma]{Corollary}
\newtheorem{theorem}[lemma]{Theorem}
\newtheorem{rem}[lemma]{Remark}

\newtheorem{definition}[lemma]{Definition}

\newtheorem{assumption}[lemma]{Assumption}
\newtheorem{summary}[lemma]{Summary}
\begin{document}
\title{Double Greedy Algorithms: Reduced Basis Methods for Transport Dominated Problems}
\thanks{This work has been supported in
part by the DFG Special Priority Program SPP-1324, by the DFG SFB-Transregio 40, and by the DFG Research Group 1779}
\author{Wolfgang Dahmen}\address{Institut f\"ur Geometrie und Praktische Mathematik, RWTH Aachen, Germany, e-mail:  dahmen@igpm.rwth-aachen.de}
\author{Christian Plesken}\address{Institut f\"ur Geometrie und Praktische Mathematik, RWTH Aachen, Germany, e-mail: plesken@igpm.rwth-aachen.de}
\author{Gerrit Welper}\address{Institut f\"ur Geometrie und Praktische Mathematik, RWTH Aachen, Germany, e-mail: welper@igpm.rwth-aachen.de}
\date{February 20, 2013}
\begin{abstract}
The central objective of this paper is to develop reduced basis methods for parameter dependent
 transport dominated problems
that are rigorously proven to exhibit rate-optimal performance when compared with the Kolmogorov $n$-widths
of the solution sets. The central ingredient is the construction of computationally feasible ``tight'' surrogates
which in turn are based on deriving a suitable well-conditioned variational formulation for the parameter dependent
problem. The theoretical results are illustrated by numerical experiments for convection-diffusion and pure transport
equations. In particular, the latter example sheds some light on the smoothness of the dependence of the solutions on the parameters.
\end{abstract}
\maketitle
\section{Introduction}\label{intro}
Over the past few years {\em model order reduction}
has become an indispensable constituent of large scale design or optimization problems. In particular, the {\em Reduced Basis Method} (RBM) is perhaps by now one of the most
important paradigms   for highly complex  {\em frequent query problems}
  involving parameter dependent PDEs, see e.g. \cite{RHP,SVHDNP, PR}.
Among other things, at least under certain   circumstances,
modeling errors are rigorously controlled and can be upgraded if necessary. 

While the development of RBMs has been a very active area with impressive success stories in
by now a variety of important application fields, it is fair to say that a theoretical underpinning
of what one might call ``near-optimal  performance''  - in a sense to be made precise later -   is still confined to a 
relatively narrow problem class. The central purpose of this paper is therefore to extend the scope
of problems for which RBMs can be developed and rigorously proven to perform in that near optimal sense.
The focus of the present work is on performance in terms of the {\em accuracy} offered by the reduced  model, roughly speaking, centering around the question how to ensure any certified target tolerance
 of the reduced model by a possibly small number of reduced basis functions, of course,
 always insisting on the standard {\em offline-online}  division of
 the overall computational work.
 
\subsection{General Framework}\label{sect1.1}
Suppose that $B_\mu:X\to Y'$, $\mu\in \mathcal{P}$, is a family of (linear) operators
from a Hilbert space $X$ onto the dual $Y'$ of another Hilbert space $Y$, depending on parameters $\mu$
from a compact set $\mathcal{P}\subset \R^p$. Under appropriate conditions on $\{B_\mu\}_{\mu\in\mathcal{P}}$  the solution set
\begin{equation}
\mathcal{M} := \{p(\mu)=B_\mu^{-1}f: \mu\in\mathcal{P}\} \subset X
  \label{eq:solution-set}
\end{equation}
for the family of operator equations 
\begin{equation}
\label{opeq}
B_\mu p(\mu)=f,\quad \mu\in \mathcal{P},
\end{equation} 
is a compact subset of $X$. In the context of frequent query problems, like steering a functional $\ell(p(\mu))$ of the solution towards a 
target value,
RBMs try to exploit the fact that $\mathcal{M}$ may be a very thin subset of $X$. In fact, compactness of $\mathcal{M}$
means that
  the {\em Kolmogorov $n$-widths}
\begin{equation}
\label{Kolm}
d_n(\mathcal{M})_X := \inf_{{\rm dim}\,V\leq n} \dist\,(\mathcal{M},V)_X,
\end{equation}
 tend to zero as $n\to \infty$, where $V$ is taken from the set of {\em all} $n$-dimensional subspaces of $X$ and
\begin{equation*}
    \dist\,(\mathcal{M},X_n)_X := \sup_{p\in \mathcal{M}}\inf_{q\in X_n}\|p-q\|_X. 
\end{equation*}
The objective is then to construct
(problem dependent) subspaces $X_n\subset X$ of possibly small dimension $n$ such that for a given 
{\em target accuracy} ${\rm tol}$, say, 
\begin{equation}
\label{certify}
\dist\,(\mathcal{M},X_n)_X \leq {\rm tol}
\end{equation}
is {\em guaranteed} to hold. 
In particular, this implies that
for any $p\in \mathcal{M}$ and any bounded linear functional $\ell\in X'$, a trivial estimate immediately gives   $|\ell(p)-\ell(P_{X,X_n}p)|\leq \|\ell\| {\rm tol}$ (which could even be improved by duality arguments, see e.g. \cite{SVHDNP}), where  $P_{X,X_n}$   is the $X$-orthogonal projection onto $X_n$. 

Of course, a key question is how to practically construct spaces $X_n$ warranting \eqref{certify} for possibly small $n$.
A common strategy of essentially all RBMs is the following. Given $X_n$, find a   {\em surrogate} $R(\mu,X_n)$, $\mu\in\mathcal{P}$,
such that
\begin{equation}
\label{surr1}
\|p(\mu) -P_{X,X_n}p(\mu)\|_X \leq C_R R(\mu,X_n)
\end{equation}
holds for some constant $C_R$ independent of $\mu$ and $n$.
Here it is crucial that  the evaluation of $R(\mu,X_n)$   {is} sufficiently efficient so that the maximization of $R(\mu,X_n)$ over 
$\mu\in\mathcal{P}$ is computationally  feasible. Then perform the {\em greedy algorithm} {\bf GA} based on this surrogate, as described in Algorithm \ref{alg:greedy}

\begin{algorithm}[htb]
  \caption{greedy algorithm}
  \label{alg:greedy}
  \begin{algorithmic}[1]
    \Function{{\bf GA}}{}
  \State Set $X_0 := \{0\}$, $n=0$,
  \While{$\argmax_{\mu\in\mathcal{P}} R(\mu,X_n) \ge tol$}
    \State
    \begin{equation}
      \begin{aligned}
        \mu_{n+1} & := \argmax_{\mu\in\mathcal{P}} R(\mu,X_n), \\
        p_{n+1} & := p(\mu_{n+1}), \\
        X_{n+1}& := {\rm span}\,\big\{X_n,\{p(\mu_{n+1})\}\big\} = {\rm span}\,\{p_1, \dots, p_{n+1}\}
      \end{aligned}
      \label{greedy1}
    \end{equation}
  \EndWhile
  \EndFunction
  \end{algorithmic}
\end{algorithm}

We have ignored for the moment   the fact that the snapshots $p(\mu_{n })$
can, of course, not be computed exactly but only approximately within some sufficiently large but finite dimensional ``truth space''.

To see whether such a greedy space search produces good reduced models one can compare them with
the ``best possible'' spaces.  Clearly, the $n$-width $d_n(\mathcal{M})_X$ from \eqref{Kolm}
is a {\em lower bound} for the accuracy attainable by any RBM, i.e.,
\begin{equation}
\label{greedyerror}
d_n(\mathcal{M})_X\leq  \sigma_n(\mathcal{M})_X := \sup_{\mu\in\mathcal{P}}\|p(\mu) -P_{X,X_n}p(\mu)\|_X  .
\end{equation}
 Unfortunately,  in general it seems to be impossible to compute the precise
 {\em optimal} subspaces  for which the $n$-width is attained.
Nevertheless, the closer $\sigma_n(\mathcal{M})_X$ is to $d_n(\mathcal{M})_X$
the better the choice of $X_n$.  

To see what can be achieved in this regard, recall from  \cite{BCDDPW,Buffaetal} that
even when $R^*(\mu,X_n):=\|p(\mu)-P_{X,X_n}p(\mu)\|_X$ is the {\em ideal} surrogate,
in a direct comparison  $\sigma_n(\mathcal{M})_X\leq K_n d_n(\mathcal{M})_X$
the constant $K_n$ can be as large as $2^n$.  Nevertheless, the following more favorable
results in terms of {\em convergence rates} hold for surrogates that are {\em tight}, i.e.,  
if in addition to the upper bound \eqref{surr1} it uniformly sandwiches the exact distance.

\begin{definition}
\label{def:tightsurr}
We call the surrogate $R(\mu,X_n)$,
{\em tight} if
 there exist  positive constants $c_R, C_R$ such that 
\begin{equation}
\label{surr2}
c_R R(\mu,X_n)\leq   \|p(\mu) -P_{X,X_n}p(\mu)\|_X \leq C_R R(\mu,X_n), 
\end{equation}
  uniformly in $\mu\in\mathcal{P}$. Moreover, we call
\begin{equation}
\label{condR}
\kappa(R):= \inf\,\{ C_R/c_R : c_R,\, C_R\,\,\mbox{satisfiy}\,\, \eqref{surr2}\,\, \mbox{for all}\,\,\mu\in\mathcal{P},\,  n\in\N\},
\end{equation}
the {\em condition} of the surrogate $R$.
\end{definition}
\begin{rem}
\label{remweakgreedy}
As   already observed in \cite{BCDDPW}  whenever the surrogate is tight, i.e.  \eqref{surr2})  holds,
then the snapshots $p_n=p(\mu_n)$ from \eqref{greedy1}   satisfy the {\em weak greedy condition}
\begin{equation}
\label{gamma}
\|p_n - P_{X,X_n}p_n\|_X \geq \kappa(R)^{-1} 
\dist\,(\mathcal{M},X_n)_X, 
\quad n\in \N,
\end{equation}
where $\kappa(R)$ is given by \eqref{condR}. 
\end{rem}

The following statements are then readily derived  from the results in \cite{BCDDPW,DPW}.
\begin{theorem}
\label{thrates}
Assume that the spaces $X_n$ are obtained through a greedy algorithm {\bf GA}, \eqref{greedy1} based on
tight surrogates. 
Then,
if $d_n(\mathcal{M})_X=O(n^{-\alpha})$, for some
$\alpha >0$ or if $d_n(\mathcal{M})_X=O(e^{-cn^\alpha})$, for some
$c, \alpha >0$, one has 
\begin{align}
\label{rates}
\dist\,(\mathcal{M},X_n)_X & = O(n^{-\alpha}), & \dist\,(\mathcal{M},X_n)_X & = O(e^{-\t cn^{ {\alpha} }}),& n& \to \infty,
\end{align}
respectively, where the constants depend on $\alpha, c$, and $\kappa(R)$ 
with exact specification given in \cite{BCDDPW,DPW}.
Moreover, these bounds remain valid up to the tolerance ${\rm tol}^*$ when all computations are carried
out within this accuracy.
\end{theorem}
We call an RBM {\em rate-optimal} if the generated spaces $X_n$ satisfy ``Kolmogorov optimal'' 
 bounds of the type \eqref{rates}.

There are two important points to be drawn from these results that guide the subsequent developments. 
The first one is: although dispensing with   the (infeasible)  ideal surrogate $R^*(\mu,X_n):=\|p(\mu)-P_{X,X_n}p(\mu)\|_X$, a tight surrogate still ensures that the accuracy provided by the reduced bases is in terms of rates
still essentially as good as that of the ``Kolmogorov-best'' subspaces. The second point is {\em quantitative}.
It is absolutely vital to make sure that the condition $\kappa(R)$  stays as small  as possible. In fact, a look at the dependence of the constants in \eqref{rates} on $\kappa(R)$ (see  \cite{BCDDPW,DPW}) reveals that the closer $\kappa(R)$ is kept to one, the better is
the accuracy of the reduced spaces, in comparison with the best spaces,  already for a {\em small reduced dimension},
which is at the heart of model reduction.

Hence, the central objective of this paper is to develop a rigorous conceptual framework to obtain 
practically feasible tight surrogates  whose 
{\em condition} $\kappa(R) \leq C_R/c_R$ is  as close to one as possible, in particular, for problem
classes for which this is currently not known.

\subsection{Objectives and Layout}\label{sect-layout}
To provide an orientation for subsequent developments the corresponding {\em ideal scenario} and
the corresponding basic mechanisms are briefly recalled in Section \ref{sect1.2}.
It is by and large confined to problems that are uniformly elliptic with respect to the parameters.
The perhaps next best understood case is the reduction of a parabolic problem
to a sequence of elliptic problems \cite{Haasdonk,Grepl1,Grepl2}, where however, the lower bound in \eqref{surr2} - and hence tightness - seems to be missing. This has been recently significantly improved in \cite{PU} using a space-time variational formulation.
Moreover,
important progress has   been made in  \cite{GV,GV2,Rozza} developing RBMs for specific saddle
point - hence indefinite - problems such as the Stokes system. In particular, in the present terminology  {stability and, as a consequence, tight surrogates
are obtained by enriching the velocity spaces by {\em supremizers}. More precisely, there are two
approaches. For standard affine parameter dependence of the involved bilinear forms one can 
determine {\em a priori} an enrichment, depending on the number of terms in the bilinear forms, that ensures
that the infinite dimensional inf-sup-constant is preserved, see in \cite{GV,Rozza}.
Since the number of these supremizers is possibly quite large, as an alternative, it is proposed in \cite{GV2}
to {\em adaptively}
add supremizers until a desired inf-sup-stability is reached. It is observed experimentally that in the
tested examples this adaptive enrichment results in an overall much smaller number of stabilizing functions
although the actual guaranteed termination of such a procedure has apparently not been discussed.
Although termination in the context treated in \cite{GV,GV2,Rozza} is apparent, we shall encounter situations where this
is no longer the case. Nevertheless, relating also the stabilizing enrichments to greedy approximations
allows us to treat this case as well, see Section \ref{sec:term}.
 
Although the present paper addresses a rather different problem class the treatment of saddle point problems 
turns out to be an important point of contact. In fact,
the  stabilizing enrichment of the reduced velocity spaces by adding supremizers can be viewed as a special instance
of the interior loop of what we call here {\em double greedy schemes}, presented first at a workshop in Paris, 2011 \cite{paris}. 
The central objective of this paper is in fact  to develop rate-optimal RBMs - viz.   identify well-conditioned 
 tight surrogates -  for 
a much wider scope of problems, in particular, to those that are at present notoriously not covered by current RBM
theory, namely 
{\em transport dominated} problems.  Two model problems are formulated in Sections \ref{ssect:con-diff},
\ref{ssect:transport}, exhibiting increasing levels of obstructions.
We emphasize though that the general methodology presented below is
not restricted to those problems at all. 
 A key role in this context is played by deriving {\em stable variational formulations} for such  problems that are
 necessarily of {\em Petrov-Galerkin} type. The main features of this approach,
 being valid for a wide range of
 problems including indefinite, unsymmetric  {and} singularly perturbed problems, are shown in Section
 \ref{sect:stabvar}. They can  be summarized as    follows:
 \begin{itemize}
 \item[(i)]
 Tight a-posteriori bounds for the truth spaces as well as reduced spaces warrant certification. In particular, truth and reduced spaces can be upgraded without discarding prior computations, see the robustness results in \cite{BCDDPW}.
 \item[(ii)]
 While remaining feasible in the sense of an online/offline decomposition through a built in stabilization loop, the scheme 
 automatically gives rise to
 stability constants that can, in principle,  be made arbitrarily close to one, see also  \eqref{eq:disc-res} and Section \ref{sect:appl}.
 \item[(iii)]
   Viewing time as an additional ``spatial''  variable, the results can be applied to time dependent problems
 through corresponding space-time discretizations, which is one reason to focus on transport problems, see \cite{DHSW}.
 \end{itemize}
In summary,   the particular variational formulations presented in Section \ref{sect:stabvar} combined
 with certain stabilization techniques optimally  inherits the analytic structure of the underlying infinite dimensional
 problem to the reduced model. 
 
 Section \ref{sec:DG} is then devoted to the algorithmic development and analysis of a {\em double greedy} scheme
 giving rise to rate-optimal RBMs.
 
 The theoretical findings are then applied in Section \ref{sect:appl} to the two model problems 
 concerning convection-diffusion and pure transport equations. First numerical experiments quantify the results
 and highlight several particular obstructions.
 
 In Section \ref{sectsaddle} we apply the (slightly modified) scheme to other types of saddle point problems not
 necessarily stemming from the generation of well-conditioned variational formulations. As a simple consequence
 we obtain rate-optimality also for the problems considered in \cite{GV,GV2,Rozza}.

To simplify the exposition we write $a\lsim b\,$ to express that $a$ is bounded by some constant multiple of $b$
independent on any parameters $a,b$ may depend on. Likewise $a\sim b$ means $a\lsim b$ and $b\lsim a$.

\section{Conceptual Preview}\label{sectpreview}
\subsection{Feasibility}\label{sect2.1}

In all subsequent developments  we will be dealing
exclusively with {\em affine} parameter dependence, see e.g. \cite{RHP}. Under this assumption we insist on the
usual division of the computational work into an {\em offline} and {\em online} mode.
Solving a problem in the full space $X$, which is typically computationally very intense, 
 happens only in offline mode where it is understood that
actual computations   take place
in some sufficiently large but finite dimensional subspace $X_\mathcal{N}$ of $X$ which is commonly 
referred to as  the ``truth space''.  Typically $X_\mathcal{N}$ is chosen so as to guarantee 
\begin{equation}
\label{truth}
\sup_{p\in \mathcal{M}}\inf_{v\in X_\mathcal{N}}\|p-v\|_X\leq {\rm tol}^*,
\end{equation}
for some tolerance ${\rm tol}^*$ that   is sufficiently small for the application at hand.
The subscript $\mathcal{N}$ is sometimes supressed when there is no risk of confusion.
The greedy search for the reduced basis functions falls therefore into the offline mode.
This requires evaluating the surrogate for a sufficiently large {\em training set} of parameters which for simplicity
we also denote by $\mathcal{P}$.
In what follows, we call the surrogate {\em feasible} if each evaluation of the surrogate 
requires solving only a problem in the small current reduced space $X_n$.  We sometimes say then that the offline mode
is (computationally offline) feasible. 

Likewise, the {\em online evaluation} is called  {\em feasible} if each reduced basis approximation of some $p(\mu)$ requires solving only
a ``small'' problem of dimension $n$ in the reduced space $X_n$. In this mode solving a ``large'' problem in $X_\mathcal{N}$
is prohibited.

Note that a feasible surrogate is not allowed to   explicitly contain the true solution
$p(\mu)$ (in the truth space). This is why one is essentially forced to resort to {\em residuals} to estimate the true error,
which in turn requires a tight {\em error-residual} relation.

\subsection{The Ideal Setting}
\label{sect1.2}
 The following ``ideal setting'' reveals the basic mechanisms leading to {\em residual based} tight surrogates.

To this end, let $b_\mu(\cdot,\cdot):X\times X\to \R$ be a symmetric uniformly $X$-elliptic bilinear form and $\ell\in X'$,
i.e.
\begin{equation}
\label{1.2.1}
c_a \|q\|_X^2\leq b_\mu(q,q), \quad b_\mu(p,q)\leq C_a\|p\|_X \|q\|_X,\,\,  p,q \in X,\, \mu\in \mathcal{P},
\end{equation}
holds uniformly in $\mu\in \mathcal{P}$. 
For compact $\mathcal{P}$ one obtains a compact
solution set $\mathcal{M}\subset X$ for: given $\ell\in X'$, find $p(\mu)\in X$ such that
\begin{equation*}
b_\mu(p(\mu),q)=\langle \ell,q\rangle, \quad q\in X.
\end{equation*}
There are two key properties that ensure rate-optimality in this setting:\\

 \noindent
{\bf (MP)  Mapping property:} 
the operator
 $B_\mu$, defined by $\langle B_\mu p,q\rangle = b_\mu(p,q)$, $p,q \in X$,
is   for each $\mu\in\mathcal{P}$ an isomorphism from $X$ onto $X'$, i.e.
\begin{equation}
\label{iso}
\|p\|_X \sim \|B_\mu p\|_{X'},\quad \mbox{uniformly in}\,\, \mu\in \mathcal{P}.
\end{equation}
In other words, errors measured in the ``energy norm'' $\|\cdot\|_X$ are
equivalent to residuals in the dual norm $\|\cdot\|_{X'}$.\\
{\bf (BAP) Best Approximation Property:} The Galerkin projection to the current reduced space, which can be done in online mode,  produces, up to constants, a best approximation with respect to the $X$-norm.

In fact, denoting by  $\Pi_{\mu,X_n}$  the Galerkin-projector onto $X_n$ defined by
\begin{equation*}
b_\mu(p(\mu),q)=b_\mu(\Pi_{\mu,X_n}p(\mu),q),\quad q\in X_n,
\end{equation*}
combining  Cea's Lemma with {\bf MP} provides 
 for $p_n(\mu):= \Pi_{\mu,X_n}p(\mu)$
\begin{eqnarray}
\label{res1}
\|p(\mu)- P_{X,X_n}p(\mu)\|_X &\sim & 
\sup_{q\in X}\frac{\langle \ell, q\rangle - 
b_\mu(p_n(\mu),q)}{\|q\|_X} 
:= R (\mu, X_n).
\end{eqnarray}
Thus, {\bf MP} and {\bf BAP} imply that the residual based surrogate, defined by \eqref{res1}, is tight, while the computation
of $p(\mu)$ is completely avoided but traded against the cheap computation of the Galerkin projection in $X_n$.
However, the condition $\kappa(R) $ of the surrogate (see \eqref{condR}) depends on the 
{\em condition number} $\kappa_{X,X'}(B_\mu)\leq C_a/c_a$ (see \eqref{1.2.1}) of the operator $B_\mu$, which should 
therefore be of moderate size.

Finally, feasibility of the surrogate in \eqref{res1} is well known (see e.g. \cite{RHP}) to be ensured when the parameter dependence
of $b_\mu(\cdot,\cdot)$ is {\em affine}, see \eqref{affine} below.

\subsection{Two Model Problems}\label{sect:modelpr}

As soon as one leaves the elliptic setting {\bf MP}, {\bf BAP},   tightness  of residual based
surrogates, are no longer for free.
In particular, so far {\em well-conditioned}  tight surrogates do {\em not} seem to be available yet for  {many unsymmetric PDEs like convection dominated or pure transport problems.} We shall discuss two model problems that bring out several principal obstructions.
 The first example concerns convection-diffusion equations for which, in principle,
classical variational formulations are available. The second example concerns pure transport for which
a ``natural''  variational formulation is less obvious and for which the parameter dependence of the solutions turns out to be
 less regular. Perhaps more importantly, the two examples represent two different scenarios regarding the {\em spaces} associated
 with the bilinear form $b_\mu(\cdot,\cdot)$, an issue that has apparently not been addressed in the RBM context.

\subsubsection{Convection-Diffusion Equations}\label{ssect:con-diff}
As a first example we consider  the linear {\em convection-diffusion} equation
\begin{equation}
\label{eq:cd}
-{\rm div}(\e  \nabla p(x)) + b(\mu)\cdot \nabla p(x) +c p(x)= f(x),\quad \mbox{in}\,\, \Omega, \quad p=0\,\, \mbox{on}\,\, \partial\Omega,
\end{equation}
where for simplicity   we assume for now that only the convection $b(\mu)$ depends on a parameter $\mu$ while $\e$ could be {\em arbitrarily small}.  We could as well include the viscosity and the reaction term varying in suitable regimes.
Its classical weak formulation is
\begin{equation}
  \begin{aligned}
    b_\mu(p,q):=\epsilon (\nabla p, \nabla q) + (b(\mu) \cdot \nabla p, q) + (c p,q) & = \dualp{f,q}, & 
    q \in \MM = H^1_0(\Omega).
  \end{aligned}
  \label{eq:condiff}
\end{equation}
It is well known that
$  b(\mu)  \in W^{1,\infty}(\Omega)^d$,  $c  \in L_\infty(\Omega),\, \mu\in \mathcal{P}$,
such that 
\begin{equation}
  \label{eq:assumption-2}
  -\frac{1}{2} \vdiv b(\mu) + c \ge 0,
\end{equation}
implies {\em well-posedness} of \eqref{eq:condiff} in the sense that the induced operator $B_\mu: H^1_0(\Omega)\to (H^1_0(\Omega))' $  is an isomorphism, i.e., there exists for each $\mu\in \mathcal{P}$
 a unique solution $p(\mu)$ 
to \eqref{eq:condiff} in $H^1_0(\Omega)$.
However, although \eqref{1.2.1} is still valid, the condition number $\kappa_{H^1_0(\Omega),H^1_0(\Omega)}(B_\mu)$
behaves like the P\'{e}clet number $|b(\mu)|_\infty/\e$ and hence is inacceptably large for strongly
dominating convection. As a consequence, in this case the condition   $\kappa(R)$ of the corresponding surrogate  \eqref{res1}
based on the  $H^{-1}(\Omega)$-residual grows with the P\'{e}clet number. Hence, although such a surrogate is theoretically
tight, as long as $\epsilon\geq \epsilon_0$ where $\epsilon_0 >0$ is fixed, the condition $\kappa(R)$ (see \eqref{condR}) is so large,
that, due to the constants in \eqref{rates}, one can  expect essentially no control of the quality of the reduced spaces  for very small $\epsilon_0$ and moderate $n$.
   
Therefore, we are mainly interested here in a robust treatment of {\em arbitrarily} large P\'{e}clet numbers $|b(\mu)|_\infty/\e$
which to our knowledge is currently not well covered by RBM methodology. 

Unfortunately, an easy cure based on  the standard (mesh-dependent) stabilization methods  such as SUPG (see e.g.  \cite{RST} for a survey) does not give rise to an error-residual relation that stays independent of the P\'{e}clet number  
$|b|_\infty/\epsilon$ either.

Instead we pursue here a different line based on stabilizing the problem on the {\em infinite dimensional} level which,
in particular, involves {\em unsymmetric} variational formulations, i.e., $b_\mu(\cdot,\cdot)$ is viewed
as a bilinear form on a pair of (possibly) {\em different} and {\em parameter dependent} Hilbert spaces
 $\MM_\mu,  \HH_\mu$,   $\mu\in \mathcal{P}$.
 
\subsubsection{Linear Transport Equations}\label{ssect:transport}
 
In some sense the situation is even aggrivated when the diffusion vanishes completely as 
in pure {\em parametric  transport equations} forming the core ingredient of Boltzmann equations and related kinetic models 
as well as kinetic formulations of conservation laws.
Already the simplest version  {of a (stationary)  linear transport equation
\begin{equation}
    \mu \cdot \nabla p + cp = f, \quad \text{in } \Omega = [0,1]^d,\quad
    p  = p_b, \quad \text{on } \Gamma_-(\mu),
  \label{eq:transport}
\end{equation}
will be seen to represent the
 ``worst scenario'' from the RBM perspective,
where, denoting by $n(x)$ the outward normal at the point $x$,
\begin{equation*}
  \Gamma_-(\mu) := \{x \in \partial \Omega: \, n(x) \cdot \mu < 0 \},
\end{equation*}
is the inflow boundary for the given convection vector $\mu$.} An example of a
parameter domain would be the sphere $S^{d-1}$ appearing in radiative transfer models, see \cite{KMRW, Sch}.  
It will be seen that the two examples differ in a subtle but essential way, in particular,
regarding smoothness of  the dependence of the solutions
on the parameter. 

 {A possible variational formulation of \eqref{eq:transport} can be found in \cite{EG04}. In order to eventually apply the $n$-width benchmark, it is preferable to measure all parameter dependent solutions in a single reference norm. Therefore we employ here a slightly different variational formulation from \cite{DHSW}:}
 multiplying \eqref{eq:transport} by a test function and integrating by parts, yields 
\begin{equation*}
  ( p, -\mu \cdot \nabla q + c q ) + \int_{\partial \Omega \setminus \Gamma_-} n \cdot \mu p q = \langle f, q \rangle - \int_{\Gamma_-} n \cdot \mu p q .
\end{equation*}
If we now take test functions $q$ that vanish on $\partial \Omega \setminus \Gamma_-$ the boundary integral on the left hand side is zero. Furthermore, we may replace the function $p$ in the boundary integral on the right hand side by the boundary condition $p_b$ so that we obtain
\begin{equation}
\label{btransport}
b_\mu(p,q):= 
\langle p, -\mu \cdot \nabla q + c q \rangle = \langle f, q \rangle - \int_{\Gamma_-} n \cdot \mu p_b q.
\end{equation}
For this variational formulation it is natural to define the function spaces
\begin{equation} 
\label{transspaces}
  \HH_\mu := \operatorname{clos}_{\|\cdot\|_{\HH_\mu}} \left\{ q \in C^\infty(\Omega) : \, q|_{\partial \Omega \setminus \Gamma_-} = 0 \right\}, \quad
  \MM_\mu  := L_2(\Omega)
\end{equation} 
endowed with the norms
\begin{equation} 
\label{transnorms}
  \|q\|_{\HH_\mu}  := \|\bo_\mu^* q \|_{L_2}, \quad
  \|p\|_{\MM_\mu}  := 
   \|p\|_{L_2}.
\end{equation} 
It is shown in \cite{DHSW} that the operator $B_\mu$ induced by $b_\mu(\cdot,\cdot)$, is an isomorphism
 $B_\mu: X_\mu \to Y_\mu'$ so that   \eqref{transnorms}  indeed defines a norm.
 \begin{rem}
 \label{remYdiffer}
 Notice that the spaces $\HH_\mu$ differ even as sets for different $\mu$. Moreover, in contrast to the previous example we
must have $\MM_\mu\neq \HH_\mu$ here.
 \end{rem}

\section{Robust Error-Residual Mappings for Unsymmetric Problems}\label{sect:stabvar}

\subsection{The Basic Principle for {\bf MP}}\label{sec:recipe}
 {In the following, we consider general bilinear forms $b_\mu(\cdot,\cdot) : \MM_\mu \times \HH_\mu\to \R$ for possibly parameter dependent Hilbert spaces $\MM_\mu$ and $\HH_\mu$ giving 
rise to} what one may call an
infinite dimensional  {\em Petrov-Galerkin} formulation where the  {\emph{trial space}} $\MM_\mu$ generally differs from the  {\emph{test space}}
$\HH_\mu$. Thus, the operator 
$B_\mu$ given by
$
\langle \bo_\mu q,v\rangle = b_\mu(q,v),\quad q\in \MM_\mu,\,\ v\in \HH_\mu,
$
is now viewed as a mapping from $\MM_\mu$ to $\HH_\mu'$. In accordance with the preceding examples,
we shall assume that this operator is
actually an {\em  isomorphism},  i.e.,  the operator equation
\begin{equation}
  \bo_\mu  p(\mu) = f,
  \label{eq:op-eq}
\end{equation}
has for any $f\in \HH_\mu'$ a unique solution in $\MM_\mu$.  However, $B_\mu$ may possibly have an unacceptably  large $\mu$-dependent condition which can be quantified with the aid of 
 Babuska's Theorem: 
 if there exist constants $0 < \beta(\mu), C_b(\mu) < \infty$ such that
\begin{equation}
\label{babuska}
\inf_{q\in \MM_\mu}\sup_{v\in \HH_\mu}\frac{b_\mu(q,v)}{\|q\|_{\MM_\mu}\|v\|_{\HH_\mu}}\geq \beta(\mu), \quad
\sup_{q\in \MM_\mu} \sup_{v\in \HH_\mu}\frac{|b_\mu(q,v)|}{\|q\|_{\MM_\mu}\|v\|_{\HH_\mu}} \leq C_b(\mu)
\end{equation}
and for every $v\in \HH_\mu$ there exists a $q\in \MM_\mu$ such that $b_\mu(q,v)\neq 0$, 
then one has   $\kappa_{\MM_\mu,\HH_\mu}(B_\mu)\leq C_b(\mu)/\beta(\mu)$.

As in the case of convection dominated convection-diffusion equations $\kappa_{\MM_\mu,\HH_\mu}(B_\mu)$
could be very large, severely degrading a greedy selection of snapshots in a RBM. The goal of this subsection is
to describe how to ``stabilize'' the problem on the {\em infinite dimensional} level which could be viewed as
preconditioning \eqref{eq:op-eq}. 
The underlying basic principle  has been used before in several works for different purposes, see e.g. \cite{DJ1,DJ2,MMRS,W}. Here our main orientation is taken from 
\cite{cdw,DHSW,W}. We briefly rehash the essential facts in order to bring in an additional {\em new} element, namely the interrelation 
of  {\em  Petrov-Galerin} schemes and {\em associated saddle point problems}, which plays an essential role for eventually constructing 
well-conditioned tight surrogates.

We begin with collecting a few useful preliminaries.
It will be usefull to identify for a given $q\in \MM_\mu$ the {\em supremizer} $v_q$ for which $\sup_{v\in \HH_\mu}b_\mu(q,v)/\|v\|_{\HH_\mu}$ is attained, see e.g. \cite{PR, GV}.

\begin{rem}
\label{remcond2}
For every $q\in \MM_\mu$ the optimal test function is given by
\begin{equation}
\label{supremizer}
v_q:=\argmax_{v\in \HH_\mu}\frac{b_\mu(q, v)}{\|v\|_{\oldHs} }= R_{\oldHs}^{-1}\bo_\mu  q,
\end{equation}
where $R_{\HH_\mu}: \HH_\mu\to \HH_\mu'$ is the Riesz-map defined by
\begin{equation}
  \langle R_{\HH_\mu} v,w\rangle  = (v,w)_{\HH_\mu}, \quad v,w  \in \HH_\mu, \quad\quad\quad \|\cdot\|_{\HH_\mu}^2  = (\cdot,\cdot)_{\HH_\mu},
  \label{eq:riesz-Hs}
\end{equation}
Hence, in particular, one has
\[
\inf_{q\in \MM_\mu}\frac{\|R_{\HH_\mu}^{-1}B_\mu q\|_{\HH_\mu}}{\|q\|_{\MM_\mu}} = 
\inf_{q\in \MM_\mu}\sup_{v\in \HH_\mu}\frac{b_\mu(q,v)}{\|q\|_{\MM_\mu}\|v\|_{\HH_\mu}}.
\]
\end{rem}

For convenience we recall the simple argument. 
Written in variational form, the supremizer is defined by $(v_q, w)_{\HH_\mu} = b_\mu(q, w)$ for all $w \in \HH_\mu$, which yields
\begin{equation*}
  \sup_{v \in \HH_\mu} \frac{b_\mu(q, v)}{\|v\|_{\HH_\mu}} = 
  \sup_{v \in \HH_\mu} \frac{(v_q, v)_{\HH_\mu}}{\|v\|_{\HH_\mu}} = \|v_q\|_{\HH_\mu},
\end{equation*}
which readily confirms the claim.\\

Although for most of the following considerations  the   dependence of the involved bilinear forms on the parameter $\mu\in\mathcal{P}$
is irrelevant it will be convenient for later purposes to retain the parameter dependence in the notation.\\

\paragraph*{\bf Renormation:}
The possible ill-conditioning reflected by a very  large $\kappa_{\MM,\HH}(B)\leq C_b/\beta$ in \eqref{babuska} can be remedied by properly modifying 
one of the two norms $\|\cdot\|_{\HH_\mu}$ or $\| \cdot\|_{\MM_\mu}$ while keeping the other one fixed. 
Specifically, we wish to choose an equivalent but possibly different norm $\|\cdot\|_{\hat\MM_\mu}$ for $\MM$ so that   ideally   $C_b(\mu)=\beta(\mu)=1$,
which then means that
\begin{equation}
\label{err-res}
\| p(\mu) - q\|_{\hat\MM_\mu} = \| f -  B_\mu q\|_{\HH_\mu'},\quad q \in \MM_\mu,\,\, \mu\in \mathcal{P}.
\end{equation}
In this event the {\em residual} of a best approximation would be an {\em ideal}  surrogate even sparing one the computation of stability constants for
the error certification.

Our starting point is exactly this latter {\em ideal error-residual relation}. Specifically,
given $\|\cdot\|_{\HH_\mu}$, we endow now $\MM_\mu$ with a new norm $\|\cdot\|_{\hat\MM_\mu}$, defined by
\begin{equation}
\label{Mu2}
  \|p\|_{\hat\MM_\mu}:= \sup_{v\in \HH_\mu}\frac{b_\mu(p,v)}{\|v\|_{\HH_\mu}}=  \|\bo_\mu  p\|_{\oldHs'}= \|R_{\oldHs}^{-1}\bo_\mu  p\|_{\oldHs}   ,\quad p\in \MM_\mu,\, \mu\in \mathcal{P}.
\end{equation}
 {
Note that this is indeed a well-defined norm because $R_{\oldHs}^{-1} \bo_\mu: \MM_\mu \to \HH_\mu$ is an isomorphism, hence injective, and that because of 
$
  \|R_{\oldHs}^{-1}\bo_\mu q\|_{\HH_\mu}^2 = \langle \bo_\mu^* R_{\oldHs}^{-1}\bo_\mu q, q \rangle,
$  
the corresponding Riesz map $R_{\hat\MM_\mu }: \MM_\mu \to \MM_\mu'$ is given by $R_{\hat\MM_\mu } := \bo_\mu^* R_{\oldHs}^{-1}\bo_\mu$.
In addition, this shows that the $\hat\MM_\mu$-norm is equivalent to the original norm, i.e., there are $c_M, C_M >0$ such that
\begin{equation}
  c_M \|q\|_{\MM_\mu} \le \|q\|_{\mnew} \le C_M \|q\|_{\MM_\mu}, \quad q\in \MM_\mu,\,\, \mu\in \mathcal{P}.
  \label{eq:M-equiv}
\end{equation}
}
Note also that \eqref{transnorms} is a special case of \eqref{Mu2}, where
\begin{equation}
\label{specialRiesz}
R_{\HH_\mu} = B_\mu B_\mu^*.
\end{equation}
and thus $\|\cdot\|_{\hat\MM_\mu} = \|(B_\mu B_\mu^*)^{-1} B_\mu \cdot\|_{\HH_\mu} = \|\cdot\|_{L_2}$.

\begin{rem}
\label{remcond3}
For the $\|\cdot\|_{\hat\MM_\mu}$ norm one has optimal continuity and stability constants
$C_b(\mu)= \beta(\mu)=1$,  $\mu\in \mathcal{P}$, i.e.
\begin{equation}
\label{infsup2}
 \sup_{q\in \mnew}\sup_{v\in \HH_\mu}\frac{b_\mu(q, v)}{\|v\|_{\oldHs}\|q\|_{\mnew}}=
 \inf_{q\in \mnew}\sup_{v\in H_\mu}\frac{b_\mu(q, v)}{\|v\|_{\oldHs}\|q\|_{\mnew}} = 1.
\end{equation}
Hence, $\kappa_{\hat\MM_\mu,\HH_\mu'}(\bo_\mu)=1$, i.e., $B_\mu$ is an isometry for these norms,  which is  the desired robust - in fact optimal -  error-residual relation
\eqref{err-res}
 {\bf MP}.  
\end{rem}
\begin{proof}  The first relation follows from
\begin{eqnarray*}
|b_\mu(q, v)|& = & |\langle R_{\oldHs}^{-1}\bo_\mu q, R_{\oldHs}v\rangle|\leq  
\|R_{\oldHs}^{-1}\bo_\mu q\|_{\oldHs} \|R_{\oldHs}v\|_{\oldHs'}
= \|q\|_{\mnew} \|v\|_{\oldHs}.
\end{eqnarray*}
On the other hand,
note that
for any $q\in \MM_\mu$ its supremizer $v_q:= R_{\HH_\mu}^{-1}B_\mu q\in \HH_\mu$ gives
 by \eqref{Mu2}, \eqref{eq:riesz-Hs}, $b_\mu(q,v_q)= \langle \bo_\mu q,R_{\HH_\mu}^{-1}\bo_\mu q\rangle =
\|q\|_{\hat\MM_\mu}^2$ and $\|v_q\|_{\HH_\mu}= \|q\|_{\hat\MM_\mu}$ so that 
\begin{equation}
\label{infsup1}
\inf_{q\in \MM_\mu}\sup_{v\in \HH_\mu}\frac{b_\mu(q,v)}{\|q\|_{\hat\MM_\mu}\|v\|_{\HH_\mu}}\geq 
\inf_{q\in \MM_\mu}\frac{\langle \bo_\mu q,R_{\HH_\mu}^{-1}\bo_\mu q\rangle}{\|q\|_{\hat\MM_\mu}^2}=1,
\end{equation}
which completes the proof.
 \end{proof}

\subsection{Petrov-Galerkin and Saddle Point Problems}\label{sect:petgal-saddle}
 The validity of {\bf BAP} is no longer automatic for unsymmetric formulations.
In principle, it can be approached through contriving suitable  {\em Petrov Galerkin} discretizations.
A central issue in this section is to relate such Petrov-Galerkin schemes to {\em equivalent  saddle-point problems}. In particular, this avoids the explicit computation of the respective test spaces which could be parameter dependent.
 
To this end let $W\subset \MM_\mu$ be a ``generic'' trial space which will play several different roles.
It may stand for the full infinite dimensional space,  or for the truth space, or eventually for the reduced space. 
Notice first that the {\em best approximation} $p_W(\mu) \in W$ for $p(\mu)=\bo_{\mu}^{-1}f$ is 
\begin{equation}
\label{err-res2}
p_W(\mu) :=\argmin_{q\in W}\|p(\mu)-q\|_{\hat\MM_\mu} = \argmin_{q\in W}\|f- \bo_\mu q\|_{\HH_\mu'} ,
\end{equation}
which is therefore given by the {\em normal equation}: find $p_W(\mu) \in \mf$ such that
\begin{equation}
  \begin{aligned}
    (f - \bo_\mu  p_W(\mu), \bo_\mu  q)_{\oldHs'} & = 0, &  q& \in \mf.
  \end{aligned}
  \label{eq:lsq}
\end{equation}
What keeps us from using this as the basis for a variational discretization,
 is the fact that the $\oldHs'$-scalar product is usually hard to evaluate numerically. Noting that  $R_{\oldHs'} = R_{\oldHs}^{-1}$   the last equation is equivalent to
\begin{equation}
  \begin{aligned}
    \dualp{R_{\oldHs}^{-1} \left( f - \bo_\mu  p_W(\mu) \right), \bo_\mu  q} & = 0, &  q & \in \mf.
  \end{aligned}
  \label{res4-n}
\end{equation}
Introducing the auxiliary variable $u(\mu) := R_{\oldHs}^{-1} \left( f-\bo_\mu  p_W(\mu) \right)$, or rather
\begin{equation}
  \begin{aligned}
    \dualp{R_{\oldHs} u(\mu), v} & = \dualp{f - \bo_\mu  p_W(\mu), v}, & v & \in \oldHs,
  \end{aligned}
\label{auxr-n}
\end{equation}
in weak form, the relation \eqref{res4-n} and hence \eqref{err-res2} can be equivalently written as
\begin{equation}
\label{system-n}
\begin{array}{lccl}
  \langle R_{\oldHs} u(\mu),v\rangle  + b_\mu(p_W(\mu),v) 
   & = & \langle f,v\rangle, & v\in \oldHs,\\
b_\mu(q, u(\mu)) 
&=& 0,& q \in \mf,
\end{array}
\end{equation}
which now just involves standard $L_2$-inner products.
Of course, in particular for $W=\MM_\mu$
\begin{equation}
\label{system-full}
\begin{array}{lccl}
  \langle R_{\oldHs} u(\mu),v\rangle  + b_\mu(p(\mu),v) 
   & = & \langle f,v\rangle, & v\in \oldHs,\\
b_\mu(q, u(\mu)) 
&=& 0,& q \in \MM_\mu,
\end{array}
\end{equation}
is equivalent to the original problem \eqref{eq:op-eq}, which now takes the form of
a {\em saddle point problem}. Bijectivity of $B_\mu$ readily shows that
\begin{equation}
\label{umu}
u(\mu)=0,\quad \mu\in\mathcal{P}.
\end{equation}
Hence, the solution manifold  of the saddle point
problem \eqref{system-full} in   $\bigcup_{\mu\in\mathcal{P}}\MM_\mu\times \HH_\mu$ can be identitied according to
\begin{equation}
\label{sol-M}
\mathcal{M} = \mathcal{M}_{\MM}\times \{0\},\quad \mathcal{M}_{\MM}:= \{p(\mu): p(\mu) \, \mbox{solves \eqref{eq:op-eq}}\},
\end{equation}
as it should, with   the one  for the original problem \eqref{eq:solution-set}.

\begin{rem}
\label{rem-sol-M}
Even when the spaces $\HH_\mu$ differ as sets when $\mu$ varies, as e.g. in \eqref{transnorms} for
the transport equation, the solution manifold is
still compact as long as the norms $\|\cdot\|_{\MM_\mu}$ are all equivalent to a reference norm. Hence,
the greedy errors are guaranteed to tend to zero and the $n$-widths benchmark is applicable. The issue of parameter dependence of the involved spaces will be taken up in Section \ref{sect:pardep} again.
\end{rem}

Now given a finite dimensional subspace $W$, we cannot treat \eqref{system-n} yet, since we cannot
test by all $v\in \HH_\mu$. The following interpretation of this idealized situation is  {immediate from the normal equation \eqref{res4-n}.}

\begin{rem}
\label{remoptpg}
The problem  \eqref{system-n} is equivalent to the Petrov-Galerkin scheme: find $p_W(\mu)$ such that
\begin{equation}
\label{petgalopt}
b_\mu(p_W(\mu),v)=\langle f,v\rangle,\quad v\in Y_W ,
\end{equation}
where
\begin{equation}
\label{opttestspace}
Y_W:= R_{\HH_\mu}^{-1}\bo_\mu W, 
\end{equation} 
is the {\em optimal test space} associated with $W$ and $p_W(\mu)$ is the best $\MM_\mu$-approximation to 
$p(\mu)$ in $\MM_\mu$.
\end{rem}

Since \eqref{petgalopt} is practically infeasible  a natural strategy is to replace $\HH_\mu$ by a sufficiently large 
finite dimensional subspace
$V\subset \HH_\mu$ that inherits ``sufficient'' stability. 
The following observation, which plays a crucial role in what follows, explains the interrelation between a practically
feasible version of \eqref{petgalopt} and a fully finite dimensional version of \eqref{system-n}.   
\begin{prop}
\label{propPG}
The solution component $p_{W,V}(\mu)$ of the saddle point problem 
\begin{equation}
\label{system-disc}
\begin{array}{lccl}
  \langle R_{\oldHs} u_{V,W}(\mu),v\rangle  + b_\mu( p_{W,V}(\mu),v) 
  & = & \langle f,v\rangle, & v \in \hf,\\
b_\mu(q,u_{V,W}(\mu)) 
&=& 0,& q\in \mf.
\end{array}
\end{equation}
solves the Petrov-Galerkin problem \eqref{petgalopt} with
the optimal test space $\HH_W$ replaced by
$\tilde Y_W= P_{\HH_\mu,V}(R_{\HH_\mu}^{-1}B_\mu(W))$ where $P_{\HH_\mu,V}$ denotes the $\HH_\mu$-orthogonal projection.
\end{prop}
\begin{proof}
For any $q\in W$, consider $v_q:= P_{\HH_\mu,V}(R_{\HH_\mu}^{-1}B_\mu q)\in V$ and note that, by
the first equation \eqref{system-disc},
\begin{eqnarray*}
b_\mu(p_{W,V}(\mu),v_q)&=& \langle B_\mu p_{W,V}(\mu), v_q\rangle = \langle f, v_q\rangle - \langle R_{\HH_\mu}u_{V,W},v_q\rangle.
\end{eqnarray*}
Since
\begin{eqnarray*}
\langle R_{\HH_\mu}u_{V,W},v_q\rangle &= &(u_{V,W}(\mu),v_q)_{\HH_\mu}=(u_{V,W}(\mu),R_{\HH_\mu}^{-1}B_\mu q)_{\HH_\mu}
=  b_\mu(q,u_{V,W}(\mu))=0,
\end{eqnarray*}
where we have used the second equation in \eqref{system-disc}. \end{proof}

Clearly, the larger $V$ the closer $\tilde Y_W$ is to $\HH_W$ so that the choice of $V$ can be viewed as
a {\em stabilization}. 
 To quantify this observation, we call \emph{$V$ is $\delta$-proximal for $W$} if 
\begin{equation}
\label{delta-prox}
\|(I- P_{\HH_\mu,V})R_{\HH_\mu}^{-1}B_\mu q\|_{\HH_\mu} \leq \delta \|R_{\HH_\mu}^{-1}B_\mu q\|_{\HH_\mu},
\quad q\in W, 
\end{equation}
holds for some fixed $0\leq \delta <1$, see \cite{DHSW,W}.
\begin{prop}
\label{prop:stab2}
 Assume that for given $W\times V \subset \MM_\mu\times \HH_\mu$ the test space $V$ is $\delta$-proximal for $W$, i.e. \eqref{delta-prox} is satisfied. Then, one has
\begin{equation}
\label{BAP1}
\|p(\mu)- p_{W,V}(\mu)\|_{\hat\MM_\mu}\leq \frac{1}{1-\delta}\inf_{q\in W}\|p(\mu)-q\|_{\hat\MM_\mu}.
\end{equation}
and 
\begin{equation}
\label{BAP2}
\|p(\mu)- p_{W,V}(\mu)\|_{\hat\MM_\mu}+ \|u(\mu)- u_V(\mu)\|_{\HH_\mu}\leq \frac{2 }{1-\delta} 
\inf_{q\in W}\|p(\mu)-q\|_{\hat\MM_\mu}.
\end{equation}
Moreover, one has
\begin{equation}
\label{inf-sup-0}
\inf_{q\in \mf} \sup_{v\in \hf}\frac{b_\mu(q, v)}{\|v\|_{\HH_\mu}\|q\|_{\hat\MM_\mu}} \geq  {\sqrt{1-\delta^2}} .
\end{equation}
\end{prop}
\begin{proof} Let $p_W(\mu)$ denote the best $\hat\MM_\mu$-approximation to the exact solution $p(\mu)$ of 
\eqref{system-full}. Then, for any $q\in W$ one has, on account of Remark \ref{remcond3},
\begin{eqnarray*}
(p_W(\mu)-p_{W,V}(\mu),q)_{\hat\MM_\mu}&=& (p(\mu)-p_{W,V}(\mu),q)_{\hat\MM_\mu}
= (B_\mu(p(\mu)-p_{W,V}(\mu)),B_\mu q)_{\HH_\mu'}\\
&=& \langle B_\mu(p(\mu)-p_{W,V}(\mu)), R_{\HH_\mu}^{-1}(B_\mu q)\rangle
 =  b_\mu(p(\mu)-p_{W,V}(\mu), R_{\HH_\mu}^{-1}(B_\mu q))\\
&=& b_\mu(p(\mu)-p_{W,V}(\mu), (I- P_{\HH_\mu,V}) R_{\HH_\mu}^{-1}(B_\mu q)),
\end{eqnarray*}
where we have used Petrov-Galerkin orthogonality, asserted by Proposition \ref{propPG}, in the last step.
By duality, Remark \ref{remcond3}, \eqref{err-res}, respectively \eqref{Mu2}, and \eqref{delta-prox},
we conclude that
\begin{eqnarray*}
\|p_W(\mu)-p_{W,V}(\mu)\|_{\hat\MM_\mu}&=& \sup_{q\in W}\frac{b_\mu(p(\mu)-p_{W,V}(\mu), 
(I- P_{\HH_\mu,V}) R_{\HH_\mu}^{-1}(B_\mu q))}{\|q\|_{\hat\MM_\mu}}\\
&\leq & \frac{ \|p(\mu)-p_{W,V}(\mu)\|_{\hat\MM_\mu}\delta\|R_{\HH_\mu}^{-1}(B_\mu q)\|_{\HH_\mu}}{\|q\|_{\hat\MM_\mu}}
= \delta \|p(\mu)-p_{W,V}(\mu)\|_{\hat\MM_\mu},
\end{eqnarray*}
from which \eqref{BAP1} follows by triangle inequality.

Next  recall from \eqref{umu} that, in view of the first relation in \eqref{system-disc},
\begin{eqnarray*}
  \|u(\mu)- u_{V,W}(\mu)\|^2_{\HH_\mu}& = & \| u_{V,W}(\mu)\|^2_{\HH_\mu} = (u_{V,W}(\mu),u_{V,W}(\mu))_{\HH_\mu}
 =  \langle f- B_\mu p_{W,V}(\mu), u_{V,W}(\mu)\rangle\\
&\leq & \|f- B_\mu p_{W,V}(\mu)\|_{\HH_\mu'}\|u_{V,W}(\mu)\|_{\HH_\mu}
 =  \|p(\mu) - p_{W,V}(\mu)\|_{\hat\MM_\mu} \|u(\mu)-u_{V,W}(\mu)\|_{\HH_\mu},
\end{eqnarray*}
which together with \eqref{BAP1} confirms \eqref{BAP2}.

{Finally, the inf-sup estimate \eqref{inf-sup-0} is an immediate consequence of the more general Proposition \ref{propstab} below.} \end{proof}

We shall use the saddle point formulations to contrive rate-optimal RBMs, namely, on one hand,  for computing {\em truth snapshots} in $W=\MM_\mathcal{N}$
with a suitable $\delta$-proximal test space $\HH_\mathcal{N}$, and on the other hand,  for computing Galerkin projections in reduced spaces 
$W=\MM_n$ again with an associated $\delta$-proximal test space $\HH_n$, whose construction will be discussed in the next section.

To put this into proper perspective,
given any $W\subset \MM_\mu$, the condition \eqref{delta-prox} on a $V\subset \HH_\mu$ implies the
best approximation property {\bf BAP} for the Galerkin solution component $p_{W,V}(\mu)$ of \eqref{system-disc}
with a constant that becomes the closer to one the smaller the relative error becomes in \eqref{delta-prox}.
Moreover, \eqref{BAP2} says that the accuracy of the second ``auxiliary'' component $u_{V,W}(\mu)$ is
automatically  completely
governed by the accuracy of the first component $p_{W,V}(\mu)$.
Finally, \eqref{delta-prox} {\em implies}  {\em  inf-sup stability} of \eqref{system-disc}. It will be shown
below (for later purposes in a little more generality)  that conversely inf-sup stability \eqref{inf-sup-0} {\em implies} $\delta$-proximality.
In fact, since the bilinear form $a_\mu(v,w):= \langle R_{\HH_\mu}v,w\rangle = (v,w)_{\HH_\mu}$ is trivially $\HH_\mu$-elliptic
with coercivity and continuity constants $c_a(\mu)= C_a(\mu)=1$ (see \eqref{1.2.1}), we could have
derived the best approximation property {\bf BAP} \eqref{BAP2} {\em directly} from a uniform inf-sup condition
 from standard facts about  general saddle point problems,
see e.g. \cite{BF}. We have presented the relatively short self-contained derivation in order to identify the
precise constants and to bring out the particular role of the $\delta$-proximality condition \eqref{delta-prox}.
As we shall show later both conditions \eqref{delta-prox} and \eqref{inf-sup-0} can be used algorithmically 
to ensure stability of the saddle point problem and hence tightness of corresponding residual based surrogates.

The above discussion draws essentially on the use of the particular norm \eqref{Mu2}. In the context of
classical  saddle point problems, such as the Stokes system, it is more convenient to work
with the ``original'' $\MM_\mu$-norm related to $\|\cdot\|_{\hat\MM_\mu}$ by \eqref{eq:M-equiv}.
The following proposition clarifies the announced  interrelation between an inf-sup condition and $\delta$-proximality.

\begin{prop}
\label{propstab}
As before assume that $V\subset \HH_\mu$ and let $W\subset \MM_\mu$,   $0 \le \delta < 1$, and $\lambda > 0$. Consider the two   conditions:
\begin{enumerate}
  \item \label{stab1}
  \begin{equation}
  \label{stabcond}
  \|(I-\Ps) R_{\oldHs}^{-1}\bo_\mu  q\|_{\oldHs}\leq \delta \|R_{\oldHs}^{-1}\bo_\mu q\|_{\oldHs},\quad \forall\,\, q\in W,
  \end{equation}
  \item 
  \begin{equation}
  \label{infsupmu}
  \inf_{q\in W} \sup_{v\in V}\frac{b_\mu(q, v)}{\|v\|_{\oldHs}\|q\|_{\MM_\mu}} \geq \lambda.
  \end{equation}
\end{enumerate}
Then (i) implies (ii) with constant $\lambda = c_M \sqrt{1-\delta^2}$. Conversely, (ii) implies (i) with constant
$\delta = \sqrt{1-C_M^2 \lambda^2}$ i.e., $\lambda = C_M^{-1} \sqrt{1-\delta^2}$,
where $c_M$ and $C_M$ are the constants form the norm equivalence \eqref{eq:M-equiv}.
\end{prop}

 Note that  when $c_M = C_M = 1$, e.g. in case we use the $\mnew$-norm for $\MM_\mu$, both stability conditions are equivalent.

\noindent
\begin{proof} We reformulate (i) and (ii) in terms of equivalent conditions that can be more easily compared. First, squaring   \eqref{stabcond}  and using that $P_{\HH_\mu,V}$ is the $\oldHs$-orthogonal projector, we obtain
\begin{align*}
  \|R_{\oldHs}^{-1}\bo_\mu  q\|_{\oldHs}^2 - \|P_{\HH_\mu,V} R_{\oldHs}^{-1}\bo_\mu  q\|_{\oldHs}^2 & \le \delta^2 \|R_{\oldHs}^{-1}\bo_\mu q\|_{\oldHs}^2, & \forall q & \in W,
\end{align*}
which is equivalent to
\begin{align*}
  \sqrt{1 - \delta^2} \|R_{\oldHs}^{-1}\bo_\mu  q\|_{\oldHs} & \le \|P_{\HH_\mu,V} R_{\oldHs}^{-1}\bo_\mu  q\|_{\oldHs}, & \forall q & \in W.
\end{align*}
By the definition \eqref{Mu2} of the graph norm $\|\cdot\|_{\mnew}$, this is equivalent to 
\begin{align}
  \|P_{\HH_\mu,V} R_{\oldHs}^{-1}\bo_\mu  q\|_{\oldHs} & \ge \sqrt{1 - \delta^2} \|q\|_{\mnew}, & \forall q & \in W.
  \label{eq:infsup-proof-2}
\end{align}

Next, to reformulate \eqref{infsupmu}. Obviously the inf-sup condition is equivalent to
\begin{align}
  \sup_{v\in V}\frac{b_\mu(q, v)}{\|v\|_{\oldHs}} & \geq \lambda \|q\|_{\MM_\mu}, & \forall q & \in W.
  \label{eq:infsup-proof-1}
\end{align}
From  \eqref{supremizer} in Remark \ref{remcond2} we know that the left hand side is maximized by the function $v = P_{\HH_\mu,V} R_{\oldHs}^{-1}\bo_\mu  q$ which yields
\begin{eqnarray}
\label{supb}
\sup_{v\in V}\frac{b_\mu(q, v)}{\|v\|_{\oldHs}} &=& \frac{\langle P_{\HH_\mu,V} R_{\oldHs}^{-1}\bo_\mu  q, \bo_\mu  q \rangle}{\|P_{\HH_\mu,V} R_{\oldHs}^{-1}\bo_\mu  q\|_{\oldHs}} 
=  \frac{\langle R_{\oldHs}P_{\HH_\mu,V} R_{\oldHs}^{-1}\bo_\mu  q, R_{\oldHs}^{-1} \bo_\mu  q \rangle}{\|P_{\HH_\mu,V} R_{\oldHs}^{-1}\bo_\mu  q\|_{\oldHs}} 
= \|P_{\HH_\mu,V} R_{\oldHs}^{-1}\bo_\mu  q\|_{\oldHs}.
\end{eqnarray}
Substituting the right hand side in the left hand side of the condition \eqref{eq:infsup-proof-1}, yields
\begin{align*}
  \|P_{\HH_\mu,V} R_{\oldHs}^{-1}\bo_\mu  q\|_{\oldHs} & \geq \lambda \|q\|_{\MM_\mu}, & \forall q & \in W.
\end{align*}
We see that this is condition is identical to \eqref{eq:infsup-proof-2} up to an equivalence of the $\|\cdot\|_{\MM_\mu}$ and $\|\cdot\|_{\mnew}$ norms, which proves the assertion. \end{proof}

In summary, given trial space $W\subset \MM_\mu$,
a suitable $V\subset \HH_\mu$ such that the Galerkin problem \eqref{system-disc} has the best approximation property
{\bf BAP}, thereby warranting tight residual based surrogates, can be obtained by realizing
\begin{equation}
\label{eq:unsym-infsup-VW}
\inf_{q\in \mf} \sup_{v\in \hf}\frac{b_\mu(q, v)}{\|v\|_{\HH_\mu}\|q\|_{\MM_\mu}}\geq   \beta,
\end{equation}
where $\beta := \min_{\mu\in\mathcal{P}}\beta(\mu)>0$, see \eqref{babuska}.

\subsection{Parameter Dependence, Truth Spaces, and Feasibility}\label{sect:pardep}

Before applying the above findings to the construction of well-conditioned tight surrogates, we need to be a bit more precise about the
{\em parameter dependence} in order to distinguish eventually several relevant scenarios. Notice that
the spaces $\HH_\mu, \MM_\mu$ are allowed to depend on $\mu\in\mathcal{P}$ in a way that they even differ as sets and no parameter independent
reference norm may exist, see Remark \ref{remYdiffer}.
Let
\begin{align}
  \HH & := \bigcap_{\mu\in\mathcal{P}} \HH_\mu, & 
  \MM & := \bigcap_{\mu\in\mathcal{P}} \MM_\mu,
\label{H}
\end{align}
where the intersection is understood in the sense of sets. It is clear that $Y$ and $X$ are linear spaces. 
Although in general, we do {\em not} insist though, that  $\HH$  and $\MM$ are 
 endowed with  norms that are equivalent to all $\|\cdot\|_{\HH_\mu}$ and $\|\cdot\|_{\MM_\mu}$, respectively. 
 However, we do assume in what follows that $\HH, \MM$ are dense in $\HH_\mu$, $\MM_\mu$, respectively,  for all $\mu \in \mathcal{P}$. 
 Moreover, on account of the compactness of $\mathcal{P}$, we can always define (possibly stronger) norms
 \begin{equation}
 \label{supnorms}
\|v\|_{\HH}:= \sup_{\mu\in \mathcal{P}} \|v\|_{\HH_\mu},\quad \|q\|_{\MM}:= \sup_{\mu\in\mathcal{P}} \|q\|_{\MM_\mu}
\end{equation}
for $\HH,\MM$, respectively. 
Moreover,   since $\HH = \bigcap_{\mu \in \mathcal{P}} \HH_\mu$ is assumed to be dense in $\HH_\mu$, for the inf-sup condition \eqref{babuska} it suffices to take for $V=\HH_\mu$ the supremum over $\HH$ instead of   $\HH_\mu$, i.e.,
there exist subspaces $\hf \subset \HH$ for which the  discrete inf-sup condition \eqref{infsupmu} holds uniformly in the parameter $\mu$.

Of course, this setting covers, in particular, the special situation - usually considered in the RBM context -  that all the spaces 
$\HH_\mu, \MM_\mu$, $\mu\in \mathcal{P}$, agree
 as sets, respectively, and where the respective norms are uniformly equivalent, i.e., there exist
 constants $0< c_\circ, C_\circ<\infty$ such that
\begin{equation}
\label{allHequiv}
c_\circ \|v\|_{\HH} \leq \|v\|_{\HH_\mu}\leq C_\circ \|v\|_{\HH},\quad \mu\in \mathcal{P},\,\, v\in \HH,
\end{equation}
and
\begin{equation}
\label{allMequiv}
c_\circ \|q\|_{\MM} \leq \|q\|_{\MM_\mu}\leq C_\circ \|q\|_{\MM},\quad \mu\in \mathcal{P},\,\, q\in \MM.
\end{equation}
Recall from Remark \ref{remYdiffer} that for parametric  transport equations \eqref{allMequiv} is valid but \eqref{allHequiv} does {\em not} hold.

At any rate, due to the denseness of $\MM$ and $\HH$,  we can find sufficiently large but {\em finite dimensional truth spaces} $\HH_\truth \subset \HH, \MM_\truth \subset \MM$, typically finite element spaces, that can provide a desired target accuray of the {\em truth model}.
Since we are dealing here with problems for which standard tight a posteriori bounds are not available, we comment first 
on the {\em truth certification}.  {Note that this is particularly  important for convection dominated convection diffusion equations when
a complete resolution of very steep layers is prohibitively expensive 
  even for the truth solution.} We know that $\|p(\mu)- q\|_{\hat\MM_\mu}=\|f- B_\mu q\|_{\HH_\mu'}$. In order to be able to
accurately evaluate the residual in the dual norm $\|\cdot\|_{\HH_\mu'}$ one needs in {\em any setting} suitable assumptions on 
{\em data oscillation}, see e.g. \cite{CKNS, V2, DHSW}. One way to express this is to require that the projection of $R_{\HH_\mu}^{-1}f $ into the test space $\HH_\mathcal{N}$
captures enough of $R_{\HH_\mu}^{-1}f$.  To this end, we make use of the following simple
observation.

\begin{rem}
\label{WMX}
Assume that \eqref{allMequiv} holds.
Given $W\subseteq \MM_\mathcal{N}$ and any $\delta\in (0,1)$, there exists a finite dimensional test space $V \subset \HH$ such that
\begin{equation}
\label{WVdelta}
\inf_{v\in V}\|q - R_{\hat\MM_\mu}^{-1}B_\mu^* v\|_{\hat\MM_\mu} \leq \delta \|q\|_{\hat\MM_\mu},\quad q\in \mathcal{M}_{\MM} + W, \,\,\mu\in \mathcal{P},
\end{equation}
 {which implies}
\begin{equation}
\label{res-est-W}
(1-\delta^2)^{1/2} \| f- B_\mu p\|_{\HH_\mu'} \leq \| P_{\HH_\mu,V}R_{\HH_\mu}^{-1}(f- B_\mu p)\|_{\HH_\mu} \leq \| f- B_\mu p\|_{\HH_\mu'},
\quad p\in W.
\end{equation}
 In the following, we denote by $\mathcal{V}(W, \delta)$ all test spaces in $\HH_\mathcal{N}$ which satisfy the stability condition \eqref{WVdelta}.
\end{rem}

\begin{proof} {Since $\mathcal{M}_{\MM}$ is compact there is a linear space $V_\mathcal{M}$ such that
\[
\inf_{v\in V_\mathcal{M}}\|q - R_{\hat\MM_\mu}^{-1}B_\mu^* v\|_{\hat\MM_\mu} \leq \delta \inf_{w \in W} \|q + w\|_{\hat\MM_\mu},\quad q\in \mathcal{M}_{\MM}, \,\,\mu\in \mathcal{P}.
\]
It follows that the space $V_\mathcal{M} + W$ satisfies \eqref{WVdelta}.} Furthermore, since
\begin{equation}
\label{thesame}
\inf_{v\in V}\|q - R_{\hat\MM_\mu}^{-1}B_\mu^* v\|_{\hat\MM_\mu} \leq \delta \|q\|_{\hat\MM_\mu}\quad\Longleftrightarrow\quad
\inf_{v\in V}\|  R_{\hat\MM_\mu}^{-1}B_\mu q-  v\|_{\hat\MM_\mu} 
 \leq \delta \|R_{\hat\MM_\mu}^{-1}B_\mu  q\|_{ \HH_\mu},
\end{equation}
and since $\inf_{v\in V}\|  R_{\hat\MM_\mu}^{-1}B_\mu q-  v\|_{\hat\MM_\mu}  = \|(I- P_{\HH_\mu,V})R_{\hat\MM_\mu}^{-1}B_\mu  q\|_{\HH_\mu}$,
the assertion follows. \end{proof}

We shall comment later how \eqref{WVdelta} can be realized, see also \cite{DHSW,W} for a more detailed  discussion.
Since 
$ r_{V,W}(p,f):= P_{\HH_\mu,V}R_{\HH_\mu}^{-1}(f- B_\mu p)$ is given by
\begin{equation}
\label{res-truth-2}
\langle R_{\HH_\mu} r_{V,W}(p,f), z\rangle = \langle f- B_\mu p,z\rangle,\quad z\in \HH_\mathcal{N},
\end{equation}
the middle term in \eqref{res-est-W} is computable.

 \begin{rem}
 \label{rem:XYchoice}
 In what follows we shall always assume that for some fixed $\delta_\mathcal{N}<1$ and any given $\MM_\mathcal{N} \subset \MM$,
 the finite dimensional space $\HH_\mathcal{N}$ is  contained in $\mathcal{V}(\MM_\mathcal{N},\delta_\mathcal{N})$ satisfying \eqref{WVdelta}.
 Therefore, 
 abbreviating the solution of
\eqref{system-disc} for $W=\MM_\mathcal{N}, V=\HH_\mathcal{N}$, as $p_{\MM_\mathcal{N},\HH_{\mathcal{N}}} =:  p_\mathcal{N}(\mu)\in \MM_\mathcal{N}$,
$u_\mathcal{N}(\mu):= u_{\HH_\mathcal{N},\MM_\mathcal{N}}(\mu)$, we immediately conclude that
\begin{equation}
\label{X-bound}
\|p(\mu)- p_\mathcal{N}(\mu)\|_{\hat\MM_\mu}\leq  (1-\delta^2)^{-1/2} 
\| u_\mathcal{N}(\mu)\|_{\HH_\mu},
\quad \mu\in \mathcal{P}.
\end{equation}
\end{rem}

\begin{rem}
\label{rem:truth}
(i) For any desired target tolerance $\tau$, as soon as the computable quantity $\| u_\mathcal{N}(\mu)\|_{\HH_\mu}$ drops below 
$(1-\delta^2)^{-1/2} \tau$ we know that the truth solution has guaranteed accuracy $\leq \tau$ which can be achieved by the
refinement scheme in \cite{DHSW}. \\
(ii) The above choice of $\HH_\mathcal{N}$ guarantees, by Propositions \ref{prop:stab2},  \ref{propstab}, in particular, that
\begin{equation}
 \inf_{q\in \MM_\mathcal{N}} \sup_{v \in \HH_\mathcal{N}}\frac{b_\mu(q, v)}{\|v\|_{\HH_\mu}\|q\|_{\bar\MM_\mu}}\geq  \xi \sqrt{1-\delta_\mathcal{N}^2}=:  \beta_\mathcal{N} > 0,\quad \mu\in\mathcal{P},
 \,\,\mbox{where}\,\, \xi :=\left\{
 \begin{array}{ll}
 1, & \|\cdot\|_{\bar\MM_\mu}=  \|\cdot\|_{\hat\MM_\mu},\\
 \beta, &  \|\cdot\|_{\bar\MM_\mu}=  \|\cdot\|_{\MM_\mu} ,
 \end{array}\right.
 \label{eq:inf-sup}
\end{equation}
and where $\beta$ is the inf-sup constant from \eqref{eq:unsym-infsup-VW}.
Hence, $\beta_\mathcal{N}$ can, in principle be driven as close as one wishes to one or $\beta$, depending on the 
choice of norm for $\MM_\mu$. 
\end{rem}

 Note that the above statements do not contradict the possible case that the norms $\|\cdot\|_{\HH_\mu}$ or $\|\cdot \|_{\MM_\mu}$, $\mu\in\mathcal{P}$, are {\em not} equivalent to a single reference norm.

  In the following, we shall often not distinguish
for simplicity of exposition between truth and full spaces unless explicitly stated. In particular, whenever we speak of a 
computation in $\HH_\mu, \MM_\mu$ we refer to the truth spaces endowed with the norms 
$\|\cdot\|_{\HH_\mu}, \|\cdot\|_{\hat\MM_\mu}$, respectively.

Finally, the way how the bilinear forms depend on $\mu$ is important for practical {\em feasibility}.
We assume that the dependence of the bilinear forms on $\mu$ is {\em affine} in the usual sense,
i.e.
\begin{equation}
\label{affine}
b_\mu(\cdot,\cdot)= \sum_{k=1}^{m_B}\Theta^b_k(\mu)b_k(\cdot,\cdot),
\end{equation}
with parameter independent bilinear forms 
$b_k(\cdot,\cdot)$, $k=1,\ldots,m_B$,
and smooth functions $ \Theta_k^b$.

\section{Stabilization}
\label{sec:stab}
Suppose we are given a pair  $W= \MM_n\subset \MM_\mathcal{N}$, $\HH_n\subset \HH_\mathcal{N}$ of finite dimensional spaces
with bases   $\Phi_n=\{\phi_j\}_{j=1}^{n}$  and $\Psi_n=\{\psi_j\}_{j=1}^{m(n)}$, respectively. Our convention will always be that the index $n$ reflects the dimension of $\MM_n$ while generally ${\rm dim}\,\HH_n =m(n)\geq n$. While the purpose of $\MM_n$ is to approximate
$\mathcal{M}_{\MM}$ the role of $\HH_n$ is, in view of Proposition \ref{prop:stab2},  to guarantee uniform inf-sup stability.
More precisely, whenever $\HH_n$ is $\delta$-proximal for $\MM_n$ \eqref{delta-prox} for some $\delta<1$, one has
\begin{equation}
\label{eq:unsym-infsup-XnYn}
\inf_{q\in \MM_n} \sup_{v\in \HH_n}\frac{b_\mu(q, v)}{\|v\|_{\HH_\mu}\|q\|_{\hat\MM_\mu}}\geq \sqrt{1-\delta^2}, 
\quad \mu\in\mathcal{P}.
\end{equation}  
 Hence, a natural strategy is to choose a constant $0< \zeta < 1$, replace the right hand side of \eqref{eq:unsym-infsup-XnYn} by $\zeta \sqrt{1-\delta^2}$ and {\em enrich} the space $\HH_n$ until this relaxed inf-sup condition is   valid. The closer one whishes  $p_{\MM_n,\HH_n}(\mu)$ to be to the best $\hat\MM_\mu$-approximation $P_{\hat\MM_\mu,\MM_n}
p(\mu)$, the closer $\zeta$ should be chosen to one, see \eqref{BAP1}. In particular, {\em any} $\zeta <1$ is in principle feasible.

We shall formulate actually two variants of such a {\em stabilization scheme} which apply under slightly different
assumptions.
\subsection{Inf-sup stabilization}

The first   natural idea which  has already been used in \cite{Rozza,GV,GV2} is to enrich $\HH_n$ by the supremizer for the infimizing 
parameter $\bar\mu$. More precisely,   we first search for a parameter $\bar{\mu} \in \mathcal{P}$ and a function $\bar{q} \in \MM_n$ for which the inf-sup condition \eqref{infsupmu} is worst, i.e. 
\begin{equation}
\label{arginfsup}
\sup_{v\in \HH_n}\frac{b_{\bar{\mu}}(\bar{q}, v)}{\|v\|_{\HH_{\bar{\mu}}}\|\bar{q}\|_{\MM_{\bar{\mu}}}} = \inf_{\mu\in \mathcal{P}} \left( \inf_{q\in \MM_n} \sup_{v\in \HH_n}\frac{b_\mu(q, v)}{\|v\|_{\oldHs}\|q\|_{\MM_\mu}} \right).
\end{equation}
If this worst case inf-sup constant does not exceed yet a desired uniform lower bound, $\HH_n$ does not contain
an effective supremizer  {for $\bar{\mu}, \bar{q}$,} yet.  However, since the truth space satisfies the uniform inf-sup condition \eqref{eq:inf-sup}  there exists a good supremizer in the truth space which, on account of  Remark \ref{remcond2}, is given by
\[
\bar{v} = R_{\HH_{\bar{\mu}}}^{-1}\bo_{\bar{\mu}} \bar{q} = \argmax_{v \in \HH_{\bar{\mu}}}\frac{b_{\bar{\mu}}(\bar{q}, v)}{\|v\|_{\HH_{\bar{\mu}}} \|\bar{q}\|_{\MM_{\bar{\mu}}}},
\]
and provides the enrichment
\begin{equation}
  \HH_n \to {\rm span} \{ \HH_n, R_{\oldHs}^{-1}\bo_\mu  \bar{q}\}.
  \label{eq:update}
\end{equation}
This strategy can now be applied recursively until we reach a satisfactory uniform inf-sup condition for the reduced spaces.

Of course, three questions  immediately arise:
\begin{itemize}

\item[(i)] Is the computation of $\bar{\mu}$ and $\bar{q}$ feasible?

\item[(ii)]
Does this process terminate after finitely many  steps?
\item[(iii)]
If so, is the number of  necessary stabilization steps affordable?
\end{itemize}

 Assuming for the moment
to have positive answers to (ii) and (iii), we first derive
  a suitable {\em offline/online} strategy for an efficient implementation of \eqref{eq:update}. First note that for given $\bar{\mu}$ and $\bar{q}$ the new test function $\bar v:=R_{\HH_{\bar\mu}}^{-1}\bo_{\bar\mu}  \bar{q}$ can be computed by a standard Galerkin scheme
\begin{equation*}
(\bar v,v)_{\HH_{\bar\mu}}= b_{\bar\mu}(\bar q,v),\quad v\in \HH_{\bar\mu}, 
\end{equation*}
so that it remains to solve the optimization problem \eqref{arginfsup} to find $\bar{\mu}$ and $\bar{q}$. To this end, we first rewrite the inf-sup condition in terms of the coefficient vectors with respect to the reduced bases. To describe this, we denote the corresponding Gramians, respectively cross-Gramians as
\begin{equation}
\label{matrices}
\begin{array}{ccccc}
\bf{R}_{\oldHs}&:=& (\Psi,\Psi)_{\oldHs}&:= &\big((\psi_i,\psi_j)_{\oldHs}\big)_{i,j=1}^{m},\\
\bf{R}_{\MM_\mu} &:=& ( \Phi,\Phi)_{\MM_\mu}&:= & \big(( \phi_i,\phi_j)_{\MM_\mu}\big)_{i,j=1}^{n},\\
\bbo_\mu & :=& b_\mu(\Phi,\Psi)& := &\big(b_\mu(\phi_i,\psi_j)\big)_{j,i=1}^{m,n}.
\end{array}
\end{equation}
Practical feasibility relies on the following
\begin{assumption}
\label{assR}
In addition to $\bo_\mu$ the Riesz maps $R_{\HH_\mu}$ and $R_{\MM_\mu}$ depend affinely on the parameter $\mu$.
\end{assumption}
\begin{rem}
\label{remR}
Under Assumption \ref{assR} all the matrices in \eqref{matrices} can be computed online. 
Thus, by rewriting the left hand side of the inf-sup condition as
\begin{equation*}
\inf_{\boldsymbol{q} \in \R^n} \sup_{\boldsymbol{v} \in \R^{m(n)}} \frac{\boldsymbol{v}^T \bbo_\mu \boldsymbol{q}}{ \left(\boldsymbol{v}^T \bf{R}_{\oldHs} \boldsymbol{v}\right)^{1/2}  \left(\boldsymbol{q}^T \bf{R}_{\MM_\mu} \boldsymbol{q}\right)^{1/2} }
\label{eq:lpp-problem}
\end{equation*}
 we are left for each parameter $\mu$ with an optimization problem only of the size of the dimensions $m(n), n$ of $\HH_n$ and $\MM_n$, respectively.
\end{rem}
 In order to find an infimizing  $\boldsymbol{q} \in \R^{n}$ we eliminate the discrete Riesz maps in the denominator by factoring them as 
\begin{align}
\label{Rchol}
  \bf{R}_{\oldHs} & = \boldsymbol{L}_{\HH_\mu}^T \boldsymbol{L}_{\HH_\mu} &
  \bf{R}_{\MM_\mu} & = \boldsymbol{L}_{\MM_\mu}^T \boldsymbol{L}_{\MM_\mu}.
\end{align}
Here, one can think of a Cholesky factorization or of a spectral decomposition $ \boldsymbol{L}_{\HH_\mu}
=\Lambda_{\MM_\mu}^{1/2} \bf{Q}_{\MM_\mu}$ where the columns of $\bf{Q}_{\MM_\mu}$ form an eigenbasis and $\Lambda_{\MM_\mu}$
is the diagonal matrix with the eigenvalues (in descending order) on the diagonal.
 Replacing $\boldsymbol{v}$ by $\boldsymbol{L}_{H_\mu} \boldsymbol{v}$ and $\boldsymbol{q}$ by $\boldsymbol{L}_{\MM_\mu} \boldsymbol{q}$ and defining $\boldsymbol{D}_\mu := \boldsymbol{L}_{H_\mu}^{-T} \bbo_\mu \boldsymbol{L}_{\MM_\mu}^{-1}$ we find that
\begin{equation}
\label{infDmu}
  \inf_{\boldsymbol{q} \in \R^n} \sup_{\boldsymbol{v} \in \R^{m(n)}} \frac{\boldsymbol{v}^T \bbo_\mu \boldsymbol{q}}{ \left(\boldsymbol{v}^T \bf{R}_{\oldHs} \boldsymbol{v}\right)^{1/2}  \left(\boldsymbol{q}^T \bf{R}_{\MM_\mu} \boldsymbol{q}\right)^{1/2} }
  = \inf_{\boldsymbol{q} \in \R^n} \sup_{\boldsymbol{v} \in \R^{m(n)}} \frac{\boldsymbol{v}^T  \boldsymbol{D}_\mu \boldsymbol{q}}{\| \boldsymbol{v} \|_{\ell_2} \| \boldsymbol{q} \|_{\ell_2} }.
\end{equation}
and hence, one easily verifies the following fact.
\begin{rem}
\label{remSVD}
For any given $\mu$ the corresponding inf-sup constant is the smallest singular value of $\boldsymbol{D}_\mu$ and the optimal $\boldsymbol{q}$ is the corresponding right singular vector. Since the computational cost of the singular value decomposition  is polynomial in the dimensions of the reduced bases, we can afford to compute all the inf-sup constants for a sufficiently large sample set $\mathcal{S}\subset \mathcal{P}$ of parameters, yielding the optimal $\bar{\mu}$.
\end{rem}
 The complete scheme is summarized in Algorithm \ref{alg:update-inf-sup} which we formulate for the general norms $\MM_\mu$
 in \eqref{babuska} and the inf-sup constant $\beta_\mathcal{N}$ from \eqref{eq:inf-sup}.

\begin{algorithm}[htb]
  \caption{Update to achieve inf-sup stability.}
  \label{alg:update-inf-sup}
  \begin{algorithmic}[1]
    \Function{Update-inf-sup}{$\HH_n$, $\MM_n$}
    \State Choose $0 < \zeta < 1$.
    \State Select a sufficiently large sample $\mathcal{S} \subset \mathcal{P}$.
    \Repeat
      \For{$\mu \in \mathcal{S}$}
        \State Assemble the Gramians and cross-Gramians $\bf{R}_{\oldHs}$, $\bf{R}_{\MM_\mu}$, $\bbo_\mu$.
        \State Compute the Cholesky decompositions
          \begin{align*}
            \bf{R}_{\oldHs} & = \boldsymbol{L}_{\HH_\mu}^T \boldsymbol{L}_{\HH_\mu} &
            \bf{R}_{\MM_\mu} & = \boldsymbol{L}_{\MM_\mu}^T \boldsymbol{L}_{\MM_\mu}.
          \end{align*}
        \State Determine the smallest singular value $\sigma(\mu)$ and corresponding 
        \Statex \hskip\algorithmicindent \hskip\algorithmicindent \hskip\algorithmicindent right singular vector $\bar{\boldsymbol{q}}_\mu$ of the matrix $\boldsymbol{D}_\mu = \boldsymbol{L}_{\HH_\mu}^{-T} \bbo_\mu \boldsymbol{L}_{\MM_\mu}$.
         
      \EndFor
      \State Set $\bar{\mu} = \min \{ \sigma(\mu) : \mu \in \mathcal{S} \}$
      \State Update $\HH_n \leftarrow {\rm span} \{ \HH_n, R_{\HH_{\bar{\mu}}}^{-1}\bo_{\bar{\mu}}  \bar{q}_{\bar{\mu}}\}$ with $\bar{q}_{\bar{\mu}} = \sum_{i=1}^{n} (\bar{\boldsymbol{q}}_{\bar{\mu}})_i \phi_i$.
    \Until{$\sigma(\bar\mu) \ge  \zeta\beta_\mathcal{N} $} \label{stop1}
    \State \Return{$\HH_n$}
    \EndFunction
  \end{algorithmic}
\end{algorithm}

\subsection{Stabilization based on \texorpdfstring{$\delta$}{delta}-Proximality }

We shall now formulate an alternative stabilizing scheme. It is related to greedy approximation and  will shed some
light on the above stabilization algorithms regarding the questions (ii), (iii).
The idea is to enrich the space $\HH_n$ to obtain stability based on the equivalent criterion \eqref{stabcond} which 
can be rephrased as
\begin{equation}
\label{eq:inf}
\inf_{\phi \in \HH_n} \|R_{\oldHs}^{-1}\bo_\mu  q - \phi\|_{\oldHs}\leq \delta \|R_{\oldHs}^{-1}\bo_\mu q\|_{\oldHs},\quad \forall\,\, q\in \MM_n,\, \mu\in \mathcal{P}.
\end{equation}
Defining 
\begin{equation}
\MnOne(\mu) := \left\{ q \in \MM_n : \, \|q\|_{\mnew}=\|R_{\oldHs}^{-1}\bo_\mu q\|_{\oldHs}=1 \right\},
\label{eq:manifold1}
\end{equation}
this is equivalent to
\begin{equation}
\sup_{\mu \in \mathcal{P}} \sup_{q \in \MnOne(\mu)} \inf_{\phi \in \HH_n} \|R_{\oldHs}^{-1}\bo_\mu  q - \phi\|_{\oldHs}\leq \delta.
\label{eq:manifold2}
\end{equation}
We can again employ a greedy strategy to search for  {the parameter $\mu \in \mathcal{P}$ and} the element in $\MnOne:= \bigcup_{\mu\in\mathcal{P}}\MnOne(\mu)$ for which the error is worst:
\begin{equation}
(\bar{\mu}, \bar{q}) = \argmax_{\mu \in \mathcal{P}; q \in \MnOne(\mu)} \inf_{\phi \in Y_n} \|R_{\oldHs}^{-1}\bo_\mu  q - \phi\|_{\oldHs}.
\label{eq:stab-greedy}
\end{equation}
As long as the approximation error for $\bar{\mu}$ and $\bar{q}$ exceeds some fixed $\delta \in (0,1)$, we add the best approximation from the full truth space to the reduced basis:
\begin{equation*}
  \HH_n \to {\rm span} \{ \HH_n, R_{\HH_{\bar\mu}}^{-1}\bo_{\bar\mu}  \bar{q}\}.
\end{equation*}
Since, as pointed out below \eqref{Mu2}, $R_{\HH_\mu}^{-1}= B_\mu^{-*}R_{\hat\MM_\mu}B_\mu^{-1}$, we see
that
$(R_{\HH_\mu}^{-1})^{-1} = B_\mu R_{\hat\MM_\mu}^{-1}B_\mu^*$.
Hence, in view of \eqref{Mu2}, we conclude that
\begin{equation}
(\bar{\mu}, \bar{q}) = \argmax_{\mu \in \mathcal{P}; q \in \MnOne(\mu)} \big(\inf_{\phi\in \HH_{n }}\|q -  R_{\hat\MM_\mu}^{-1}B_\mu^*\phi
\|_{\hat\MM_\mu}\big),
\label{eq:stab-greedy2}
\end{equation}
which implies the following observation.
 
\begin{rem}
\label{rem-greedy}
If \eqref{allMequiv} holds so that all the spaces $\MM_\mu$ agree with a parameter independent reference space
$\MM$, the output  $(\bar{\mu}, \bar{q})$ is the result of a greedy approximation   to the compact 
 set $\MM_n^1:= \bigcup_{\mu\in\mathcal{P}} \MM_n^1(\mu)$.
 Therefore, in principle,  the scheme fits into the standard greedy theory in \cite{Buffaetal, BCDDPW, DPW}.
 In fact, by \eqref{allMequiv}, \eqref{affine}, and the fact that $B_\mu^*: \HH_\mu\to (\hat\MM_\mu)'$ is an isometry, 
 the set $R_{\MM}^{-1}B_\mu^* \HH_n$ is a finite dimensional subspace of $\MM$.  
\end{rem}

It remains to find a fast algorithm for the solution of the maximization problem \eqref{eq:stab-greedy}
which will make use of the $\|\cdot\|_{\mnew}$-norm \eqref{Mu2}  for $\MM_\mu$. 
\begin{lemma}
\label{lemstabcond}
 Let  $q= \sum_{j=1}^{n}q_j\phi_j =: \bf{q}^\top \Phi$. Referring to the matrices $\bbo_\mu, \bf{R}_{\oldHs}$ from \eqref{matrices},
 and defining $\bf{R}_{\mnew}:= (\Phi,\Phi)_{\mnew}$, 
one has
\begin{equation}
\label{matrixcond}
\|(I-\Ps)R_{\oldHs}^{-1}\bo_\mu  q\|^2_{\oldHs}=\bf{q}^\top \big(\bf{R}_{\hat \MM_\mu} - \bbo_\mu^\top \bf{R}_{\oldHs}^{-1} \bbo_\mu\big)\bf{q}  .
\end{equation}
\end{lemma}
\begin{proof}   By orthogonality of $\Ps$ and  \eqref{Mu2},
  we have
\begin{equation*}
  \|(I-\Ps)R_{\oldHs}^{-1}\bo_\mu  q\|^2_{\oldHs} = \|q\|_{\mnew}^2 - \|\Ps R_{\oldHs}^{-1}\bo_\mu  q\|^2_{\oldHs}.
\end{equation*}
By definition, we have $\|q\|_{\mnew}^2=\bf{q}^\top\bf{R}_{\mnew}\bf{q}$.
As for  the second term, note that $\Ps R_{\oldHs}^{-1} \bo_\mu q$ is the Galerkin solution of $R_{\oldHs} z = \bo_\mu q$.
Since for any $w\in \oldHs'$  the coefficient vector $\bf{z}$ of $\Ps R_{\oldHs}^{-1}w$ is given by 
\begin{equation*}
\bf{R}_{\oldHs} \bf{z}=\langle w,\Psi\rangle =: \big(\langle w,\psi_j\rangle\big)_{j=1}^{m},
\end{equation*}
we conclude that for $w:= \bf{q}^\top\bo_\mu \Phi$
one has 
\begin{equation*}
\bf{z} = \bf{R}_{\oldHs}^{-1} \langle \Psi,\bo_\mu \Phi\rangle\bf{q}=  \bf{R}_{\oldHs}^{-1}\bbo_\mu\bf{q}.
\end{equation*}
Hence
\begin{equation*}
  \|\Ps R_{\oldHs}^{-1}\bo_\mu  q\|_{\oldHs}^2 = \langle \bo_\mu q, \Ps R_{\oldHs}^{-1} \bo_\mu  q \rangle_{\oldHs} = \bf{q}^\top \bbo_\mu^\top \bf{R}_{\oldHs}^{-1} \bbo_\mu \bf{q}.
\end{equation*}
which confirms the  claim. \end{proof}

Similarly, by the definition \eqref{Mu2} of the $\hat \MM_\mu$-norm, we have $\|q\|_{\hat \MM_\mu} = \|R_{\oldHs}^{-1}\bo_\mu  q\|_{\oldHs}$ so that
\[
  \MnOne(\mu) = \left\{ q \in \MM_n: \, q = \bf{q}^\top \Phi, \, \bf{q}^\top \bf{R}_{\hat \MM_\mu} \bf{q} = 1 \right\}.
\]
It follows that the optimization problem \eqref{eq:stab-greedy} is equivalent to
\begin{equation}
\label{matrixopt}
(\bar{\mu}, \bar{\bf{q}}) = \argmax_{\mu \in \mathcal{P}; \bf{q} \in \R^n} \frac{\bf{q}^\top \big(\bf{R}_{\hat \MM_\mu} - \bbo_\mu^\top \bf{R}_{\oldHs}^{-1} \bbo_\mu\big)\bf{q}}{\bf{q}^\top \bf{R}_{\hat \MM_\mu} \bf{q}},
\end{equation}
where $\bar{\bf{q}}$ is the coefficient vector of $\bar{q}$. 
This problem can be solved analogously to the corresponding optimization problem \eqref{arginfsup}, \eqref{eq:lpp-problem} of the inf-sup condition so that we obtain the alternative algorithm \textsc{Update-$\delta$} for updating $\HH_n$.
\begin{algorithm}[htb]
  \caption{Update to achieve $\delta$-proximality.}
  \begin{algorithmic}[1]
    \Function{Update-$\delta$}{$\HH_n$, $\MM_n$}
    \State Choose $0 < \delta < 1$.
    \State Select a sufficiently large sample $\mathcal{S} \subset \mathcal{P}$.
    \Repeat
    \State Assemble the Gramians and cross-Gramians $\bf{R}_{\oldHs}$, $\bf{R}_{\hat \MM_\mu}$, $\bbo_\mu$ (see \eqref{matrices}).
    \State Compute
      \begin{align*}
        \delta_{max} & = \max_{\mu \in \mathcal{S}; \bf{q} \in \R^n} \frac{\bf{q}^\top \big(\bf{R}_{\hat \MM_\mu} - \bbo_\mu^\top \bf{R}_{\oldHs}^{-1} \bbo_\mu\big)\bf{q}}{\bf{q}^\top \bf{R}_{\hat \MM_\mu} \bf{q}}, \\
        (\bar{\mu}, \bar{\bf{q}}) & = \argmax_{\mu \in \mathcal{S}; \bf{q} \in \R^n} \frac{\bf{q}^\top \big(\bf{R}_{\hat \MM_\mu} - \bbo_\mu^\top \bf{R}_{\oldHs}^{-1} \bbo_\mu\big)\bf{q}}{\bf{q}^\top \bf{R}_{\hat \MM_\mu} \bf{q}}.
      \end{align*}
      \State Update $\HH_n \leftarrow {\rm span} \{ \HH_n, R_{\HH_{\bar{\mu}}}^{-1}\bo_{\bar{\mu}}  \bar{q}_{\bar{\mu}}\}$ with $\bar{q}_{\bar{\mu}} = \sum_{i=1}^{n} (\bar{\boldsymbol{q}}_{\bar{\mu}})_i \phi_i$.
    \Until{$\delta_{max} \le \delta$} \label{stop2}
    \State \Return{$\HH_n$}
    \EndFunction
  \end{algorithmic}
\end{algorithm}

The efficient practical execution of Algorithm \textsc{Update-$\delta$} requires assembling the matrices $\bf{R}_{\hat \MM_\mu}$
in the typical offline/online fashion. This is possible   when instead of Assumption \ref{assR} the following holds.

\begin{assumption}
\label{assM}
 {The Riesz maps $R_{\HH_\mu}$, $R_{\mnew}$ and hence their inner products  $(\cdot,\cdot)_{\HH_\mu}$, $(\cdot,\cdot)_{\mnew}$ depend affinely on the parameter $\mu\in\mathcal{P}$.}
\end{assumption}

By \eqref{Mu2}, Assumption \ref{assM} is valid if the $\HH_
\mu$-norm can be chosen independent of $\mu$, i.e., when \eqref{allHequiv} holds.  {Moreover Assumption \ref{assM} can also be satisfied for parameter dependent $\HH_\mu$-norms as e.g., in view of  \eqref{specialRiesz}, for  the transport equation.}

Finally, it is important to note that the number of operations used by both algorithms \textsc{Update-Inf-Sup} 
and \textsc{Update-$\delta$} (under Assumption \ref{assM})
only depends on the size of the sample $\mathcal{S}$ and the dimensions $n$ and $m(n)$ of the reduced bases. Especially, it is independent of the dimension of the truth spaces which renders these algorithms feasible. 

Assumption \ref{assM} is clearly more restrictive, i.e. the use of \textsc{Update-$\delta$} is more constrained than 
\textsc{Update-Inf-Sup} which applies under the standard assumptions of affine dependence and for any norm on $\MM_\mu$.

\subsection{Interrelation between both Stabilization Schemes}

We discuss next the interrelation between the schemes \textsc{Update-Inf-Sup} and  \textsc{Update-$\delta$}.

\begin{prop}
  \label{prop:update-equal}
  Assume that we use the $\|\cdot\|_{\mnew}$-norm for $\MM_\mu$ and the spectral decomposition 
  \begin{equation*}
  \bf{R}_{\mnew} = \boldsymbol{L}_{\MM_\mu}^T \boldsymbol{L}_{\MM_\mu} = \bf{Q}_\mu^\top\Lambda_\mu^{1/2}\Lambda_\mu^{1/2}\bf{Q}_\mu.
  \end{equation*}
  in \eqref{Rchol} for the scheme \textsc{Update-inf-sup}, where $\Lambda_\mu$ is the diagonal matrix with the eigenvalues and $\bf{Q}_\mu$ the matrix of corresponding eigenvectors. Then the outputs of $\textsc{Update-inf-sup}~ and~ \textsc{Update-$\delta$}$ coincide.
\end{prop}
\begin{proof}
Let $\bf{M}_\mu :=   \bbo_\mu^\top \bf{R}_{\oldHs}^{-1} \bbo_\mu$. Clearly, since $\bf{Q}_\mu$ is orthogonal,
\begin{equation*}
\lambda_{\max}(\mu):=\max_{  \bf{q} \in \R^n} \frac{\bf{q}^\top \big(\bf{R}_{\hat \MM_\mu} - \bbo_\mu^\top \bf{R}_{\oldHs}^{-1} \bbo_\mu\big)\bf{q}}{\bf{q}^\top \bf{R}_{\hat \MM_\mu} \bf{q}}
\end{equation*}
is the largest eigenvalue of the matrix $\bf{I} - \Lambda_\mu^{-1/2}\bf{Q}_\mu \bf{M}_\mu\bf{Q}_\mu^\top \Lambda_\mu^{-1/2}$
so that
\begin{equation}
\label{lambdamin}
\lambda_{\max}(\mu) =1- \lambda_{\min}(\Lambda_\mu^{-1/2}\bf{Q}_\mu \bf{M}_\mu\bf{Q}_\mu^\top \Lambda_\mu^{-1/2}).
\end{equation}
On the other hand, using the $\mnew$-norm in \textsc{Update-Inf-Sup}, i.e.,  replacing $\bf{R}_{\MM_\mu}$ by $\bf{R}_{\mnew}$
in \eqref{Rchol}, and using the spectral decomposition $\bf{R}_{\mnew}=\bf{Q}_\mu^\top\Lambda_\mu^{1/2}\Lambda_\mu^{1/2}\bf{Q}_\mu$ for $\bf{L}_{\MM_\mu}^T\bf{L}_{\MM_\mu}$, 
the matrix $\bf{D}_\mu$ in \eqref{infDmu} takes the form
$\bf{D}_\mu = \bf{L}_{H_\mu}^{-\top}\bbo_\mu \bf{Q}_\mu^\top \Lambda^{-1/2}$. Clearly, the smallest singular value of $\bf{D}_\mu$
is just $\lambda_{\min}(\bf{D}_\mu^\top\bf{D}_\mu)^{1/2}$ and the corresponding eigenvector agrees with the right singular vector of
$\bf{D}_\mu$. Since   
\begin{equation*}
\bf{D}_\mu^\top\bf{D}_\mu= \Lambda_\mu^{-1/2}\bf{Q}_\mu \bf{M}_\mu\bf{Q}_\mu^\top \Lambda_\mu^{-1/2}
\end{equation*}
we see that in this case the enrichments procued by both schemes agree, which confirms the claim.
\end{proof}

\subsection{Termination of Stabilizing Greedy Loops}\label{sec:term}
\subsubsection{The general Case}

Under the most general assumptions,  neither insisting on \eqref{allHequiv} nor on \eqref{allMequiv} we 
resort to a very crude argument that ensures termination of the
stabilizations loops
~ \textsc{Update-inf-sup} ~and ~\textsc{Update-$\delta$}.  
Our findings can be summarized as follows.
\begin{prop}
\label{thmterm}
Both schemes \textsc{Update-Inf-Sup} and  \textsc{Update-$\delta$} always terminate after finitely many
steps. 
 \end{prop}
\begin{proof}
 We   prove the assertion only for the scheme \textsc{Update-Inf-Sup}. The argument for \textsc{Update-$\delta$} is identical.
To this end, let $\HH_{n+1}$ and $\MM_{n+1}$ be the spaces obtained by applying \textsc{Update-Inf-Sup} to the input spaces $\HH_n$ and $\MM_n$. According to the update rule \eqref{eq:update} used by \textsc{Update-Inf-Sup}, the enlarged space $\HH_{n+1}$ is contained in the truth space $\HH_\mathcal{N}$. Thus, since $\HH_\mathcal{N}$ is finite dimensional, the statement of the proposition follows if each added function is linearly independent to the previous ones. To this end, assume the algorithm has already grown $\HH_n$ to $\tilde{\HH}_n$ and let $\bar v = R_{\HH_{\bar\mu}}^{-1} \bo_{\bar{\mu}} \bar{q}$ be the next function to be added (see \eqref{supremizer}). Now, assume by contradiction that it is already contained in $\tilde{\HH}_n$. Since $\bar v$ is a supremizer this implies that 
\begin{equation*}
  \sup_{v\in \tilde{\HH}_n}\frac{b_{\bar{\mu}}(\bar{q}, v)}{\|v\|_{\HH_\mu}\|\bar{q}\|_{\hat\MM_\mu}} \ge \zeta \beta_\mathcal{N}.
\end{equation*}
Recalling that $\bar{\mu}$ and $\bar{q}$ are the worst possible choices according to \eqref{arginfsup}, this violates the stopping criterion in Line \ref{stop1} of \textsc{Update-Inf-Sup}. Thus, it follows that $\bar v$ is linearly independent from $\tilde{\HH}_n$ showing finite termination of \textsc{Update-Inf-Sup}.
\end{proof}

The fact that, by the above argument, the number of stabilization steps may depend on the dimension of the truth
space is certainly very pessimistic and not satisfactory from a practical point of view. In fact, much more can be said under some additional assumptions.

\subsection{Uniformly Equivalent Norms }
Suppose now that all the spaces $\HH_\mu$, $\MM_\mu$ agree as sets with $\HH$, $\MM$, respectively (see \eqref{H}), and that 
\eqref{allHequiv}, \eqref{allMequiv})  hold. Then we can replace $\|\cdot\|_{\HH_\mu}$ by an uniformly equivalent reference
norm $\|\cdot\|_{\HH}$. Since the Riesz map $R_{\HH}$ is now independent of $\mu$,
Assumption \ref{assM} holds and the stabilizing schemes 
 \textsc{Update-inf-sup} and \textsc{Update-$\delta$} are equivalent, see Proposition \ref{prop:update-equal}.  
Moreover, recall that, by Remark \eqref{remcond3}, the supremizer for $q_n \in \MM_n$ in the inf-sup condition is given by $R_{\HH}^{-1} \bo_\mu q_n$. The key observation  is that because of the affine decomposition \eqref{affine} of $\bo_\mu$ all these supremizers together 
generate a finite dimensional space.

\begin{rem}
\label{remufinite}
Given $\MM_n\subset \MM$, $\MM_n={\rm span}\,\{\phi_j:j=1,\ldots,n\}$, let
\begin{align}
\label{hatH}
\hat{\HH}_n  & := \left\{ R_{\HH}^{-1}\bo_\mu p: \, p \in \MM_n, \, \mu \in \mathcal{P} \right\} \subseteq {\rm span} \left\{R_{\HH}^{-1} \bo_\mu \phi_j: \,
j=1,\ldots,n, \, \mu \in \mathcal{P} \right\}.
\end{align}
Then $\hat{\HH}_n$ is a finite dimensional space of dimension ${\rm dim}\,\hat \HH_n\leq m_Bn$, where $m_B$ is the number of terms in the affine expansion \eqref{affine}. Hence, 
one has
\begin{equation}
\label{infsup4}
\inf_{\mu\in \mathcal{P}} \inf_{q\in \MM_n}\sup_{v\in \hat \HH_n}\frac{b_\mu(q,v)}{\|q\|_{\MM_\mu} \|v\|_{\HH}} 
= \inf_{\mu\in \mathcal{P}} \inf_{q\in \MM_n}\sup_{v\in  \HH_\mathcal{N}}\frac{b_\mu(q,v)}{\|q\|_{\MM_\mu} \|v\|_{\HH}}\ge \beta_\mathcal{N} ,
\end{equation}
where $\beta_\mathcal{N} >0$ is the inf-sup constant from \eqref{eq:inf-sup} in Remark \ref{rem:truth}.
\end{rem}
\begin{proof} If $\bo_k$ is the operator corresponding to the bilinear form $b_k(\cdot, \cdot)$ in the affine expansion \eqref{affine} and $R_{\HH} = R_{\HH_\mu}$ is independent of $\mu \in \mathcal{P}$, we conclude that  
\begin{equation*}
\hat{\HH}_n  \subseteq {\rm span} \left\{R_{\HH}^{-1} \bo_k \phi_j: \,
j=1,\dots,n, \, k = 1, \dots, m_b \right\}
\end{equation*}
which proves the first part of the claim. Since all optimal test functions are contained in $\hat{\HH}_n$ the discrete inf-sup condition \eqref{infsup4} follows immediately from the assumed inf-sup condition \eqref{eq:inf-sup} of the full problem.
\end{proof}

The following simple observation is an immediate consequence of Remark \ref{remufinite}.

\begin{prop}
\label{propmB}
   Assume that \eqref{allHequiv}, hold. Then the update algorithm \textsc{Update-Inf-Sup}, and hence likewise \textsc{Update-$\delta$}, increases the dimension of the test space in each step and terminates with a test space of dimension at most $n m_B$.
\end{prop}
\begin{proof} 
The proof is identical to the one of Theorem \ref{thmterm} by noting that all supremizers that are added during the algorithm are not only contained in the truth space $\HH_\mathcal{N}$ but in the much smaller space $\hat{\HH}_n \subset \HH_\mathcal{N}$ which is of dimension $m_B n$.
\end{proof}

The above reasoning applies verbally to other saddle point problems like those appearing in parameter dependent Stokes
systems or constrained optimization problems. The finite dimensionality of $\hat\HH_n$ is also the basis of the {\em a priori}
choice of stabilizers in
 \cite{GV, GV2, Rozza} to guarantee inf-sup stability although the connection with a greedy stabilization
 does not seem to be made there. 
 
 The reason for nevertheless applying  such a greedy stabilization is that a sufficient
 inf-sup stability might actually be achieved at an earlier stage so that in total fewer stabilizers suffice.

\subsection{A Greedy Perspective}
\label{sec:greedy}
As we shall see in later applications, in the context of Section \ref{sec:recipe} it will be important to treat also
the case where only \eqref{allMequiv} holds but \eqref{allHequiv} is {\em not valid}. In this case
the norms $\|\cdot\|_{\MM_\mu}, \|\cdot\|_{\hat\MM_\mu}$ are all equivalent and can be replaced by a parameter independent
reference norm $\|\cdot\|_{\MM}$. For instance, in the case  \eqref{transnorms} one has
$\|\cdot\|_{L_2(\Omega)}= \|\cdot\|_{\hat\MM_\mu} =\|\cdot\|_{\MM}$, $\mu\in \mathcal{P}$, which will be further discussed in later numerical experiments.
For the remainder of this section we assume that only \eqref{allMequiv} is valid.

We have already seen that
 $
 (\t q,\t \mu) := \argmin_{q\in \MM^1_n,  \mu\in\mathcal{P}}\big(\sup_{v\in \HH_n}b_\mu(q,v)/\|v\|_{\HH_\mu}\big)
 $
 agrees with the output of \eqref{eq:stab-greedy} and, on account of Remark \ref{rem-greedy}, of a greedy approximation step
 to the set $\MM^1_n$. Hence, the question of termination of the stabilization loop is equivalent to finding the smallest $j$ for which
 \begin{equation}
 \label{sigmanj}
\sigma_{n,j}:=  \max_{q\in \MM^1_n , \mu\in\mathcal{P}}\big(\inf_{\psi\in \HH_{n}^j}\|q -  R_{ \MM }^{-1}B_\mu^*\psi
\|_{\MM }\big) \leq \delta,
\end{equation}
where $\HH_n^0=\HH_{n-1}$ and $\HH_n^j$ is the   enrichment of $\HH^0_n$ produced by the $j$th stabilization step.
Here we assume that for the preceding pair $(\MM_{n-1},\HH_{n-1})$ 
we have that $\HH_{n-1}= \HH_{n-1}^{\ell_{n-1}}$ satisfies $\sigma_{n-1,\ell_{n-1}}\leq \delta$.
 We wish to see now how $\HH_n^{ {j}}$ evolves from $\HH_{n-1}$. For convenience let $K_\mu := R_{ \MM }^{-1}B_\mu^*$

A straightforward application of the currently available greedy concepts from \cite{BCDDPW,DPW}
is complicated by the fact that the sets $\MM^1_n$ become   ``less compact'' when $n$ grows
and that  the approximating subspaces  {$R_{\MM}^{-1} B_\mu^* Y_n$} depend on $\mu$ through the
application of $B_\mu^*$.  The following discussion is merely to shed some
light on the expected behavior of $\sigma_{n,j}$, in particular, to identify some driving mechanisms,
  while we postpone a more detailed discussion to
forthcoming work.

Our first remarks concern the continuity of the mapping $\mu \mapsto K_\mu$.
To this end, recall that the space $\HH=\bigcap_{\mu\in\mathcal{P}}\HH_\mu$ is endowed with
the norm $\|\cdot\|_\HH$ from \eqref{supnorms} which is here allowed to be stronger than 
the individual norms  $\|\cdot\|_{\HH_\mu}$. 
In view of \eqref{allMequiv} and \eqref{eq:M-equiv}, we have for any $\psi\in \HH$ 
\begin{equation*}
\| K_\mu\psi \|_\MM \leq  C_0 C_M \|K_\mu\psi \|_{\hat\MM_\mu} = C_0 C_M\|\psi\|_{\HH_\mu}\leq C_0 C_M\|\psi\|_{\HH}.
\end{equation*}
 Thus, $K_\mu\in L(\HH,\MM)$ which is equivalent to saying $B_\mu^*\in L(\HH,\MM')$. Now let $B_k^*$ be the
 component of $B_\mu^*$ corresponding to the $k$th bilinear form $b_k(\cdot,\cdot)$ in \eqref{affine} which, by assumption, are smooth.
Obviously, one has  
\begin{equation}
\label{Kcont2}
\|(K_\mu - K_{\mu'})\psi\|_\MM \leq \sum_{k=1}^{m_B} |\Theta_k^b(\mu)-\Theta_k^b(\mu')|\|B_k^*\psi\|_{\MM'} 
\leq C \max_{k=1,\ldots,m_B} |\Theta_k^b(\mu)-\Theta_k^b(\mu')|\|\psi\|_\HH.
\end{equation}
which shows that the mapping $\mathcal{P} \to L(\HH,\MM), \mu\mapsto K_\mu$ is continuous in $\mu$.
By compactness of $\mathcal{P}$, we can find for each $\e >0$ a finite $\e$-net comprised  {of} $N_\e(\mathcal{P})$ centers $\mu_{\e,j}$ such that
for each $\psi\in\HH$ and any $\mu\in \mathcal{P}$ there exists a $j\in \{1,\ldots,N_\e(\mathcal{P})\}$ such that 
\begin{equation}
\label{mu-net}
\|(K_\mu - K_{{\mu_{\e,j}}})\psi\|_\MM \leq \e \|\psi\|_\HH.
\end{equation}
 In order to estimate $\|\psi\|_Y$ we introduce the  constant
\begin{equation}
\label{phiK}
 {C(n,\mathcal{N}):= \max_{\mu\in \mathcal{P}; \; q \in \MnOne}\|K_{\mu}^{-1}q\|_\HH,}
\end{equation}
 which is finite because $K_{\bar\mu}^{-1}\phi_r\in \HH_\truth \subset \HH$. However, a point of concern is
that the quantity
{$C(n, \mathcal{N})$} may depend on the truth space dimension, a point that will be taken up later again.
In particular, we have for any given $n$ and  any $\psi = K_{\bar\mu}^{-1}\phi_r$, $r\leq n$, where $\phi_r$ is any of the orthonormalized reduced basis
functions generated by snapshots from $\mathcal{M}_\MM$,
\begin{equation}
\label{mu-net2}
\|(K_{\bar\mu} - K_{ {\mu_{\e,j}}}) K_{\bar\mu}^{-1}\phi_r\|_\MM \leq \e \|K_{\bar\mu}^{-1}\phi_r\|_\HH\leq \e C(n,\mathcal{N}), \quad \bar\mu\in \mathcal{P}.
\end{equation}
For  $n=1$ the greedy stabilization would determine a sequence $\bar\mu_{1,l}$, $l=1,\ldots,\ell_1$, such that
\begin{equation}
\label{phi1}
\sup_{\mu\in\mathcal{P}}\Big\|\phi_1 - \sum_{l=1}^{\ell_1}c_l(\mu)K_\mu\big(K_{\bar\mu_{1,l}}^{-1}\phi_1)\big)\Big\|_\MM \leq \delta.
\end{equation}
Thus, for $\e \leq \delta/2C(n,\mathcal{N})$, this means that $\ell_1 \leq N_\e(\mathcal{P})$. In fact, as long as \eqref{phi1}
does not hold no two $\bar\mu_{1,l}$ can fall into a single ball of the $\e$-cover of $\mathcal{P}$ and as soon as
every ball contains a $\bar\mu_{1,l}$ \eqref{mu-net2} says that \eqref{phi1} is valid.

It is now easy to display for any given $n$ a space $V_n\subset \HH$ which is $\delta$-proximal for $\MM_n$.
In fact, pick an $\e_n$-net for $\mathcal{P}$ where $\e_n := \delta/(C(n,\mathcal{N})\sqrt{n})$ and let 
\begin{equation*}
V_n := {\rm span}\,\{K_{\mu_{\e_n,l}}^{-1}\phi_k : l=1,\ldots, N_{\e_n}(\mathcal{P}),\, k=1,\ldots,n\}.
\end{equation*}
Hence, for any $\mu\in \mathcal{P}$ there exists a center $\mu_{\e_n,k}$ such that, on account of \eqref{mu-net2}, for every $r=1,\ldots,n$,
\begin{equation*}
\|\phi_r - K_\mu(K_{\mu_{\e_n,k}}^{-1}\phi_r)\|_\MM = \|(K_{\mu_{\e_n,k}}-K_\mu)(K_{\mu_{\e_n,k}}^{-1}\phi_r)\|_\MM \leq \delta/n^{1/2}.
\end{equation*}
Since the $\phi_k$ are $\MM$-orthonormal we obtain for any $q\in \MM^1_n$
\begin{eqnarray*}
\inf_{\psi\in V_n}\|q - K_\mu \psi\|_\MM &\leq & \sum_{r=1}^n |(q,\phi_r)_\MM| \|\phi_r - K_\mu (K_{\mu_{\e_n,k}}^{-1}\phi_r)\|_\MM
\leq \delta \Big(\sum_{k=1}^n |(q,\phi_r)_\MM|^2\Big)^{1/2} \leq \delta.
\end{eqnarray*}
Of course, it is not clear whether the   greedy procedure for building the spaces $\HH_n^j$ would actually produce
a space of similar dimension ${\rm dim}\, V_n \leq n N_{\e_n}(\mathcal{P})=: N_n$. If one did instead a  separate greedy procedure
for each subspace spanned by $\phi_r$ the argument for \eqref{phi1} would say that it terminates after at most $N_n$ steps.
Intuitively, one expects that the actual greedy algorithm terminates earlier since each individual $\phi_r$ has to be
resolved only with accuracy $\delta$, not with accuracy $\e_n = \delta/(n^{1/2}C(n,\mathcal{N}))$ as above.

The perhaps simplest, although grossly pessimistic,  way of rigorously bounding the number of greedy steps providing the spaces $\HH_n^j$ in the stabilization loop,
is to use the above pidgeon hole principle and consider in addition to a $\bar\delta$-net for $\mathcal{P}$ with centers $\mu_{\bar\delta,l}$,
$l=1,\ldots,N_{\bar\delta }(\mathcal{P})$, 
also a $\bar\e$-net for the compact set $\MM^1_n$ with centers $q_i$, $i=1, \dots, N_{\bar\epsilon}(\MnOne)$, where 
$\bar\delta, \bar\e$ will be specified later.
Now suppose that $K_{\bar\mu}^{-1}\bar q$ is the new snapshot added to $\HH_n^{j-1}$ to generate $\HH_n^j$. There exist,
 by construction,  indices 
$l \in \{1, \dots, N_{\bar\delta}(\mathcal{P}) \}$ and $i \in \{ 1, \dots, N_{\bar\epsilon }(\MnOne) \}$ such that $\|\bar{q} - q_i\|_X \le \bar\epsilon $ and $\|(K_{\bar\mu} - K_{\mu_{\bar\delta,l}}) \psi\|_X \le \bar\delta \|\psi\|_Y$.

Then,
for any $(q,\mu)\in \MM_n^1 \times \mathcal{P}$ with $\|q_i-q\|_\MM\leq \bar\e$ and $\|(K_\mu - K_{\mu_{\bar\delta,l}})\psi\|_\MM \leq \bar\delta \|\psi\|_\HH$
where $q_i, \mu_{\bar\delta,l}$ are related to the new snapshot $K_{\bar\mu}^{-1} \bar q$ as above, we obtain
\begin{eqnarray}
\label{second}
  \inf_{\psi \in Y_n^{j}} \| q- K_\mu\psi\|_\MM &\leq &
\|q - K_\mu K_{\bar\mu}^{-1}\bar q\|_\MM  \leq \|q- q_i\|_\MM +\|q_i-\bar q\|_\MM +\|(K_{\bar\mu}-K_\mu)K_{\bar\mu}^{-1}\bar q\|_\MM\nonumber\\
& \le & 2(\bar\e  + \bar\delta C(n,\mathcal{N})).
 \end{eqnarray}
Now  {choose $\bar\e$ and $\bar\delta$ such that $2(\bar\e  + \bar\delta C(n,\mathcal{N}))=\delta$, e.g. by taking} $\bar\e = \delta/4$, $\bar\delta =\delta/(4C(n,\mathcal{N}))$.
It follows from \eqref{second} that a new snapshot $K_{\bar\mu}^{-1}\bar q$ can only satisfy $\inf_{\psi\in \HH_n^{j-1}}\|\bar q - K_{\bar\mu}\psi\|_\MM >\delta$
if it falls into a cover element $B_{q_i,\mu_{\bar\delta,l}}(\bar\e,\bar\delta):= \{(q,\mu): \|q_i-q\|_\MM\leq \bar\e,\, \|(K_\mu - K_{\mu_{\bar\delta,l}})\psi\|_\MM \leq \bar\delta \|\psi\|_\HH\}$ that does not contain any previous snapshot yet. This can happen at most {$N_{\bar\delta }(\mathcal{P})N_{\bar\e}(\MnOne)$}
times which therefore bounds the number of possible greedy steps in the stabilization loop.

As mentioned before, this bound is very pessimistic. In fact,  since $\MnOne$ is isometrically isomorphic to a unit sphere in $\ell_2^n$ 
the covering numbers $N_{\bar\e}(\MM_n^1)$ increase like $(12/\bar\e)^n$, {see \cite[Chapter 13]{LGM}}. 
The numbers $N_{\bar\delta}(\mathcal{P})$ instead depend only on the fixed dimension of the parameter set $\mathcal{P}$ and the smoothness
of the parameter functions $\Theta_k^b(\mu)$.

One way to ameliorate the strong dependence of the $N_{\bar\e}(\MM_n^1)$ on $n$ is to
 relate the problem to a greedy approximation to
 a compact set that is {\em independent} of $n$.  
 To this end, recall the solution set $\mathcal{M} = \mathcal{M}_{\MM}\times \{0\}$, see \eqref{sol-M},
 which under the present assumptions is compact in $\MM \times \{0\}$, independent of the truth spaces.
 As detailed later the spaces $\MM_n$ are generated by a (weak) greedy algorithm.
By compactness, the (weak) greedy errors
 \begin{equation}
 \label{distM}
\sigma_n:=\sigma_{n}( \mathcal{M}_{\MM},\MM_n):=   \dist\,(\mathcal{M}_{\MM},\MM_n)_{\MM} \to 0,\quad n\to \infty,
 \end{equation}
 tend to zero at a rate that is independent of the truth dimension.
  A repeated greedy approximation generates an $\MM$-orthonormal system $\{\phi_j\}_{j=1}^\infty\subset \MM$.
 Let $\MM^\circ$ denote the closure of the span of $\{\phi_j\}_{j=1}^\infty$, i.e.
 \begin{equation*}
\MM^\circ := \bigg\{q\in \MM : \sum_{j\in \N}(q,\phi_j)_{\MM}^2 <\infty \bigg\}.
 \end{equation*}
Let
\begin{align*}
\mathcal{F} & := \left\{q\in \MM^\circ : 
 \, |q|_* <\infty\right\}, & |q|_* & := \sup_{n\in\N}\sigma_n^{-1}\Bigg(\sum_{j=n+1}^\infty (q,\phi_j)_{\MM}^2\Bigg)^{1/2}.
\end{align*}
Obviously, $\mathcal{B} := \{q\in \mathcal{F} : \max\,\{\|q\|_{\MM}, |q|_*\}\leq 1\}$ is a compact subset of $\MM$ and 
by construction
\begin{equation}
\label{distq}
{\rm dist}\,(q,\MM_n)_{\MM} \leq \sigma_{n}|q|_*, \quad q\in \mathcal{F},\, n\in \N.
\end{equation}
Moreover, the greedy errors for $\mathcal{B}$ are comparable to the greedy errors for $\mathcal{M}_{\MM}$.
In particular,
\begin{equation}
\label{Kol1}
\dist\,(\mathcal{B},\MM_n)_\MM  {\leq\dist\,(\mathcal{M}_{\MM},\MM_n)_{\MM}} 
 ,\quad n\in \N.
\end{equation}
Furthermore, 
\begin{equation*}
\MM^1_n:= \{ q\in \MM_n: \| q\|_{\hat\MM_\mu}\leq 1,\, \mu\in\mathcal{P}\}\subseteq \{q\in \MM_n : \|q\|_{\MM}\leq c_M^{-1}\}  \subset \mathcal{F},
\end{equation*} 
since for $q\in \MM_n$
\begin{equation}
\label{Mnsigma}
|q|_* = \max_{j\leq n}\sigma_{j}^{-1}\Big(\sum_{k=j+1}^n(q,\phi_k)_{\MM}^2\Big)^{1/2} \leq \sigma_{n}^{-1}\|q\|_{\MM}
\leq c_M^{-1}\sigma_n^{-1}.
\end{equation}
Therefore, recalling \eqref{sigmanj}, we conclude that
\begin{eqnarray}
\label{sigmanj2}
\sigma_{n,j} &\leq &
 \max_{q\in \MM_n^1\cap \mathcal{B} , \mu\in\mathcal{P}}\big(\inf_{\psi\in \HH^j_{n}}\|q -  R_{\hat\MM_\mu}^{-1}B_\mu^*\psi
\|_{\hat\MM_\mu}\big)(c_M\sigma_n)^{-1}\nonumber\\
&\leq &  \max_{q\in \MM_n^1\cap \mathcal{B} , \mu\in\mathcal{P}}\big(\inf_{\psi\in \HH_{n}^j}\|q -  R_{\MM}^{-1}B_\mu^*\psi
\|_{\MM}\big) C_M(c_M\sigma_n)^{-1},
\end{eqnarray}
where $c_M, C_M$ are the constants from \eqref{Mu2}.
Hence, termination of the stabilization loop reduces to analyzing the the necessary number of
steps needed to enrich $\HH_{n-1}^{\ell_{n-1}}=\HH_n^0$ until
$\max_{q\in \MM_n^1\cap \mathcal{B} , \mu\in\mathcal{P}}\big(\inf_{\psi\in \HH_{n}^j}\|q -  R_{\MM}^{-1}B_\mu^*\psi
\|_{\MM}\big)\leq c_M\sigma_n/C_M$. Clearly, $N_\e( \MM_n^1\cap \mathcal{B})\leq N_\e(\mathcal{B})$ where
$\mathcal{B}$ is now a fixed compact set. We can now apply the same reasoning as above with 
$\MM_n^1$, $\bar\e, \bar\delta$ replaced by $\mathcal{B}$, $\bar\e c_M\sigma_n/C_M, \bar\delta c_M\sigma_n/C_M$,
respectively. This leads to the alternative bound $N_{\delta\sigma_n c_M/(4 C_M)}(\mathcal{B}) N_{\delta \sigma_n c_M/(C_M C(n,\mathcal{N}))}(\mathcal{P})$
for the maximal number of greedy steps. Note that in this case $C(n,\truth)$ can be replaced by
\begin{equation*}
C(\mathcal{B},\truth):= \max_{\mu\in \mathcal{P}, q\in \mathcal{B}}\|K_\mu^{-1}q\|_\HH.
\end{equation*}
Since every $q\in \MM_n^1$ or $q\in \mathcal{B}$ is a   linear combination of   snapshots $B_{\mu_l}^{-1}f = {p}(\mu_l)$
and since the enrichments of the test spaces $\HH_n^j$ are of the form $B_{\t\mu_j}^{-*}R_\MM q_j$, $q_j\in \MM_n^1$, they are
linear combinations of elements of the form $B^{-*}_{\mu'}R_\MM B_{\mu''}^{-1}f$.
Since the operators $B^{-*}_{\mu'}R_\MM B_{\mu''}^{-1}$ at least preserve the regularity of $f$ the
quantities $K_\mu^{-1}q$, $q\in \MM_n^1$ ($q\in \mathcal{B}$), where 
now the inversion is understood  in the infinite dimensional spaces,  possess the required additional regularity in $\HH$ 
when $f$ is sufficiently regular, see the discussion of the transport problem in Section \ref{sect:wcs}.

\begin{summary}
\label{greed-summ}
We can now summarize the above findings as follows:
\begin{enumerate}
\item
If the constants $C(n,\truth)$ or $C(\mathcal{B},\truth)$ are uniformly bounded independently of the choice of the truth spaces the stabilization loops terminate after a number of steps that is independent of the truth spaces. Their dependence on $n$ can be bounded in terms of the metric entropy of $\MM_n^1$ or  
the metric entropies of $\mathcal{B}$ and $\mathcal{P}$, coupled in the latter case with the greedy errors $\sigma_n$ for $\mathcal{M}_\MM$.
\item
The constants $C(n,\truth)$,  $C(\mathcal{B},\truth)$ remain independent of the truth spaces when $f$ is sufficiently regular.
\end{enumerate}
\end{summary}

\section{A Double Greedy Scheme}\label{sec:DG}
We shall discuss now a greedy strategy for constructing reduced spaces $\MM_n, \HH_n$ for the
saddle point problem \eqref{system-full} which is a weak formulation of \eqref{eq:op-eq}.

The basic outline of such a strategy looks as follows:
\begin{itemize}
\item
  {\em Stabilization:} Given a pair $\HH_n, \MM_n$, enrich $\HH_n$ until $\beta_{\HH_n, \MM_n}(\mu)\geq \zeta 
  \beta_{\mathcal{N}}$, $\mu\in\mathcal{P}$, where $\beta_\mathcal{N}$ is given by \eqref{eq:inf-sup} and $\beta_{\HH_n, \MM_n}(\mu)$, $\mu\in\mathcal{P}$, is the inf-sup constant \eqref{eq:unsym-infsup-VW} for the reduced spaces $\MM_n$ and $\HH_n$.
\item
{\em Approximation update:} In view of the best approximation property \eqref{BAP1}, \eqref{BAP2}, we then
 improve the accuracy of the reduced spaces with the aid of a greedy step.
\end{itemize}
That last greedy step, in turn, requires a tight  residual based surrogate as detailed next.

\subsection{Tight Surrogates} 
Suppose now that the pair of spaces $\MM_n\subset\MM$, $\HH_n\subset \HH$ satisfy  the $\delta$-proximality condition \eqref{delta-prox} for some $\delta \in (0,1)$
and abbreviate the corresponding solutions of \eqref{system-disc} as $p_n(\mu) := p_{\MM_n,\HH_n}(\mu)$, $u_n(\mu) := u_{\HH_n,\MM_n}(\mu)$.
By Propositions  \ref{prop:stab2},  \ref{propstab}, 
the definition \eqref{Mu2} of the $\hat\MM_\mu$-norm says then that
\begin{equation*}
  \|p(\mu) - p_{n}(\mu)\|_{\hat\MM_\mu} = \|f - \bo_\mu  p_n(\mu)\|_{\oldHs'}, \quad \mu\in \mathcal{P},
\end{equation*}
i.e., the residual based surrogate
\begin{equation}
\label{R*}
R(\mu, \MM_n \times \HH_n) := \|f - \bo_\mu  p_{n}(\mu)\|_{\oldHs'} 
\end{equation}
is in this case almost ideal. In fact, combining \eqref{R*} with \eqref{BAP1} yields
 \begin{equation}
\label{tightsurr1}
\inf_{q\in W}\|p(\mu) -q\|_{\hat\MM_\mu} \leq R(\mu, \MM_n \times \HH_n) \leq  \frac{1}{1-\delta}\inf_{q\in \MM_n}\|p(\mu) -q\|_{\hat\MM_\mu}.
\end{equation}
Hence, \eqref{surr2} holds with $c_R=1-\delta, C_R= 1$.

\subsubsection{Reduction to Truth-Riesz Maps}

Of course, the dual norm $\|\cdot\|_{\HH_\mu'}$ and hence $R(\mu, \MM_n \times \HH_n)$ cannot be computed exactly.
Instead, defining 
\begin{equation}
\label{truthres}
\|\cdot\|_{\HH_\truth'} := \|P_{\HH_\mu,\HH_\truth} R_{\HH_\mu}^{-1} \cdot \|_{\HH_\mu} =  \sup_{v \in \HH_\truth'} \frac{\langle \cdot, v \rangle}{\|v\|_{\HH_\mu}}, 
\end{equation}
we  consider
the following candidate
\begin{equation}
\label{tightsurr2}
R_n(\mu):= \|f - B_\mu p_{n}(\mu)\|_{\HH_\mathcal{N}'},
\end{equation}
where we continue to assume that  $\HH_\mathcal{N} \in \mathcal{V}(\MM_\mathcal{N}, \delta_\mathcal{N})$, i.e., the truth spaces $\MM_\mathcal{N}, \HH_\mathcal{N}$ comply with Remark \ref{rem:XYchoice}.
 Then, by \eqref{res-truth-2} and \eqref{X-bound}, we conclude that
\begin{eqnarray}
\label{tol}
\|p(\mu)- p_n(\mu)\|_{\hat\MM_\mu} &\leq & (1-\delta_\mathcal{N}^2)^{-1/2}\|P_{\HH_\mu,\HH_\mathcal{N}}(R_{\HH_\mu}^{-1}(f- B_\mu p_n(\mu))\|_{\HH_\mu}
=  (1-\delta_\mathcal{N}^2)^{-1/2} R_n(\mu)\nonumber\\
& \leq & (1-\delta_\mathcal{N}^2)^{-1/2} \| f- B_\mu p_n(\mu)\|_{\HH_\mu'} \leq (1-\delta_\mathcal{N}^2)^{-1/2}(1-\delta)^{-1} \| p(\mu)-P_{\hat\MM_\mu,\MM_n}p(\mu)\|_{\hat\MM_\mu}.
\end{eqnarray}
 This immediately implies the following fact.
\begin{prop}
\label{rem:tight}
Under the above assumptions on the truth spaces the surrogate $R_n(\mu)$ given by \eqref{tightsurr2} is tight
with condition
\begin{equation}
\label{surr-cond}
\kappa(R_n) \leq \frac{1}{(1-\delta_\mathcal{N}^2)^{1/2}(1-\delta)},
\end{equation}
which, in principle, can be driven as close to one as one wishes, at a computational expense
caused by a correspondingly large truth space $\HH_\mathcal{N}$ and a possibly larger number of stabilization steps.
\end{prop}
The equivalence 
\begin{equation}
  \|f - \bo_\mu q\|_{\HH_\truth'} \sim \|f - \bo_\mu q\|_{\HH_\mu'},\quad q\in \MM_\mathcal{N},
  \label{eq:disc-res}
\end{equation}
which is nothing but a reformulation of \eqref{res-est-W} for  $W=\MM_\mathcal{N}, \HH_\mathcal{N} \in \mathcal{V}(\MM_\mathcal{N},\delta_\mathcal{N})$,
says that the $\|\cdot\|_{\HH_\truth'}$-norm  yields still a meaningful error estimate even in case the truth spaces are not rich enough to resolve all features of the infinite dimensional exact solution which will be seen below in the experiments.

The above findings can be summarized as follows.
\begin{prop}
\label{prop:feasible}
If both, \eqref{allHequiv} holds, then $R_n(\mu)$ defined by \eqref{tightsurr2}
is feasible.
\end{prop}
\begin{proof} Under the given assumptions the norms $\|\cdot\|_{\HH_\mu}$ can be replaced by a uniformly equivalent
 equivalent reference norm $\|\cdot\|_{\HH}$ so that the Riesz map $R_{\HH}$ is independent of $\mu\in\mathcal{P}$.
 Hence, $R_n(\mu)$ can, in view of \eqref{res-truth-2},  be  
   efficiently evaluated by a standard offline/online decomposition, see e.g. \cite{RHP}. 
 \end{proof}
 
\subsubsection{Iterative Tightening}\label{sec:it-tight}
Recall that in the pure transport problem \eqref{allHequiv} does not hold, see Remark \ref{remYdiffer}.
Hence, the surrogate $R_n(\mu)$ from \eqref{tightsurr2} is no longer feasible in the strict sense.
 Instead a feasible variant would be
\begin{equation}
\label{tightsurr3}
R'_n(\mu):= \| f -B_\mu p_n(\mu)\|_{\HH_n'} = \|P_{\HH_\mu,\HH_n}R_{\HH_\mu}^{-1}( f -B_\mu p_n(\mu))\|_{\HH_\mu},
\end{equation}
 where the dual norm is now induced by the {\em reduced} space $\HH_n$ instead of the truth space $\HH_\mathcal{N}$.
While the $\delta$-proximality of $\HH_n$ for $\MM_n$ (see \eqref{delta-prox}) does ensure the equivalence
$\|B_\mu q\|_{\HH_\mu'}\sim \|B_\mu q\|_{\HH_n'}$, $q\in \MM_n$, (with constants close to one, depending on $\delta$)
the analog is not clear for  $\| f -B_\mu q\|_{\HH_n'}$ since generally $f\notin B_\mu(\MM_n)$.

However, Remark \ref{WMX} immediately tells us at least a criterion for the validity of the desired residual equivalence,
namely with the aid of a somewhat strengthened $\delta$-proximality.
\begin{rem}
\label{rem:MnX}
Assume that for some $\bar\delta \in (0,1)$ one chooses $\HH_n \in \mathcal{V}(\MM_n,\bar\delta)$ so that
\begin{equation}
\inf_{v \in \HH_n} \| p - R_{\hat\MM_\mu}^{-1}B^*_\mu v \|_{\MM_\mu}\leq \bar\delta \|p\|_{\MM_\mu},\quad \forall\,\, p\in \mathcal{M}_X + \MM_n.
\label{eq:estimator-stab-3}
\end{equation}
Then
\begin{equation}
\label{res-est-Mn}
(1-\bar\delta^2)^{1/2} \| f- B_\mu q\|_{\HH_\mu'} \leq \|  f- B_\mu q\|_{\HH_n'} \leq \| f- B_\mu q\|_{\HH_\mu'}, 
\quad q\in \MM_n,\, \mu\in\mathcal{P}.
\end{equation}
and we have $\kappa(R_n')\leq (1-\bar\delta^2)^{-1/2}(1-\delta)^{-1}$.
\end{rem}

Note that we could replace $ \mathcal{M}_X $ in \eqref{eq:estimator-stab-3} by its truth approximation $\mathcal{M}_{\MM,\mathcal{N}}$ since, in
view of \eqref{eq:disc-res}, it suffices to establish $\|  f- B_\mu q\|_{\HH_n'} \sim \|  f- B_\mu q\|_{\HH_\mathcal{N}'}$.
But the main practical issue remains how to find $\HH_n$ satisfying \eqref{eq:estimator-stab-3} at affordable cost. 

To this end,   we propose a systematic   way of successively  substantiating tightness of error estimators at the expense 
of  an additional computational  effort   in the offline phase. We refer to this process as 
  {\em iterative tightening}. The idea is that once a reduced space provides sufficiently accurate approximations to $\mathcal{M}_\MM$,
  condition \eqref{eq:estimator-stab-3} becomes easier to fulfill.
To make use of this observation, assume we have a second pair of reduced spaces $\bar{\MM} \subset \MM_\mathcal{N}$ and $\bar{\HH} \subset \HH_\mathcal{N}$. We now describe how such spaces can give rise to tight surrogates and later discuss their construction. 

\begin{lemma}
\label{lem:tight}
Assume that the pair $\MM_n + \bar{\MM}$ and $\bar{\HH}$   satisfies the (standard) $\delta$-proximality condition \eqref{delta-prox} and that the approximation of $p(\mu)$ from $\MM_n + \bar{\MM}$ is superior to the approximation from $\MM_n$ alone, i.e., one has for some $0 \le \xi < 1$
\begin{equation}
  \|p(\mu) - \bar{p}^n(\mu)\|_{\MM_\mu} \le \xi \|p(\mu) - p^n(\mu)\|_{\MM_\mu},\quad \mu\in \mathcal{P},
  \label{eq:rel-error}
\end{equation} 
where $p^n(\mu)$ and $\bar{p}^n(\mu)$ are the respective best approximations to $p(\mu)$ from  $\MM_n$ and $\MM_n + \bar{\MM}$. Then
\begin{equation}
\label{bardelta}
  \inf_{\bar{v} \in \bar{\HH}} \| p - R_{\MM_\mu}^{-1}B^*_\mu \bar{v} \|_{\MM_\mu}\leq \bar\delta \|p\|_{\MM_\mu},\quad \forall\,\, p\in \mathcal{M}_X + \MM_n,
  \, \mu\in \mathcal{P},
\end{equation}
where $\bar\delta := (1 + \delta) \xi + \delta$. Hence, for $\xi,\delta$ sufficiently small,
 the surrogate $R_n'(\mu)$ from \eqref{tightsurr3} is tight with a condition given by \eqref{surr-cond} with $\delta_\mathcal{N}$ replaced by $\bar\delta$.
\end{lemma}

 Note that we use the space $\MM_n + \bar{\MM}$ as opposed to $\bar{\MM}$ alone because for the latter space the condition \eqref{eq:rel-error} would imply that $\MM_n \subset \bar{\MM}$ if both spaces are constructed from snapshots, which would be too restrictive for the application below.

\begin{proof}
Then, for each deviation $p(\mu) - p$, $p \in \MM_n$, we obtain
\begin{eqnarray}
  \inf_{\bar{v} \in \bar{\HH}} \| p(\mu) - p - R_{\MM_\mu}^{-1}B^*_\mu \bar{v} \|_{\MM_\mu} & \leq &
  \|p(\mu) - \bar{p}^n(\mu)\|_{\MM_\mu} 
 + \inf_{\bar{v} \in \bar{\HH}} \| \bar{p}^n(\mu) - p - R_{\MM_\mu}^{-1} B^*_\mu \bar{v} \|_{\MM_\mu}\nonumber \\
  & \leq & \|p(\mu) - \bar{p}^n(\mu)\|_{\MM_\mu} + \delta \|\bar{p}^n(\mu) - p\|_{\MM_\mu} 
 \leq  \big( (1+\delta) \xi + \delta \big) \|p(\mu) - p\|_{\MM_\mu}
\label{eq:estimator-stab-2}
\end{eqnarray}
Thus for $ (1 + \delta) \xi + \delta$ sufficiently small the extended $\delta$-proximality condition \eqref{eq:estimator-stab-3} is satisfied for the trial space $\MM_n$ and test space $\bar{\HH}$. Thus, Remark \ref{rem:MnX} applies  which says that the surrogate $R_n'(\mu)$ from \eqref{tightsurr3} is tight with
the claimed condition.\end{proof}

We shall describe ways of constricting the spaces $\bar{\MM}$ and $\bar{\HH}$ later in Section \ref{sect:appl}.

\subsection{Approximation Update}\label{secAU}
Either scheme \textsc{Update-Inf-Sup} or \textsc{Update-$\delta$} outputs a pair $\MM_n, \HH_n$ that is uniformly inf-sup stable, i.e.,
the corresponding inf-sup constant is uniformly bounded away from zero $\beta_{\HH_n,\MM_n}(\mu)\geq \zeta \beta_\mathcal{N}$,
$\mu\in \mathcal{P}$. 
By Proposition \ref{rem:tight}, the surrogate $R_n(\mu)$, defined by \eqref{tightsurr2}, is tight with a condition 
controlled by the $\delta$-proximality parameters.
 The feasibility of this surrogate depends 
on the way how the spaces $\HH_\mu$ depend on the parameter $\mu\in\mathcal{P}$, see Proposition \ref{prop:feasible}.
In applications, an infeasible surrogate is replaced by $R'_n(\mu)$ from \eqref{tightsurr3} combined with
iterative tightening.

This suggest the following {\em outer} greedy step \textsc{update-approximation}, defined in Algorithm \ref{eq:system-alg}, which  aims at improving on the accuracy of
the reduced model. 

\begin{algorithm}[htb]
  \caption{Update to improve the approximation quality.}
  \label{alg:update-approx}
  \begin{algorithmic}[1]
    \Require{Finite dimensional spaces $\HH_n \subset \HH$, $\MM_n \subset \MM$ that satisfy the inf-sup condition 
      \begin{equation}
        \inf_{\mu\in \mathcal{P}} \inf_{q\in \MM_n}\sup_{v\in \bar \HH_n}\frac{b_\mu(q, v)}{\|q\|_{\MM_\mu}}\geq \zeta\beta_\mathcal{N}
        \label{eq:inf-sup-alg}
      \end{equation}
      for some $0 < \zeta \leq 1$.}
    \Function{Update-approximation}{$\HH_n$, $\MM_n$}
    \State Select a sufficiently large sample $\mathcal{S} \subset \mathcal{P}$.
    \State compute 
      \begin{equation*}
        \hat \mu := \argmax_{\mu\in\mathcal{P}}\, \sur
      \end{equation*}
    \State Compute the solution $[\hat u,\hat p]\in \HH_{\hat \mu}\times \MM_{\hat\mu}$ of  
    \begin{equation}
    \label{eq:system-alg}
    \begin{array}{lcc}
    (u,v)_{\HH_{\hat{\mu}}} + b_{\hat{\mu}}(p, v) & = & \langle f,v\rangle,\quad v\in \HH_{\hat{\mu}},\\
    b_{\hat{\mu}}(q, u) &= & \langle g,q\rangle, \quad q\in \MM_{\hat{\mu}}.
    \end{array}
    \end{equation}
    \State \label{state:update} Set
      \[
      {\rm span}\, \{\MM_n,\hat p\} \to \MM_n 
      \]
      \State If \eqref{allMequiv} holds orthonormalize $\{\phi_1,\ldots,\phi_{n-1}, \hat p_n\}$ in $\MM$.
      \State \Return{$\HH_{n}$, $\MM_{n}$}
    \EndFunction
  \end{algorithmic}
\end{algorithm}

\subsection{Putting Things together}\label{sectogether}
The overall double-greedy method (see Algorithm \textsc{DG-1} below) for computing reduced spaces for \eqref{varprob} consists now in combining
the inner greedy stabilization loop with the outer greedy approximation step for the saddle point formulation \eqref{system-full}.

\begin{algorithm}[htb]
  \caption{Double greedy scheme}
  \label{alg:double-greedy}
  \begin{algorithmic}[1]
    \Function{DG-1}{}
  \State Initialize $\HH_1 = \{0\}$ and $\MM_1 = {\rm span}\,\{u(\mu_1)\}$ for an arbitrary $\mu_1 \in \mathcal{P}$.
  \State $\HH_1 \leftarrow \text{\textsc{Update-inf-sup}}(\HH_1, \MM_1)$.
  \While{$\max_{\mu\in\mathcal{P}}\, R_n(\mu) 
   > \tau$}
    \State $
    \MM_n \leftarrow \text{\textsc{Update-Approximation}}(\HH_n, \MM_n)$.
    \State $\HH_n \leftarrow \text{\textsc{Update-inf-sup}}(\HH_n, \MM_n)$.
  \EndWhile
  \State \Return{$\HH_n$, $\MM_n$}
  \EndFunction
  \end{algorithmic}
\end{algorithm}

To  analyze of algorithm \textsc{DG-1} recall the solution manifold $\mathcal{M} = \mathcal{M}_{\MM}\times \{0\}$ from
\eqref{sol-M}.
Since the inner stabilization loops ensure, by Proposition \ref{prop:stab2}
and \eqref{tightsurr1},  tightness of the surrogates, we can invoke Theorem \ref{thrates}.
The above findings can now be summarized as follows.
\begin{theorem}
\label{thm-M}
 Assume that   \eqref{allMequiv}  holds.
Let $p_n(\mu):= p_{\MM_n,\HH_n}(\mu)$, $u_n(\mu):= u_{\HH_n,\MM_n}(\mu)$ denote the solution components of 
\eqref{system-disc}
for $W=\MM_n$, $V=\HH_n$, 
were $[\HH_n,\MM_n]$ are the reduced spaces produced by algorithm \textsc{DG-1} using the surrogate \eqref{tightsurr2}.
Let
\begin{equation*}
\sigma_n(\mathcal{M}_{\MM}):= \sup_{\mu\in \mathcal{P}} 
\| p(\mu)- p_n(\mu)\|_{\hat \MM_\mu},
\quad d_n(\mathcal{M}_{\MM}):=  \inf_{{\rm dim}(Z_n)=n}\Big(  \dist(\mathcal{M}_{\MM},Z_n)_{\MM }\Big).
\end{equation*}
(a)   Then,
if $d_n(\mathcal{M}_{\MM}) =O(n^{-\alpha})$, for some
$\alpha >0$ or if $d_n(\mathcal{M}_{\MM}) =O(e^{-cn^\alpha})$, for some
$c, \alpha >0$, one has 
\begin{equation}
\label{rates2}
\sigma_n(\mathcal{M}_{\MM})= O(n^{-\alpha}),\quad \sigma_n(\mathcal{M})= O(e^{-\t cn^{ {\alpha} }}),\quad n\to \infty,
\end{equation}
respectively, with constants depending on the   parameters  $\delta_\mathcal{N}, \delta, \zeta$ in \textsc{Update-$\delta$} or
\textsc{Update-Inf-Sup}, and on the constants in 
\eqref{allMequiv}, 
\eqref{eq:M-equiv}. Moreover,   \eqref{rates2} remains valid for $\sigma_n(\mathcal{M}_{\MM})$ replaced by
\begin{equation}
\label{greedyerror-2}
\hat \sigma_{n}(\mathcal{M}_{\MM}) := \sup_{\mu\in\mathcal{P}} \big\{\|p(\mu)- p_n(\mu)\|_{\MM}+ \|u(\mu)- u_n(\mu)\|_{\HH_\mu}\big\}.
\end{equation}
(b)
Assume that  both \eqref{allHequiv} and \eqref{allMequiv} hold. 
Then, the assertion (a) holds where in addition ${\rm dim}\,(\HH_n\times \MM_n)\leq (1+m_B)n$, $n\in \N$. 
All bounds remain valid up to the tolerance ${\rm tol}^*$ when all computations are carried
out within this accuracy. Moreover, the surrogate \eqref{tightsurr2} in algorithm \textsc{DG-1} is feasible.
 \end{theorem}
\begin{proof} The output $[\HH_n,\MM_n]$ of step 6 in \textsc{DG-1} is uniformly inf-sup stable so that the surrogate
\eqref{tightsurr2} used in step 5 is uniformly tight, with a condition depending on the stabilization thresholds $\delta, \delta_\mathcal{N}$. 
Concerning $\hat \sigma_{n}(\mathcal{M}_{\MM})$ we use \eqref{BAP2}. By 
\eqref{allMequiv} 
the surrogates remain uniformly tight for the reference norm $\|\cdot\|_{\MM}$. 
Hence, Theorem \ref{thrates} applies. The rest of the assertion follows
from Proposition \ref{propmB} and Proposition \ref{prop:feasible}.\end{proof}

In general, under the assumption (a), the well conditioned surrogate \eqref{tightsurr2} is not feasible. Employing the feasible surrogate 
\eqref{tightsurr3} instead, requires, in order to gurarantee rate optimality, an additional iterative tightening as described in Section \ref{sec:it-tight} and later in connection with numerical experiments.
Note also that under the  assumption (a)   ${\rm dim}\,\HH_n$   could be significantly larger than $n$, see the discussion in Section \ref{sec:greedy}.   In the case of uniformly equivalent norms, i.e., when both conditions 
\eqref{allHequiv} and \eqref{allMequiv} hold, the  dimension of the stabilizing spaces $\HH_n$ remains proportional to the dimension
of the reduced primal space.

We conclude this section with a remark on the online evaluation of $p_n(\mu)$. Recall that the corresponding
component $u_n(\mu)\in \HH_n$ is only an auxiliary variable tending to zero.
\begin{rem}
\label{remaltcomp} 
Assume that \eqref{allHequiv} and \eqref{allMequiv} hold, i.e. the spaces $\MM_\mu$ and $\HH_\mu$ can be choose parameter independent. Instead of solving for a given $\mu$ the 
saddle point problem \eqref{system-disc} for $W=\MM_n$, $V=\HH_n$, whose dimension is $n+m(n)$,
 one can compute in the offline phase the {\em test basis functions} $\psi_{k,j}$, $j=1,\ldots n$, $k=1,\ldots,m_B$
 \begin{equation}
 \label{opttestbasis}
 (\psi_{k,j},v)_{\HH}= b_k(\phi_j,v),\quad v\in \HH_n,\,\, j=1,\ldots, n,
 \end{equation}
 where $b_k$ are the components of the affine decomposition \eqref{affine}. Then, defining
\begin{equation}
\label{psinmu}
\psi^n_j(\mu):= \sum_{k=1}^{m_B}\Theta^b_k(\mu)\psi_{k,j},
\end{equation}
on account of Proposition \ref{propPG}, for each $\mu \in \mathcal{P}$, the solution $p_n(\mu) = p_{\MM_n,\HH_n}(\mu)$ of the saddle point problem \eqref{system-disc}, also solves 
the Petrov-Galerkin problem
\begin{equation}
\label{PGn}
b_\mu(p_n(\mu),\psi^n_j(\mu))= \langle f_\mathcal{N},\psi^n_j(\mu)\rangle,\quad j=1,\ldots, n.
\end{equation}
Hence the online complexity is indeed determined by the size $n$ of the trial basis.
\end{rem}
 
\section{Application to the Model Problems}\label{sect:appl}
\subsection{Singularly Perturbed Convection-Diffu\-sion Problems}
We refer to the setting in Section \ref{ssect:con-diff} and consider  {the convection-diffusion} problem \eqref{eq:cd}
for large Peclet numbers.

To this end, we shall briefly discuss two scenarios concerning the truth spaces, namely (a) boundary layers are to be
resolved completely by the truth spaces, and (b) due to a possibly very small diffusion, even the truth spaces 
are not required to resolve boundary layers.

In case (a) solutions in  the truth spaces could be obtained by simple standard Galerkin discretizations and
a modified variational formulation according to \eqref{Mu2} is only needed for the computation
of reduced basis functions which then also resolve  boundary layers well. 

In this example we prescribe  
  the space $\oldHs$ and adjust $\MM_\mu$ according to \eqref{Mu2}. We first decompose $\bo_\mu$ into its symmetric and skew-symmetric parts:
\begin{equation*}
s_\mu(u,v)  := \frac{1}{2} \dualp{\bo_\mu  u, v} + \dualp{\bo_\mu  v, u},\quad
k_\mu(u,v)  := \frac{1}{2} \dualp{\bo_\mu  u, v} - \dualp{\bo_\mu  v, u},
\end{equation*}
and define 
\begin{equation}
\label{norms}
\|v\|_{\oldHs}^2 := s_\mu(v, v) =   \epsilon |v|_{H^1(\Omega)}^2 + \left\|\left(c - \frac{1}{2} \vdiv b(\mu)\right)^{1/2} v \right\|_{L_2(\Omega)}^2,
\end{equation}
see \cite{Verfurth05,Sangalli05,cdw} for details. $\|\cdot\|_{\HH_\mu}$ is then equivalent to the standard $H^1(\Omega)$-norm with constants depending on the diffusion $\epsilon$. This works perfectly when the discretization (adaptive or not) resolves
the boundary layers.
However,  when layers are not resolved, 
although stable, the scheme \eqref{system-disc} would give rise to unpleasant
numerical artifacts, due to the nature of the involved norms, see the detailed discussion in \cite{cdw, W}.

Therefore, we briefly recall next an alternative variational formulation of \eqref{eq:condiff} avoiding the numerical artifacts, regardless of choosing sufficiently large truth spaces that fully resolve boundary layers or not. 
 In essence, in case the finite element truth space does not resolve boundary layers this scheme behaves like a solver of the corresponding transport problem for $\epsilon = 0$ which is, however, ill-posed when insisting on zero boundary conditions on
 all of $\partial\Omega$.  
 We resort to a remedy proposed in \cite{cdw,W}.  {We retain the construction of the norms} but modify the outflow boundary condition. Instead of building them into the trial space, we impose them only weakly. To this end, let
\begin{equation*}
  \Gamma_+(\mu) := \{ x \in \partial \Omega: \, n(x) \cdot b(\mu, x) > 0 \} 
\end{equation*}
be the outflow boundary where $n(x)$ is the outward unit normal at $x$. Now, we take 
\begin{equation}
\label{barM}
  \bar{\MM}_\mu := \{ q \in H^1(\Omega): \, q|_{ \Gamma_-(\mu)} = 0 \}
\end{equation}
as a set with norm defined below. Here and in the following, restrictions to the boundary are implicitly considered in a trace sense. Thus, zero boundary conditions are only built into $\bar{\MM}_\mu$ on part of the boundary. To find a weak form of the boundary conditions at the outflow boundary recall from \eqref{err-res2} the connection of \eqref{Mu2}
with the optimization 
  problem   
\begin{equation}
\label{eq:opt}
  \|f - \bo_\mu  \bar p\|_{\oldHs'}^2 \to \min
\end{equation}
where $\bar p$ belongs  now to the larger space $\bar{\MM}_\mu$. So far we have not changed $\oldHs$ which is still $H^1_0(\Omega)$ endowed with the norm \eqref{norms}. Due to the missing outflow boundary conditions, $\bo_\mu $ has a nontrivial kernel so that the optimization problem is not uniquely solvable. One simple remedy is to add the outflow boundary condition as a penalty term:
\begin{equation}
  \|f - \bo_\mu  \bar p\|_{\oldHs'}^2 + \omega \|\bar p\|_{\outn(\mu)}^2 \to \min,
  \label{eq:opt+}
\end{equation}
where $\|\cdot\|_{\outn(\mu)}$ is a norm for $H^{1/2}_{00}(\Gamma_+(\mu))$ and $\omega > 0$, see \cite{cdw}. Practically, this weak enforcement of the outflow boundary condition applied to a subspace $W\subset \bar \MM_\mu$ has the following effect: typically boundary layers are found at the outflow boundary which are too narrow to be resolved at affordable cost. If $\omega$ is chosen small, then the enforcement of the outflow boundary condition has little weight so that it is almost ignored which, in turn, 
removes layer artifacts.
 If, however,  $W$ is sufficiently rich so as to resolve  layers,   $\inf_{q\in W}\|f - \bo_\mu   q\|_{\oldHs'}$ becomes so
 small  that the boundary penalty becomes important and the boundary conditions are approximately satisfied, see Figure \ref{fig:condiff-1}.
The rationale is that as long as the layer is not resolved the error with respect to conventional norms (including the SUPG-norm)
is mostly concentrated in the layer region, which   therefore stays, roughly speaking, as large as not realizing the boundary conditions
at the outflow boundary at all. Putting a small weight on this error contribution   actually increases accuracy away from the outflow boundary, see \cite{cdw}. 
Putting it in a slightly different way,   by allowing more freedom in the outflow boundary layer,   the $\MM_\mu$-norm is changed in such a manner that the error in the boundary layer has very small weight. This in turn allows one to better control the error away form the layer. 

To apply the theory of Section \ref{sec:recipe}, we define the test space
\begin{equation}
\label{barH} 
\begin{array}{rcl}
  \bHs & :=  & \HH_\mu\times {\outn(\mu)}' = H^1_0(\Omega) \times H^{1/2}_{00}(\Gamma_+(\mu))' , \\[3mm]
  \|[v,g]\|_{\bHs}^2 & := &  \|v\|_{\oldHs}^2 + \omega \|g\|_{{\outn}(\mu)'}^2,
\end{array}
\end{equation}
and the operator
\begin{equation*}
  \opOut p := [\bo_\mu  p, p|_{\Gamma_+(\mu)}]. 
\end{equation*}
According to the definition of the graph-norm \eqref{Mu2}, this yields the norm
\begin{equation*}
  \|p\|_{\bar{\MM}_\mu}^2 := \|\opOut p\|_{\bHs'}^2 = \| \opOut p\|_{\oldHs'}^2 + \|p\|_{{\outn}(\mu)}^2
\end{equation*}
for the trial space $\bar{\MM}_\mu$. 

Note that in the optimization problem \eqref{eq:opt+} the norm of the boundary penalty is not a dual norm. This allows us to replace the system \eqref{system-disc}, which in or case is a $3 \times 3$ block system, by the simpler system 
\begin{equation}
\label{system-disccondiff}
\begin{array}{lclccl}
  \langle R_{\oldHs} u_V,v\rangle  & + & \langle \bo_\mu  p_W,v \rangle & = & \langle f,v \rangle, & v \in \hf,\\
  \langle \bo_\mu^* u_V, q\rangle  & - & \mu \langle p_W, q \rangle_{\outn(\mu)} &=& 0,& q \in \mf,
\end{array}
\end{equation}
in all practical computations. This system is derived by the same reasoning as in Section \ref{sec:recipe} applied to the first term $\|f - \bo_\mu p_W\|_{\HH_\mu'}$ of the optimization problem \eqref{eq:opt+} only.  In \cite{cdw,W}, it is shown how to transform this system to the equivalent saddle point problem \eqref{system-disc}, so that the theory of the present paper still applies to \eqref{system-disccondiff}.

In summary, we have found a stable variational formulation of the convection-diffusion problem \eqref{eq:condiff} that fits into the general framework of Section \ref{sec:recipe}.
However, note that the spaces $\bar \HH_\mu$ and $\bar \MM_\mu$ may differ even as sets for different $\mu\in \mathcal{P}$, see \eqref{barH},
\eqref{barM}. Specifically, the dependence of $\bar \HH_\mu$ on $\mu$ lies only in the boundary conditions.
However, for a polyhedral domain $\Omega$ one can find a finite cover $\{ \mathcal{P}_l:   l=1,\ldots, P\}$ of 
$\mathcal{P}$ so that the outflow boundary portions $\Gamma_+(\mu)=\Gamma_{+,l}$ stay the same for $\mu\in \mathcal{P}_l$.
Hence, the spaces $\bar \HH_\mu, \bar \MM_\mu$ all agree as sets for $\mu\in \mathcal{P}_l$. Clearly, the solution manifold
$\mathcal{M}$ (see \eqref{sol-M}) is a finite union of solution manifolds $\mathcal{M}(l)$ corresponding to the subsets $\mathcal{P}_l$.
Since each $\mathcal{M}(l)$ is compact so is the finite union $\mathcal{M}$. Note that for $\mu\in\mathcal{P}_l$ the Riesz map $R_{\bar \HH_\mu}$
is independent of $\mu$. Therefore, we can apply Theorem \ref{thm-M} to each component $\mathcal{P}_l$ leading
to the following result.
\begin{cor}
\label{corcondiff}
The scheme \textsc{DG-1} based on \eqref{system-disccondiff} is rate-optimal.
\end{cor}
In this case the online evaluations can be based on Remark \ref{remaltcomp}.
Some first numerical experiments are presented in the following section.
\subsection{Numerical Experiments for Convection-Diffusion Problems}

We consider the convection-diffusion problem
\begin{equation}
  -\epsilon \Delta p + \begin{pmatrix} \cos \mu \\ \sin \mu \end{pmatrix} \cdot \nabla p + p  = 1
      , \,\,  \text{in } \Omega = (0,1)^2,\quad
    p  = 0, \, \text{ on } \partial \Omega.
\label{eq:ex1}
\end{equation}
In all test cases we use the variational fomulation based on \eqref{eq:opt+} regardless of the choice of the truth spaces.}
First, we treat scenario (a), i.e.,   with $\epsilon = 2^{-5}$ which is already convection dominated. 
However, we use a truth space that completely resolves  the layers. 
Specifically, for $\MM_\truth$ and $\HH_\truth$ we choose bilinear finite elements which are continous on a rectangular uniform grid of meshsize $2^{-9}$ and $2^{-10}$, respectively. 
For all computations we used an equidistant sample set $\mathcal{S} \subset \mathcal{P}=[0.2, \pi-0.2]$ of cardinality 500.
Using finite element a-posteriori error estimators from \cite{cdw}, based on \eqref{R*}, the respective truth-accuracy is bounded by $0.00569966$. 
We note that these a-posteriori bounds represent the truth residual and hence the energy error only within some fixed constants. 
This is in contrast to the surrogate bounds for the reduced spaces which are much tighter.
The number of adaptively generated basis functions for the reduced test space together with the corresponding constant of the $\delta$-proximality, as well as the maximal  surrogate are given in Table \ref{tab:condiff}\subref{tab:condiff-1}. 
Figure \ref{fig:condiff-1} shows a reduced basis solution for the angle $\mu = 0.885115$. 
The parameter dependent direction of the first order term is visualized by a plane.
\begin{figure}[h]
    \centering
    \subfloat[$\epsilon = 2^{-5}$]{
        \includegraphics[width=0.3\textwidth]{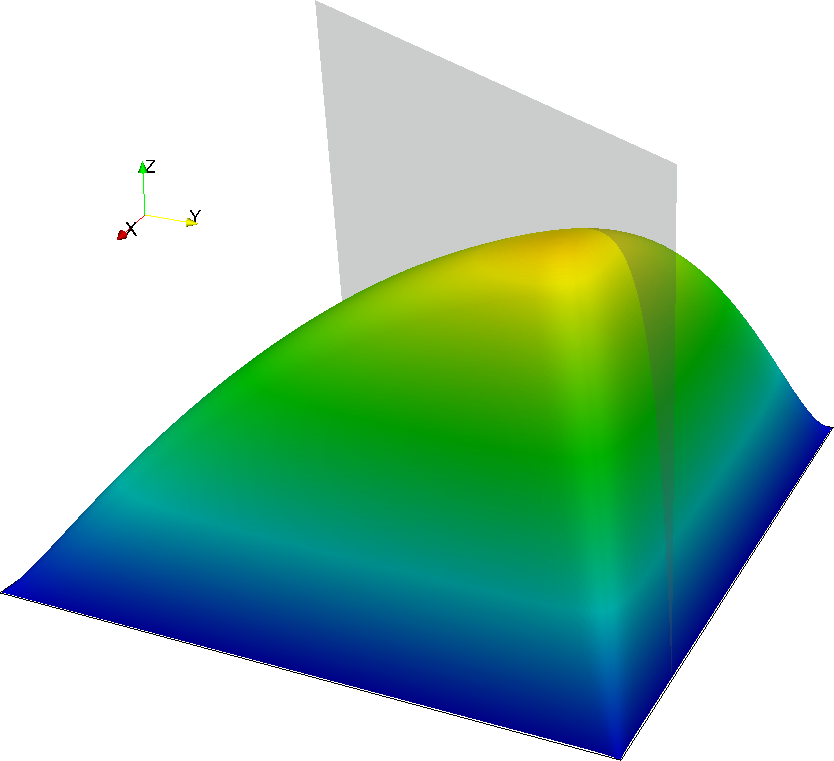}
        \label{fig:condiff-1}}\qquad
    \subfloat[$\epsilon = 2^{-7}$]{
        \includegraphics[width=0.3\textwidth]{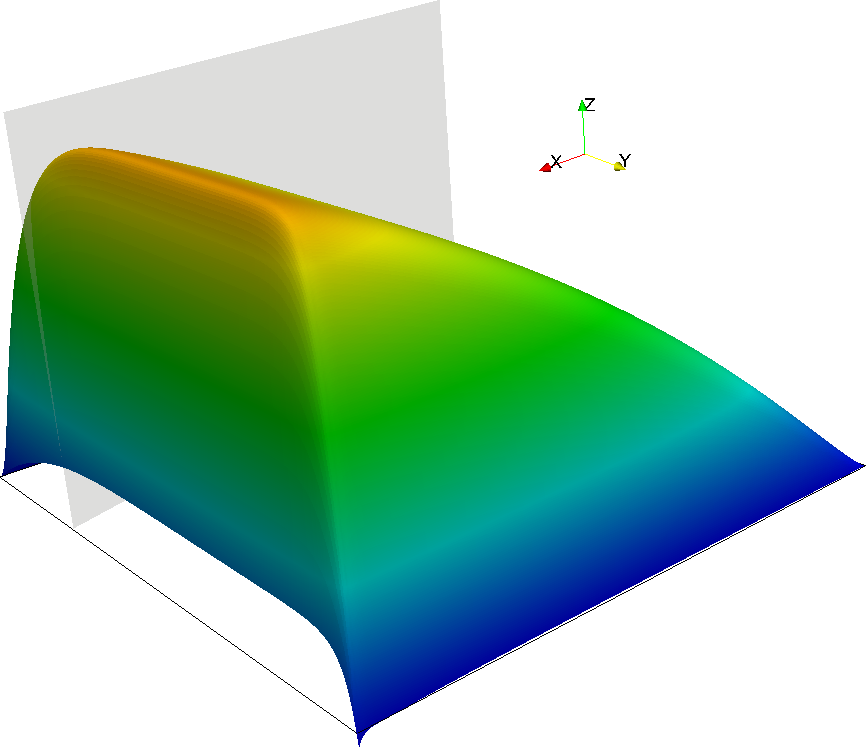}
        \label{fig:condiff-3}}
    \subfloat[$\epsilon = 2^{-26}$]{
        \includegraphics[width=0.3\textwidth]{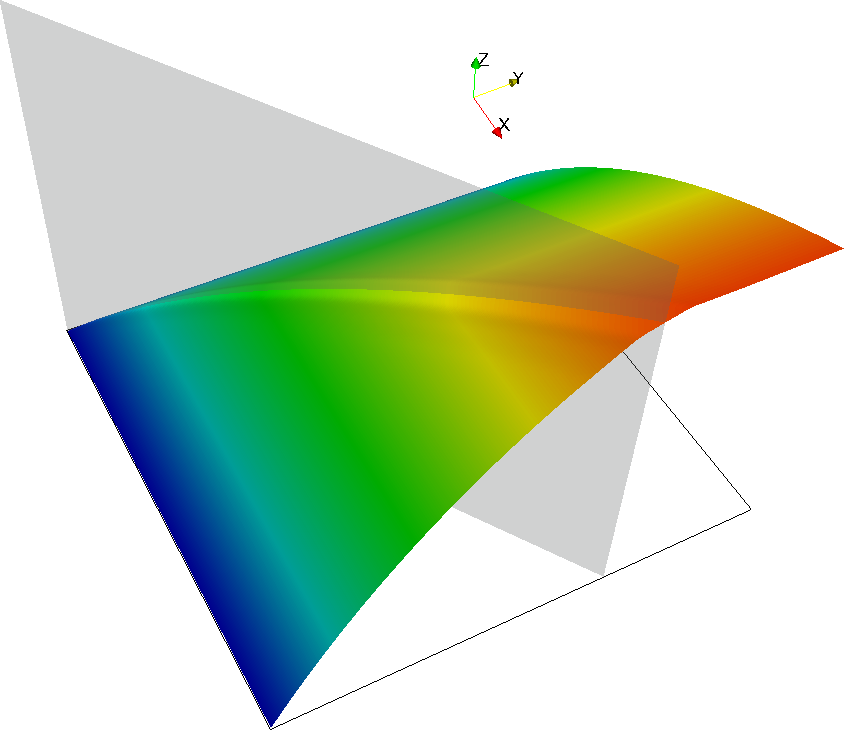}
        \label{fig:condiff-2}}
    \caption{Solutions of the convection-diffusion problem \eqref{eq:ex1}. 
    (a) RB dimension $n=6$, $m(n)=13$, angle $\mu = 0.885115$,
    (b) RB dimension $n=7$, $m(n)=20$, angle $\mu = 0.257484$,
    (c) RB dimension $n=20$, $m(n)=57$, angle $\mu = 0.587137$}
    \label{penalty}
    \begin{tikzpicture}
        \begin{axis}[
                title={$\epsilon=2^{-5}$},
                ymin=0,
                xlabel={reduced basis trial dimension},
                width=0.32\textwidth
            ]
            \addplot table[x=dim_trial,y=max_residual] {pictures/all-cases-14};
        \end{axis}
    \end{tikzpicture}~
    \begin{tikzpicture}
        \begin{axis}[
                title={$\epsilon=2^{-7}$},
                ymin=0,
                xlabel={reduced basis trial dimension},
                width=0.32\textwidth
            ]
            \addplot table[x=dim_trial,y=max_residual] {pictures/all-cases-8};
        \end{axis}
    \end{tikzpicture}~
    \begin{tikzpicture}
        \begin{axis}[
                title={$\epsilon=2^{-26}$},
                ymin=0,
                xlabel={reduced basis trial dimension},
                width=0.32\textwidth
            ]
            \addplot table[x=dim_trial,y=max_residual] {pictures/all-cases-9};
        \end{axis}
    \end{tikzpicture}
    \caption{Surrogates of the reduced basis approximation for the convection-diffusion problem \eqref{eq:ex1}.}
    \label{fig:condiff-errors}
\end{figure}
\begin{table}[htb]
    \centering
    \subfloat[$\epsilon = 2^{-5}$, maximal a-posteriori error 0.00569966]{
        \begin{tabular}{|c|c|c|c|c|}
            \hline
            \multicolumn{2}{|c|}{dimension} & & maximal & surr / \\
            trial & test & $\delta$ & surrogate & a-post \\
            \hline
            \hline
            2 & 3 & 2.51e-01 & 7.04e-02 & 1.24e+01 \\
            3 & 6 & 3.74e-01 & 3.08e-02 & 5.40e+00 \\
            4 & 7 & 3.74e-01 & 7.43e-03 & 1.30e+00 \\
            5 & 10 & 3.51e-01 & 5.81e-03 & 1.02e+00 \\
            6 & 13 & 1.86e-01 & 5.70e-03 & 1.00e+00 \\
            \hline
        \end{tabular}
        \label{tab:condiff-1}
    }
    \subfloat[$\epsilon = 2^{-7}$, maximal a-posteriori error 0.0197011]{
        \begin{tabular}{|c|c|c|c|c|}
            \hline
            \multicolumn{2}{|c|}{dimension} & & maximal & surr / \\
            trial & test & $\delta$ & surrogate & a-post \\
            \hline
            \hline
            2 & 5 & 8.92e-03 & 1.25e-01 & 6.37e+00 \\
            3 & 8 & 1.22e-01 & 9.65e-02 & 4.90e+00 \\
            4 & 11 & 1.13e-02 & 3.21e-02 & 1.63e+00 \\
            5 & 14 & 1.27e-02 & 2.61e-02 & 1.32e+00 \\
            6 & 17 & 5.55e-03 & 2.21e-02 & 1.12e+00 \\
            7 & 20 & 4.82e-03 & 1.97e-02 & 1.00e+00 \\
            \hline
        \end{tabular}
        \label{tab:condiff-3}
    }\\
    \subfloat[$\epsilon = 2^{-26}$, maximal a-posteriori error 0.001055]{
        \begin{tabular}{|c|c|c|c|c||c|c|c|c|c|}
            \hline
           \multicolumn{2}{|c|}{dimension} & & maximal & surr / & \multicolumn{2}{|c|}{dimension} & & maximal & surr / \\ 
           trial & test & $\delta$ & surrogate & a-post &trial & test & $\delta$ & surrogate & a-post \\
            \hline
            \hline
            2  & 5  & 1.35e-03 & 2.11e-01 & 2.00e+02 & 12 & 33 & 3.47e-04 & 1.60e-02 & 1.52e+01 \\
            4  & 9  & 1.09e-02 & 7.58e-02 & 7.19e+01 & 14 & 39 & 1.10e-04 & 8.46e-03 & 8.02e+00 \\
            6  & 15 & 1.61e-03 & 5.02e-02 & 4.76e+01 & 16 & 45 & 9.39e-05 & 7.87e-03 & 7.46e+00 \\
            8  & 21 & 7.99e-04 & 2.39e-02 & 2.26e+01 & 18 & 51 & 6.11e-05 & 7.69e-03 & 7.29e+00 \\
            10 & 27 & 3.55e-04 & 2.10e-02 & 2.00e+01 & 20 & 57 & 5.28e-05 & 6.35e-03 & 6.02e+00 \\
            \hline
        \end{tabular}
        \label{tab:condiff-2}
    }
    \caption{Numerical results for the convection-diffusion problem (\ref{eq:ex1}).}
    \label{tab:condiff}
\end{table}

The intermediate case $\epsilon = 2^{-7}$ shown in Figure \ref{fig:condiff-3} demonstrates how the formulation handles a not
fully resolved boundary layer which is not far-off being resolved either.

The other example, scenario (b), refers to the same problem \eqref{eq:ex1} again, however, with a very small viscosity $\epsilon = 2^{-26}$.
Hence,  this case is even more strongly  convection dominated and poses difficulties for resolving the boundary layers even for the truth space itself. 
Since the boundary layers are not resolved \eqref{eq:opt+} does not strictly enforce
strong boundary conditions at the outflow boundary even in the truth space. Accordingly, the approximate solutions 
from the reduced space do not satisfy the boundary conditions in a strict sense either.
In fact, we choose the same truth spaces as in the preceding experiment. 
Thus we have to employ the norms \eqref{barH} based on the variational formulation \eqref{system-disccondiff}. 
The numerical results are summarized in Table \ref{tab:condiff}\subref{tab:condiff-2} 
and a corresponding reduced basis solution is displayed in Figure \ref{fig:condiff-2}.
 
Figure \ref{fig:condiff-errors} displays a surrogate plot for the values $2^{-5}$, $2^{-7}$ and $2^{-26}$ of $\epsilon$. 
One observes that the error of the reduced basis approximation decays rapidly already for small reduced bases.
In fact, since    
$ \inf_{p \in \MM_\truth} \|f - B_\mu p\|_{\HH'_\mu} \le \inf_{p \in \MM_n} \|f - B_\mu p\|_{\HH'_\mu}$,
the error of the truth approximation is always a lower bound for the error of the reduced basis approximation. This contrasts standard reduced basis methods where one, in our terminology, chooses $\HH_\mathcal{N} = \MM_\mathcal{N}$. Instead, we assume a larger space $\HH_\mathcal{N} \in \mathcal{V}(\MM_\mathcal{N}, \delta_\mathcal{N})$ which, according to Remark \ref{WMX}, implies that the surrogate \eqref{tightsurr2} is equivalent to the true error with respect to the infinite dimensional solution,
regardless of whether the truth space resolves all solution features like boundary layers or not.
Comparing the surrogate plots with the Tables \ref{tab:condiff}, 
one sees that the reduced basis errors stagnate roughly at the error level of the truth solution.
Due to the very small $\delta$-proximality thresholds, the surrogates reflect the true reduced errors very accurately, see \eqref{tol},
Proposition \ref{rem:tight}.

Note that in all cases the inner stabilization loop produces at most $m_B=3$ addidional test basis functions for the test space, see
Proposition \ref{propmB}.

\subsection{Transport problems - the Worst Scenario}\label{sect:wcs}
We address now the transport equation \eqref{eq:transport} in Section \ref{ssect:transport}.
Aside from its essential appearance 
in more general kinetic models and 
Boltzmann type equations, it can be viewed as a ``limit'' of convection-diffusion problems. The particular interest lies
in the complete lack of viscosity as a ``classical'' stabilizing ingredient, see e.g. \cite{PU}. 
Moreover, as we shall see, the conditions \eqref{allMequiv} and \eqref{allHequiv} do {\em not}
hold simultaneously, not even for suitable subsets of $\mathcal{P}$. Moreover, the parameter dependence will be seen to
be significantly less smooth.
 
We have already proposed a variational formulation \eqref{btransport} along with the spaces $\HH_\mu, \MM_\mu$
in \eqref{transspaces} endowed with the norms \eqref{transnorms}. With these definitions, the operator $\bo_\mu   : \MM_\mu \to \HH_\mu$  is an isomorphism with condition number $1$, i.e. $\|\cdot\|_{\MM_\mu} = \|\cdot\|_{\hat\MM_\mu}= \|\cdot\|_{\HH_\mu'}$, see \cite{DHSW}.
Notice that in this case the Riesz map $R_{\HH_\mu}$ are given by
\begin{equation*}
R_{\HH_\mu}= B_\mu B_\mu^*,\quad \mbox{i.e.,} \,\,\,  (v,w)_{\HH_\mu}= \langle B_\mu^*v,B_\mu^*w\rangle,\quad
\|\cdot\|_{\hat\MM_\mu}= \|\cdot\|_{L_2(\Omega)},
\end{equation*}
so that \eqref{affine} and \eqref{allMequiv} are  valid. Finally,
  the Riesz maps $R_{\HH_\mu}$ and $R_{\MM_\mu}$ depend affinely on the parameter so that the double greedy scheme can be applied. 
  
However, since the $\HH_\mu$-norm is {\em not} independent of $\mu$, we cannot evaluate the surrogate given by 
\eqref{tightsurr2} 
in the usual way.
As a remedy, we use the surrogate $R_n'(\mu)$ from \eqref{tightsurr3}, i.e.,  we approximate this inverse Riesz map by projecting on the reduced basis space 
$\HH_n$ instead of the truth space
$\HH_\mathcal{N}$. To ensure that this surrogate is also tight we take up the criterion in Remark \ref{rem:MnX}.
Specifically, we wish to apply Lemma \ref{lem:tight} and try to construct suitable pairs $\bar\MM, \bar\HH$ as follows.

We run the double greedy scheme (possibly) several times which yields the sequences of 
reduced spaces $\MM^i_1, \MM^i_2, \dots$ and $\HH^i_1, \HH^i_2, \dots$, $i=0,1,2,..$ in the $i$th run of the full double-greedy algorithm. Now, say we stop the first run at index $N$ and define $\bar{\MM} := \MM^0_N$. For the second run, we use the same inital spaces as for the first run, however, the calls of \textsc{Update-Inf-Sup$(\HH^1_n, \MM^1_n)$} are replaced by \textsc{Update-Inf-Sup$(\HH^1_n, \bar{\MM} + \MM^1_n)$}, so that $\delta$-proximality is guaranteed for the larger space $\bar{\MM} + \MM^1_n$. Then, with the $n$-dependent choice $\bar{\HH} = \HH^1_n$, the estimate \eqref{eq:estimator-stab-2} implies that for the second run the surrogates are tight as long as the condition \eqref{eq:rel-error} is satisfied. 

Of course, neither can  this latter condition   be rigorously checked since we cannot rely on the surrogates, nor
have we specified the terminating index $N=N_0$. We briefly sketch now   several options
of iteratively tightening the surrogates $R_n'(\mu)$. One could stop the first run $i=0$ at the  smallest $N_0$ for which
$R'_{N_0}(\mu)/\tau_\mathcal{N} \leq \alpha$ for some $\alpha \ll 1$, where $\tau_\mathcal{N}$ is the truth error tolerance.  The second run $i=1$ with $\bar\MM=\bar\MM^1=\MM_{N_0}^0$ will stop at step $N_1$. In general,
the $i$th run with $\bar\MM^i =\bar\MM^{i-1} + \MM_{N_{i-1}}^{i-1}$ stops at $N_i$.
 One expects that $N_{i+1} \geq N_i$ since the surrogates, being lower bounds for the true residuals, become tighter as long as $\bar\MM^i$ grows.
A practical stopping criterion would be, for instance, that $N_{i+1}\leq N_i$, or $R'_{N_i}(\mu)/\|f- B_\mu p_{N_i}(\mu)\|_{\HH_\mathcal{N}'}
\sim R'_{N_{i+1}}(\mu)/\|f- B_\mu p_{N_{i+1}}(\mu)\|_{\HH_\mathcal{N}'}$.

An alternative strategy is to apply the double greedy scheme to the defect problem
\begin{equation*}
B_\mu \bar p(\mu) = f- B_\mu p_{N_0}(\mu),\quad \mu\in \mathcal{P},
\end{equation*}
and form $\bar\MM$ as the sum of $\MM_{N_0}^0$ and the largest reduced space for the defect problem. Since the relative accuracy to be achieved for the 
defect problem only needs to meet the constant $\xi$ in \eqref{eq:rel-error} one expects that a few steps suffice. Since $\bar\MM$ now contains
``complementary'' information $\MM_{N_0}^0$ is enlarged more effectively than in the first method.

The upshot of these comments is  that investing additional computational offline effort is guaranteed to tighten the surrogates
and thereby improves the choice of the reduced spaces. This is  in contrast to greedy strategies based on surrogates that are not based
on well-conditioned variational formulations
and  therefore most likely fail to detect   the most effective snapshots. 
These issues will be addressed in forthcoming work.

 Since the basis function $\phi_j\in \MM_n$ can now be
orthonormalized in $L_2(\Omega)$ and $\|\cdot\|_{\hat \MM_\mu}=\|\cdot\|_{L_2(\Omega)}$,  
 Theorem \ref{thm-M} applies and yields the following result.

\begin{cor}
\label{cor-transport}
If $R_n'(\mu)$ from \eqref{tightsurr3} is based on iterative tightening with $\bar\MM_{i(n)}$ satisfying \eqref{eq:rel-error}
for sufficiently small $\xi$, 
then the scheme \textsc{DG-1}
using \textsc{Update-inf-sup} is rate-optimal for $\mathcal{M}_{\MM}$.
\end{cor}

We could also reverse the roles of the spaces $\HH_\mu, \MM_u$, choosing $L_2(\Omega)$ as the test space,
see \cite{EG04, W}.
In this case the trial spaces would essentially depend on the parameter $\mu$ so that the understanding of the
solution set $\mathcal{M}$ is less clear. On the other hand, this choice would correspond to the limit   of the formulation
\eqref{system-disccondiff} for vanishing viscosity.

Since \eqref{allHequiv} does not hold we cannot apply Proposition \ref{propmB} to predict   a strict a priori bound on the number of stabilization steps in \textsc{Update-$\delta$} 
or \textsc{Update-Inf-Sup}. Adhering to the notation in Section \ref{sec:greedy},
we have here $K_\mu = B_\mu^*$, see \eqref{transnorms}. 
Since in the present case $R_\MM$ is the identity, as pointed out there, the enrichments of the test spaces $\HH_n$ are 
are
linear combinations of elements of the form $B^{-*}_{\mu'}B_{\mu''}^{-1}f$ where $\mu', \mu''$  
are different most of the time, due to the greedy selection. As a consequence, when $f\in L_2(\Omega)$, this means
that indeed $B^{-*}_{\mu'}B_{\mu''}^{-1}f\in H^1(\Omega)=\HH$. Of course, the $H^1$-norm may deteriorate
when $\mu', \mu''$ get closer, which however may be offset to some extent by the expectation that these snapshots are most relevant for the stabilization of solutions with nearby parameters.
It is also clear that higher regularity of $f$ would indeed ensure sufficient regularity of the $q\in \MM_n^1$ (or $q\in \mathcal{B}$), independently of $\mu', \mu''$ and hence
allows one to  control 
the constants $C(n,\mathcal{N})$ (or $C(\mathcal{B},\truth)$). This effect is reflected to some extent by the experiments below. 
\subsection{ Numerical Experiments for transport problems}

We consider the analog of the convection-diffusion problem \eqref{eq:ex1} with zero diffusion $\epsilon = 0$ and corresponding boundary conditions, i.e.
\begin{equation}
  \begin{pmatrix} \cos \mu \\ \sin \mu \end{pmatrix} \cdot \nabla p + p  = 1
      , \,\, \text{in } \Omega = (0,1)^2, \quad
    p  = 0, \,  \text{ on } \Gamma_-.
\label{eq:ex2}
\end{equation}
We employ a truth trial space with mesh size $2^{-8}$, using  
discontinuous piecewise bilinear finite elements with proper boundary conditions.
To ensure stable truth discretizations, the test truth space is comprised of globally continuous piecewise bilinear finite elements,
therefore being contained in  $\HH= \bigcap_{\mu\in\mathcal{P}}\HH_\mu$,
on a finer mesh with mesh size $2^{-9}$ to ensure $\delta$-proximality.
Recall that the spaces $\HH_\mu$ now
differ even as sets.
The results are shown in Table \ref{tab:transport-1} and a reduced basis solution for the angle $\mu = 0.244579$ is given in Figure \ref{fig:transport}. 
Specifically, in addition to the dimensions of the trial and test spaces in columns (1)  ``trial'', (2) ``test'', it records
the values of the surrogates in column (4)  ``surr'', the error between the reduced basis solution and the best $L_2$-approximation
of the exact solution in the truth space in column (6)  ``rb L2'', the error between the 
reduced basis solution and the truth solution in column (5)  ``rb truth'', and finally in column (7) ``surr/err'' 
the ratio between the computed surrogate and the error in ``rb L2''. All values reflect the worst case over the parameter range. 

As pointed out above, unlike the convection-diffusion problem, the surrogate \eqref{tightsurr3} for the transport problem is not necessarily well-conditioned
 from the start. 
Therefore, Table \ref{tab:transport-1} contains one column which shows the ratio of the surrogate compared to the true error of the reduced basis approximation, 
maximized over a sample of the angles with the largest values of the surrogate. 
 Although this is at this point not founded rigorously, we see that this ratio stays uniformly bounded with respect to
  the size of the reduced basis.  Hence it already does reflect the accuracy of the reduced model.
  However, the ratio ist not close to one yet, as it would be for 
  a well-conditioned  surrogate given by the truth-exact evaluation of the residual corresponding to a well-conditioned variational formulation.
  To further improve this ratio by approximating the   residual more accurately, we resort to iterative tightening as described above.

The results for a single iteration are recorded in Table \ref{tab:transport-2}. 
It is seen that already after a single run the ratio of the surrogate and the true error between the reduced basis approximation and true solution has become much closer to one. 

In agreement with the discussion in Section \ref{sec:greedy} the experiments show that a slightly larger number of test basis functions than for the convection-diffusion problem is needed here.
In particular, unlike in the elliptic case non-smooth data (right hand side, boundary shape, and boundary conditions) affect 
the smoothness of the dependence of the solutions on the parameter. In our examples at most a low order polynomial decay
of the $n$-widths can be expected.  According to Summary \ref{greed-summ} in Section \ref{sec:greedy}, since the right hand side is actually smooth,   we expect that  $\sigma_{n,j}$, defined in \eqref{sigmanj}, that controls the termination of the inner stabilization loop  drops below the desired $\delta <1$ after an acceptable bounded number of steps independent of the dimension of the truth space.
In fact, one observes that 
  the growth of the test basis stays surprisingly moderate.

\begin{table}
  \centering
  \begin{tabular}{|c|c|c|c|c|c|c|}
      \hline
      \multicolumn{2}{|c|}{dimension} & & maximal & \multicolumn{2}{|c|}{maximal error between} & surr / \\
      trial & test & $\delta$ & surr &  rb truth &  rb L2 & err \\
      \hline
      \hline
      4 & 11 & 3.95e-01 & 8.44e-03 & 2.45e-02 & 2.45e-02 & 3.45e-01 \\
      6 & 17 & 4.49e-01 & 7.06e-03 & 1.40e-02 & 1.40e-02 & 5.04e-01 \\
      8 & 25 & 4.87e-01 & 4.16e-03 & 9.05e-03 & 9.05e-03 & 4.60e-01 \\
      10 & 33 & 4.32e-01 & 3.37e-03 & 5.74e-03 & 5.74e-03 & 5.87e-01 \\
      12 & 40 & 4.83e-01 & 2.65e-03 & 4.65e-03 & 4.65e-03 & 5.71e-01 \\
      14 & 48 & 4.23e-01 & 1.64e-03 & 3.39e-03 & 3.39e-03 & 4.83e-01 \\
      16 & 57 & 4.32e-01 & 1.50e-03 & 2.56e-03 & 2.56e-03 & 5.84e-01 \\
      18 & 65 & 4.66e-01 & 1.17e-03 & 2.33e-03 & 2.33e-03 & 5.03e-01 \\
      20 & 74 & 4.16e-01 & 1.21e-03 & 2.10e-03 & 2.10e-03 & 5.77e-01 \\
      22 & 83 & 3.83e-01 & 1.02e-03 & 1.93e-03 & 1.93e-03 & 5.29e-01 \\
      24 & 91 & 4.05e-01 & 7.27e-04 & 1.58e-03 & 1.58e-03 & 4.61e-01 \\
      \hline
  \end{tabular}
  \caption{Numerical results for the transport problem \ref{eq:ex2}, maximal error truth L2 0.000109832.}
  \label{tab:transport-1}
\end{table}

\begin{figure}[htb]
    \hfill
    \begin{minipage}[c]{0.49\textwidth}
  \begin{center}
      \includegraphics[width=0.75\textwidth]{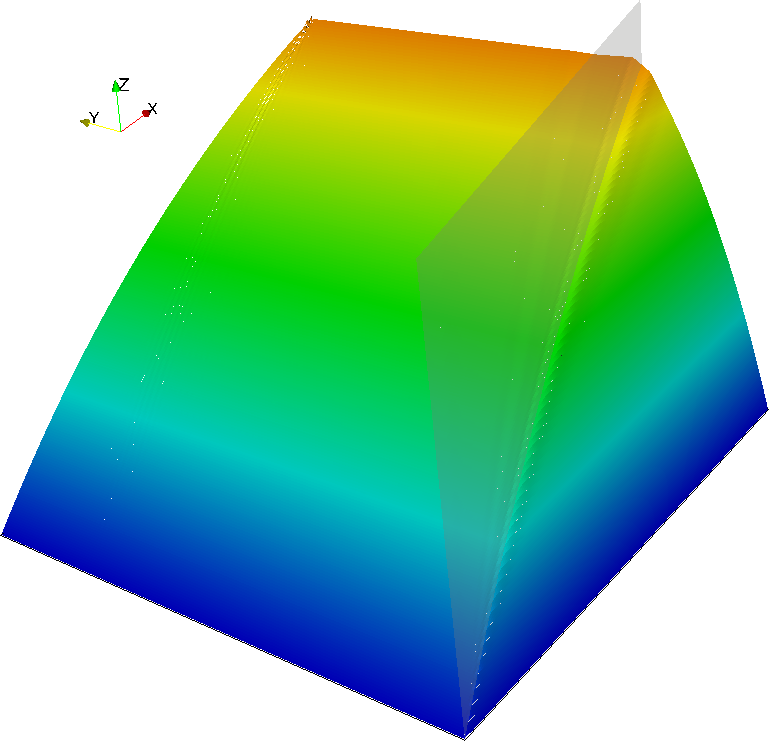}
  \end{center}
  \caption{Solution of the transport problem \eqref{eq:ex2}, with reduced basis of dimension $n=24$, $m(n)=91$ and angle $\mu = 0.244579$.}
  \label{fig:transport}
  \end{minipage}
  \hfill
    \begin{minipage}[c]{0.49\textwidth}
  \begin{center}
      \includegraphics[width=0.75\textwidth]{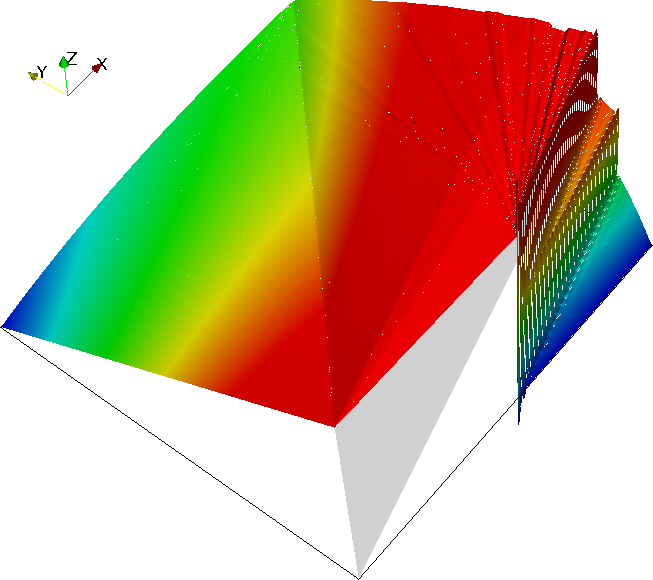}
  \end{center}
  \caption{Solution of the transport problem \eqref{eq:ex3}, with reduced basis of dimension $n=24$, $m(n)=96$ and angle $\mu = 0.256311$.}
  \label{fig:transport-3}
  \end{minipage}
  \hfill~
\end{figure}

\begin{table}
  \centering
  \begin{tabular}{|c|c|c|c|c|c|c|}
      \hline
      \multicolumn{2}{|c|}{dimension} & & maximal & \multicolumn{2}{|c|}{maximal error between} & surr / \\
      trial & test & $\delta$ & surr & rb truth & rb L2 & err \\
      \hline
      \multicolumn{7}{|c|}{first reduced basis creation} \\ 
      \hline
      20 & 81 & 3.73e-01 & 2.71e-02 & 5.46e-02 &  5.62e-02 & 4.82e-01 \\
      \hline
      \multicolumn{7}{|c|}{second reduced basis creation} \\ 
      \hline
      10 & 87 & 3.51e-01 & 6.45e-02 & 7.40e-02 &  7.53e-02 & 8.57e-01 \\
      \hline
  \end{tabular}
  \caption{Numerical results for the transport problem \ref{eq:ex2} after a single cycle of iterative tightening.
  Maximal error truth L2 0.0154814}
  \label{tab:transport-2}
\end{table}

Finally, Table \ref{tab:transport-3} and Figure \ref{fig:transport-3} show the results for the problem 
\begin{equation}
\begin{aligned}
  \begin{pmatrix} \cos \mu \\ \sin \mu \end{pmatrix} \cdot \nabla p + p  & = 
      \left\lbrace
      \begin{array}{cl}
          0.5 & x < y \\
          1 & x \geq y
      \end{array}
      \right.
     , & & \text{ in } \Omega = (0,1)^2, \quad
     p &=  
    \left\lbrace
    \begin{array}{cl}
        1-y &  x \leq 0.5 \\
        0 &  x > 0.5
    \end{array}
    \right.
    , & & \text{ on } \Gamma_-.
\end{aligned}
\label{eq:ex3}
\end{equation}
Now the right hand side as well as the boundary conditions exhibit  {jump discontinuities where the latter is transported trough the domain}. 
This causes a further significant reduction of the smoothness of the dependence on the solutions on the parameter.
Problem \eqref{eq:ex3} therefore represents an extreme example involving interacting jump discontinuities caused
by the right hand side and by the boundary conditions.
The small  ripples observed in the solution plot  Figure \ref{fig:transport-3} originate from the superposition of the jumps of the various snapshots involved in the solution.
Since they do neither grow nor expand   one can conclude that the scheme is in fact stable. 

As indicated before,  varying the transport direction for such data
shows that the dependence of the solution on the parameter is even less smooth than in the previous example
so that the Kolmogorov widths of the solution manifold are expected to decay more slowly. Hence the greedy errors cannot decay too rapidly either. Again, by Summary \ref{greed-summ}, the quantities $\sigma_{n,j}$ in \eqref{sigmanj}, estimating the
number of stabilization steps for $\MM_n$, are expected to
decay even more slowly than in the case of zero boundary conditions. Tabel \ref{tab:transport-3} confirms this in that slightly more test 
basis functions are generated than in example \ref{eq:ex2}.
Nevertheless, one observes that in the initial phase already  a few reduced basis functions decrease the error
very effectively so that a reduced space with trial dimension as low as ten realizes an accuracy that would require a conventional  finite element space
of much larger dimension. Overall, the performance, at least in the given range of truth accuracy, is  only slightly weaker than 
for the milder case of
zero inflow boundary conditions. The precise implications on the approximation of functionals of the solution and possible 
strategies for alternative ways of enriching the trial dictionary will be explored in forthcoming work.

\begin{table}
  \centering
  \begin{tabular}{|c|c|c|c|c|c|c|}
      \hline
      \multicolumn{2}{|c|}{dimension} & & maximal & \multicolumn{2}{|c|}{maximal error between} & surr / \\
      trial & test & $\delta$ & surr &  rb truth &  rb L2 & err \\
      \hline
      \hline
      4 & 14 & 4.97e-01 & 5.91e-02 & 1.29e-01 & 1.30e-01 & 4.54e-01 \\
      6 & 23 & 4.92e-01 & 4.29e-02 & 1.00e-01 & 1.02e-01 & 4.22e-01 \\
      8 & 31 & 4.29e-01 & 4.34e-02 & 7.78e-02 & 7.95e-02 & 5.46e-01 \\
      10 & 40 & 4.15e-01 & 3.84e-02 & 7.78e-02 & 7.95e-02 & 4.83e-01 \\
      12 & 49 & 3.71e-01 & 3.48e-02 & 7.40e-02 & 7.53e-02 & 4.63e-01 \\
      14 & 57 & 3.76e-01 & 3.12e-02 & 6.20e-02 & 6.41e-02 & 4.87e-01 \\
      16 & 64 & 3.74e-01 & 2.99e-02 & 6.20e-02 & 6.41e-02 & 4.67e-01 \\
      18 & 73 & 4.63e-01 & 2.86e-02 & 6.20e-02 & 6.41e-02 & 4.47e-01 \\
      20 & 81 & 3.73e-01 & 2.71e-02 & 5.46e-02 & 5.62e-02 & 4.82e-01 \\
      22 & 87 & 4.09e-01 & 2.42e-02 & 5.46e-02 & 5.62e-02 & 4.32e-01 \\
      24 & 96 & 3.91e-01 & 2.51e-02 & 4.51e-02 & 4.79e-02 & 5.25e-01 \\
      \hline
  \end{tabular}
  \caption{Numerical results for the transport problem \ref{eq:ex3}, maximal error truth L2 0.0154814.}
  \label{tab:transport-3}
\end{table}

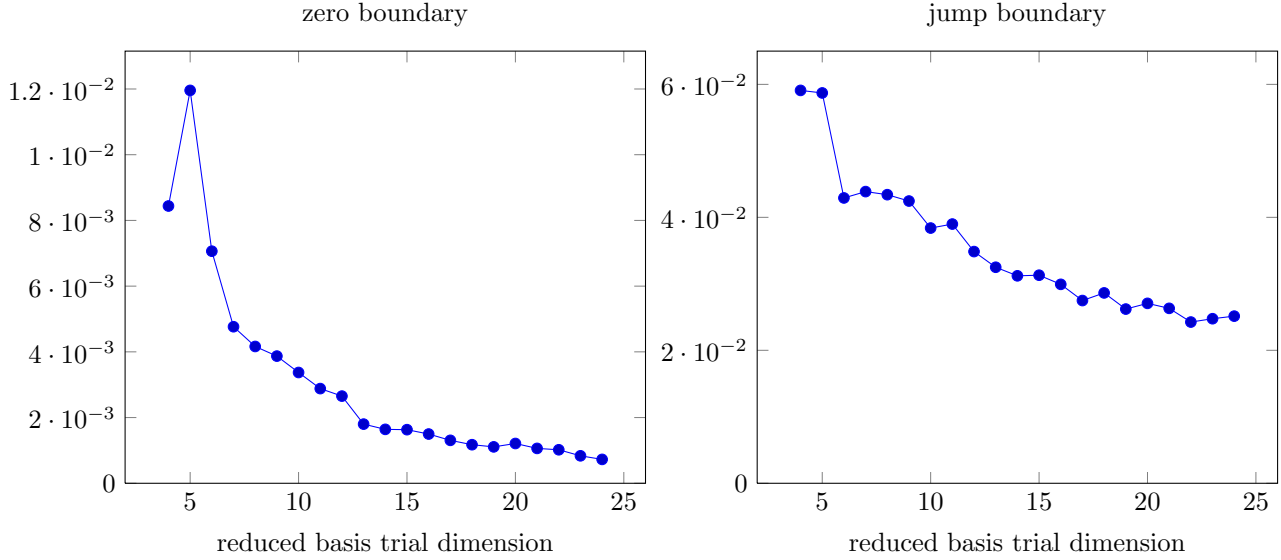
\begin{figure}[htb]
    \begin{center}
        \begin{tikzpicture}
            \begin{axis}[
                title={zero boundary},
                ymin=0,
                xlabel={reduced basis trial dimension},
                width=0.5\textwidth
                ]
                \addplot table[x=dim_trial,y=m_res] {pictures/all-cases-11};
            \end{axis}
        \end{tikzpicture}~
        \begin{tikzpicture}
            \begin{axis}[
                title={jump boundary},
                ymin=0,
                xlabel={reduced basis trial dimension},
                width=0.5\textwidth
                ]
                \addplot table[x=dim_trial,y=m_res] {pictures/all-cases-12};
            \end{axis}
        \end{tikzpicture}
    \end{center}
    \caption{Surrogates of the reduced basis approximation for the transport problem \eqref{eq:ex2} and \eqref{eq:ex3}.}
    \label{fig:transport-errors}
\end{figure}

\section{General  Saddle Point Problems}\label{sectsaddle}
The crucial role of saddle point problems for the generation of well-conditioned variational formulations
is apparent from the preceding discussion. On the other hand, the concepts developed in this context
have an immediate bearing on more general saddle point problems of ``classical type''.
By this we mean (parameter dependent) variational formulations e.g.  of the Stokes system or those
arising in mixed formulations and constrained optimization problems. To see this, it is useful to point
out the main distinctions between the two settings when considering the following general formulation
for parameter dependent   bilinear forms $a_\mu(\cdot,\cdot): \HH_\mu\times \HH_\mu\to \R$, $b_\mu(\cdot,\cdot):\MM_\mu\times \HH_\mu\to
\R$
such that
for $\mu\in \mathcal{P}$ 
\begin{equation}
\label{varprob}
\begin{array}{lcc}
a_\mu(u(\mu),v) + b(p(\mu), v) & = & \langle f,v\rangle,\quad v\in \HH_\mu,\\
b_\mu(q, u(\mu)) &= & \langle g,q\rangle, \quad q\in \MM_\mu.
\end{array}
\end{equation}
For classical problems the following conditions
\begin{equation}
\label{bcont}
|a_\mu(v,w)| \leq C_a(\mu) \|v\|_{\HH_\mu} \|w\|_{\HH_\mu},\quad  |b_\mu(q, v)| \leq  C_b(\mu) \|v\|_{\HH_\mu} \|q\|_{\MM_\mu}, \quad v, w \in \HH_\mu\quad
q\in \MM_\mu,
\end{equation}
as well as
 \begin{equation}
\label{ellip}
\inf_{q\in\MM_\mu}\sup_{v\in\HH_\mu}\frac{b_\mu(q,v)}{\|q\|_{\MM_\mu}\|v\|_{\HH_\mu}}\geq \beta(\mu),\quad\quad
 a_\mu(v,v) \geq  c_a(\mu)  \|v\|^2_{\HH_\mu},\,\,v\in V(\mu),
\end{equation}
where
\begin{equation*}
V(\mu):= \{ v\in \HH_\mu: b_\mu(q, v)=0,\,\forall \, q\in M\}= {\rm ker}\,\bo_\mu^*,
\end{equation*}
 are usually met.

In comparison, the specific structure of the ``stabilizing'' saddle point problem \eqref{system-full} is the following.
\begin{rem}
\label{remellip}
$B_\mu$ is an isomorphism and for $a_\mu(v,w):= \langle R_{\HH_\mu}v,w\rangle$ one has $V(\mu)=\{0\}$ and condition \eqref{ellip} holds with $c_a(\mu)=C_a(\mu)=1$
even on $\HH_\mu$. Moreover, when using the norm $\|\cdot\|_{\hat\MM_\mu}$ on $\MM_\mu$, \eqref{bcont}
holds with $C_b(\mu)=\beta(\mu)=1$. However, on the downside, one may encounter failure of either \eqref{allHequiv} or \eqref{allMequiv}.
\end{rem}

For the classical problems considered in \cite{GV,GV2,Rozza} one can state the following.
\begin{rem}
\label{remclassical}
The conditions \eqref{allHequiv}, \eqref{allMequiv} are both satisfied so that a single reference norm
$\|\cdot\|_{\HH \times \MM}$ can be used. Hence,  that renormation \eqref{Mu2} is not necessary for achieving
tightness of residual based surrogates which now involve both component spaces $\HH'\times \MM'$
which can be evaluated by the standard offline-online decomposition, see e.g. \cite{GV, GV2}.
\end{rem}

It is well known  (see e.g. \cite{BF}) that, given \eqref{bcont} and \eqref{ellip}, the validity of the mapping property {\bf MP} and
the best approximation property {\bf BAP} hinges again on the inf-sup condition
\begin{equation}
\label{infsupVW}
 \inf_{q\in W}\sup_{v\in V}\frac{b_\mu(q, v)}{\|v\|_{\HH_\mu}\|q\|_{\MM_\mu}} \geq \beta_{V,W}(\mu).
\end{equation}
For classical problems it is known for $V=\HH, W=\MM$,
has to be ensured for the truth spaces $V=\HH_\mathcal{N}$, $W=\MM_\mathcal{N}$ through suitably chosen finite element spaces, say, and
again need to be ensured by stabilizing strategies for the reduced spaces 
$V=\HH_n$, $W=\MM_n$. 

In view of Remark \ref{remclassical}, both schemes \textsc{Update-$\delta$} and \textsc{Update-inf-sup}
can be applied. Since the spaces $\HH_n$ no longer just serve as stabilizers but need to contribute to the target approximation
accuracy of the full solution manifold
 \begin{equation}
\label{solmanifold}
\mathcal{M} := \{[u(\mu),p(\mu)]:\,\,\mbox{solves}\,\, \eqref{varprob}, \, \mu\in\mathcal{P}\} =: \mathcal{M}_{\HH}\times \mathcal{M}_{\MM}, 
\end{equation}
the only changes that need to be incorporated in a slightly modified version {\bf DG-2} 
of {\bf DG-1} are:
\begin{itemize}
\item
In {\bf Algorithm} \eqref{alg:update-approx} replace step 5 by:\\
Set
\begin{equation*}
{\rm span}\,\{\MM_n,\hat p\} \rightarrow \MM_n,\qquad {\rm span}\,\{\HH_n,\hat u\} \rightarrow \HH_n,
\end{equation*}
i.e., both component spaces are updated in the outer greedy step.
\item
In {\bf Algorithm} \eqref{alg:double-greedy} step 5 is replaced by \\
$\HH_n,  \MM_n \leftarrow \text{\textsc{Update-Approximation}}(\HH_n, \MM_n)$
\item Replace the surrogate by
\[
R^*(\mu, V\times W) := \|f - A_\mu u_{W,V}(\mu) - \bo_\mu p_{W,V}(\mu)\|_{\HH_\mathcal{N}'} + \|g - \bo^* u_{W,V}(\mu) \|_{\MM_\mathcal{N}'},
\]
see \cite{GV, GV2}.
\end{itemize}

Clearly, under the given assumptions \eqref{allHequiv}, \eqref{allMequiv}, $\mathcal{M}$ is compact. Denoting again by $p_n(\mu), u_n(\mu)$ the solution
components produced by the scheme {\bf DG-2} and comparing the greedy errors  \eqref{greedyerror-2}
\begin{equation*}
\sigma_n(\mathcal{M}):= \sup_{\mu\in \mathcal{P}}\big\{\|p(\mu)-p_n(\mu)\|_{\MM}
+ \|u(\mu)-u_n(\mu)\|_{\HH}\big\},
\end{equation*}
with the $n$-widths $d_n(\mathcal{M})_{\MM\times\HH}$  
and keeping Proposition \ref{propmB} 
 in mind,  we extend the results in \cite{GV,GV2,Rozza} as
follows.

\begin{cor}
The scheme {\bf DG-2} applied to \eqref{varprob} is under the above assumptions rate-optimal.
\end{cor}

\section{Concluding Remarks}
The generation of well-conditioned variational formulations for non-coercive or indefinite problems
has been proposed as the central ingredient of a general strategy for constructing tight surrogates for RBMs
also for such problem classes. 
In contrast to previous work,  well-conditioned  tight surrogates are obtained in a feasible way 
in all settings warranting a near-optimal performance of the corresponding RBM, which does not seem to be achievable 
with the aid of previously known concepts.   We emphasize that these concepts apply as well to space-time discretizations of unsteady problems
(see \cite{DHSW})
offering interesting perspectives with regard to robustly capturing long-term dynamics.
The presented application to two simple model problems is to be viewed
as a first proof of concept.   The two examples are to bring out some  essential obstructions and raise
issues that have so far not been addressed in this context.  In particular, they hint at the principal limitations of RBMs
in their standard formulations, especially regarding the smoothness of the parameter dependence.

\end{document}